\documentclass[onefignum,onetabnum]{siamart250211}

\usepackage{graphicx,epsfig,xcolor,lineno}
\usepackage{amsmath}
\usepackage{stmaryrd}
\usepackage{amsfonts}
\usepackage{xcolor}
\usepackage[left=3cm,right=3cm,top=4cm,bottom=4cm]{geometry}

\ifpdf
\hypersetup{
	pdftitle={PWB},
	pdfauthor={J. Coyle and N. Nigam}
}
\fi

\usepackage{lipsum}
\usepackage{amsfonts}
\usepackage{graphicx}
\usepackage{epstopdf}
\usepackage{algorithmic}
\ifpdf
\DeclareGraphicsExtensions{.eps,.pdf,.png,.jpg}
\else
\DeclareGraphicsExtensions{.eps}
\fi

\newsiamremark{remark}{Remark}
\newsiamremark{hypothesis}{Hypothesis}
\crefname{hypothesis}{Hypothesis}{Hypotheses}
\newsiamthm{claim}{Claim}
\newsiamremark{fact}{Fact}
\crefname{fact}{Fact}{Facts}

\headers{Preconditioning PWB}{J. Coyle and N. Nigam}

\title{The whys and hows of conditioning of DG plane wave Trefftz methods:  a single element. \thanks{Submitted to the editors DATE.
		\funding{NN's research is supported by the Discovery Grants program of the Natural Sciences and Engineering Research Council of Canada.}}}

\author{%
	{	Joseph Coyle\thanks{ Department of Mathematics, Monmouth University, New Jersey, USA. \email{jcoyle@monmouth.edu}:}}
	\and {
		Nilima Nigam\thanks{Department of Mathematics, Simon Fraser University, Burnaby, Canada. \email{nigam@math.sfu.ca}}
}}

\usepackage{amsopn}

\newcommand{\e}{\rm{ e}}
\renewcommand{\d}{\mathbf{d}}
\newcommand{\x}{\mathbf{x}}
\newcommand{\M}{\mathtt{M}}
\newcommand{\D}{\mathtt{D}}
\renewcommand{\S}{\mathtt{S}}
\newcommand{\SyS}{\mathtt{SyS}}
\renewcommand{\v}{\mathbf{v}}
\newcommand{\U}{\mathtt{U}}
\newcommand{\C}{\mathtt{C}}
\newcommand{\T}{\mathtt{T}}
\renewcommand{\P}{\mathtt{P}}
\newcommand{\BM}[1]{\mbox{\boldmath{$#1$}}}
\newcommand{\pbm}{{\BM p} }
\newcommand{\nbm}{{\BM n} }
\newcommand*{\ldblbrace}{\{\mskip-5mu\{}
\newcommand*{\rdblbrace}{\}\mskip-5mu\}}
\usepackage{xspace}
\newcommand{\MATLAB}{\textsc{Matlab}\xspace}

\begin{document}
\maketitle

\begin{abstract}
	{Plane-wave Trefftz methods (PWB) for the Helmholtz equation offer significant advantages over standard discretization approaches whose implementation employs more general polynomial basis functions.  A disadvantage of these methods is the  poor conditioning of the system matrices. In the present paper, we carefully examine the conditioning of the plane-wave discontinuous Galerkin method with reference to a single element. The properties of the mass and stiffness matrices depend on the size and geometry of the element. We study the mass and system matrices arising from a PWB on a single disk-shaped element. We then examine some preconditioning strategies, and present results showing their behaviour with three different criteria: conditioning, the behaviour of GMRES residuals, and impact on the $L^2$-error.}
\end{abstract}



\section{Introduction}
Finite element approaches based on Trefftz methods \--- where the approximation basis is built from solutions of the differential operator being studied  \--- are well-established. They offer many advantages including flexibility in terms of tesselating domains, their theoretical convergence rates and potential for dimension reduction. Unfortunately, it has also been observed that the Trefftz methods for the Helmholtz operator lead to severely ill-conditioned matrices. The theoretically-predicted rates of convergence may not be observed in practice. It is our goal in this paper to study these issues within the context of a {\it single} spatial element.

We recall from \cite{trefftzreview}
that the  Trefftz basis functions are plane waves of the form 
\begin{equation}\label{eq:phi}
	\varphi_\ell:=\e^{\iota \kappa \d_\ell\cdot(\x-\x_{K})}
\end{equation}
where $\d_\ell\in \mathbb{R}^n$ are distinct propagation directions, $\kappa$ is the (real) wave number, and $\x_{K}$ is a point in the interior of the physical element $K$({\it the center}). The plane-wave basis functions (PWB) associated with distinct directions  form a linearly independent set. The local space is constructed as
\begin{equation}
	V_{p}(K)=\left  \{\v:\v(\x) = \sum_{\ell=1}^{p} \alpha_\ell {\e}^{\iota \kappa \d_\ell\cdot(\x-\x_{K})}, \quad \alpha_\ell \in \mathbb{C} \right \}.
\end{equation}
It is well-documented, see for example \cite{trefftzreview}, that as the number of plane wave directions $p$ increases, the conditioning of the system matrix resulting from  variational methods using PWB degenerates. 

A standard Helmholtz boundary value problem is 
\begin{eqnarray} 
	\Delta u + \kappa^2 u &=&0, \quad x\in \Omega, \label{eq:helmeq} \\
	\frac{\partial u }{\partial \nbm } + i \kappa u &=& g,\quad x\in \partial \Omega, \label{eq:helmbc}
\end{eqnarray}
where $\Omega$ is a bounded Lipschitz domain in ${\mathbb R}^2$ and $\nbm$ is  the unit outward normal to the boundary of $\Omega,$ denoted by $\partial \Omega.$
A typical  discrete formulation in ${\mathbb R}^2$ \cite{trefftzreview} would be to determine 
$\displaystyle u_p \in {V}_p ({\cal T}) =  \bigcup_{K \in {\cal T}} V_p(K)$
such that  ${\cal A}(u_p,v_p) = \ell(v_p)$ for all $v_p \in V_p({\cal T}),$ where
\begin{eqnarray}
	{\cal A}(u_p,v_p) &=& \int_{{\cal E}^{I}} \ldblbrace u_p \rdblbrace \llbracket {\overline{ \nabla v_p}} \rrbracket_N ds  + 
	i \kappa^{-1} \int_{{\cal E}^{I}} \beta \llbracket \nabla u_p \rrbracket_N \llbracket {\overline{\nabla v}} \rrbracket_N ds \nonumber \\
	&-& \int_{{\cal E}^{I}} \ldblbrace \nabla_h u_p \rdblbrace \cdot \llbracket {\overline{  v_p}} \rrbracket_N ds  + 
	i \kappa \int_{{\cal E}^{I}} \alpha \llbracket  u_p \rrbracket_N  \cdot \llbracket {\overline{ v_p}} \rrbracket_N ds \nonumber \\
	&+& \int_{{\cal E}^{B}} (1-\delta) u_p \, \overline{ \nabla v_p \cdot {\BM n}} \, ds  + 
	i \kappa^{-1} \int_{{\cal E}^{B}} \delta \, \nabla u_p \cdot {\BM n} \,  {\overline{\nabla v_p \cdot {\BM n}}}ds \nonumber \\
	&-& \int_{{\cal E}^{B}} \delta \, \nabla u_p \cdot {\BM n} \, \overline{  v_p}  ds  + 
	i \kappa \int_{{\cal E}^{B}} (1-\delta) u_p \overline{ v_p} ds, \label{eq_A} 
\end{eqnarray}
\begin{equation}\label{eq_ell} 
	\ell(v_p) = i \omega^{-1} \int_{{\cal E}^{B}} \delta g \overline{ \nabla v_p \cdot {\BM n}} dS + 
	\int_{{\cal E}^{B}} (1-\delta) g \overline{v_p} dS,
\end{equation}
where
\begin{equation}\label{eq_g} 
	g  = \nabla u_p \cdot \nbm + i \kappa u_p \mbox{ on } \partial \Omega
\end{equation}
and ${\cal T}$ is a regular mesh, designating the   interior and boundary edges as ${\cal E}^{I}$ and ${\cal E}^{B},$ respectively.  The derivation is based on an initial formulation over a single physical element $K \in {\cal T}$ followed by defining appropriate fluxes on the boundary by employing well-defined jumps $\llbracket \cdot \rrbracket_N$ and averages $\ldblbrace \cdot \rdblbrace$ while summing over the physical elements in ${\cal T}.$ The choice of the parameter $\delta \in (0,\frac{1}{2}]$ yields different approximation properties, \cite{trefftzreview}.

Consider the example where we set $\d = \langle \cos(\pi/4),\sin(\pi/4) \rangle$, $ \pbm_0= (2,-4)$, and define
\begin{equation}
	u(\x) = \e^{\iota \kappa ( \x \cdot\d)} + 4J_0(\kappa |\x_m - \pbm_0|)
\end{equation}
where $J_0$ is the Bessel function of the first kind, order zero and in both (\ref{eq_A}) and (\ref{eq_ell}) we set $\alpha = \beta = \delta = 1/2$.  The domain of interest is taken to be the square shown, with mesh, in Figure \ref{fig1}(a).  We discretize (\ref{eq_A}) and (\ref{eq_ell}) and compute the solution for $p$ evenly spaced plane waves on each element for $p = 4,...,50$ and $\kappa = \pi,5\pi,$ and $10\pi.$ The real part of the computed solution for $\kappa = 5 \pi$ and $p = 26$ is shown in Figure \ref{fig1}(b).  For each $p$-value, we compute both the condition number of the system matrix, Figure  \ref{fig1}(c), and the $L^2(\Omega)$ relative error of the solution, Figure \ref{fig1} (d).  One can see the typical degradation of the condition number as the number of plane waves $p$ increases and the corresponding error in each case.  
\begin{figure}\label{fig1}
	\centering
	$\begin{array}{cc}
		\includegraphics[width=0.47\linewidth]{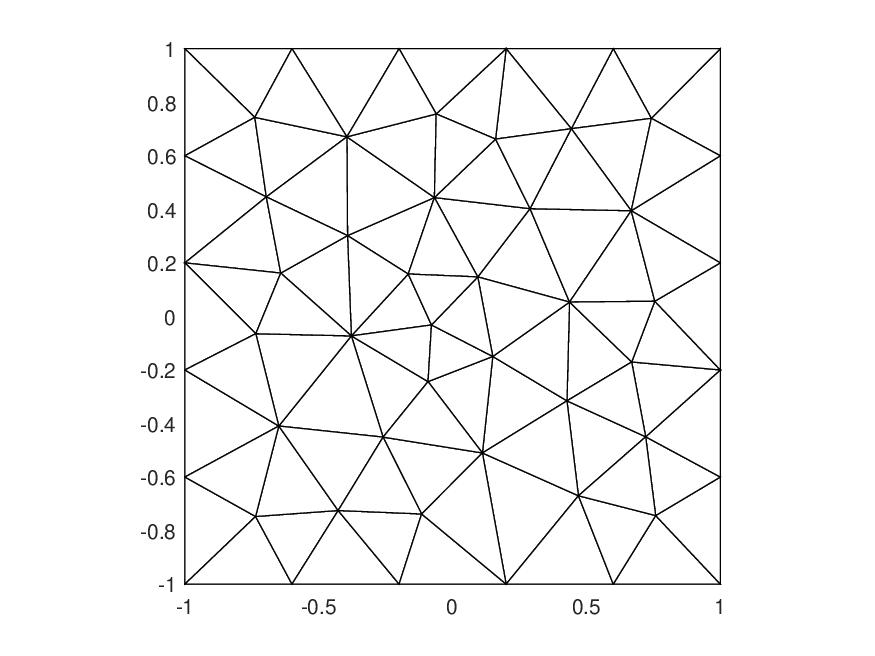} &
		\includegraphics[width=0.47\linewidth]{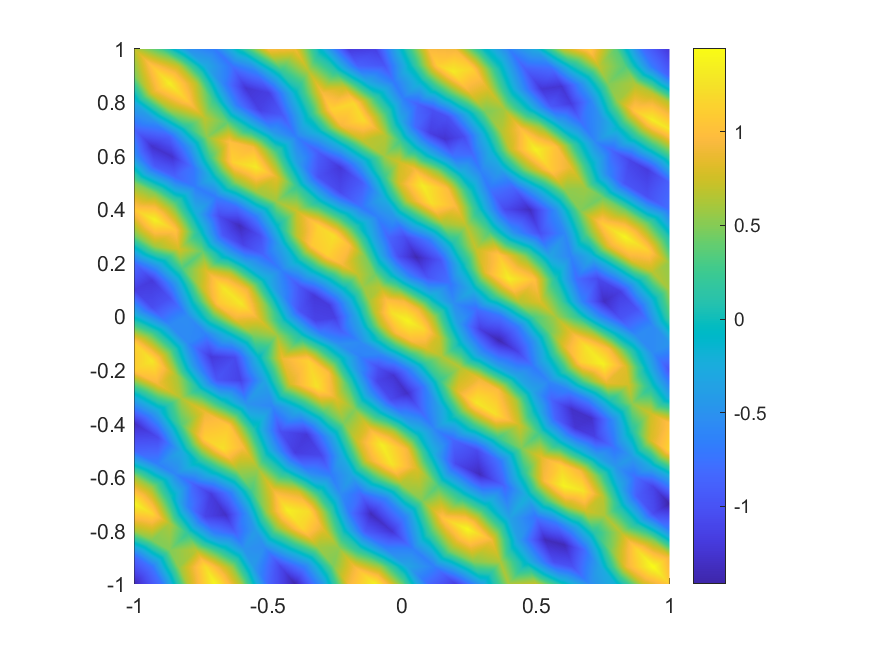} \\
		\mbox{(a) Domain and Mesh} & \mbox{(b) $Re(u_p)
			$, $p = 26$ and $\kappa = 5 \pi$} \\
		\includegraphics[width=0.45\linewidth]{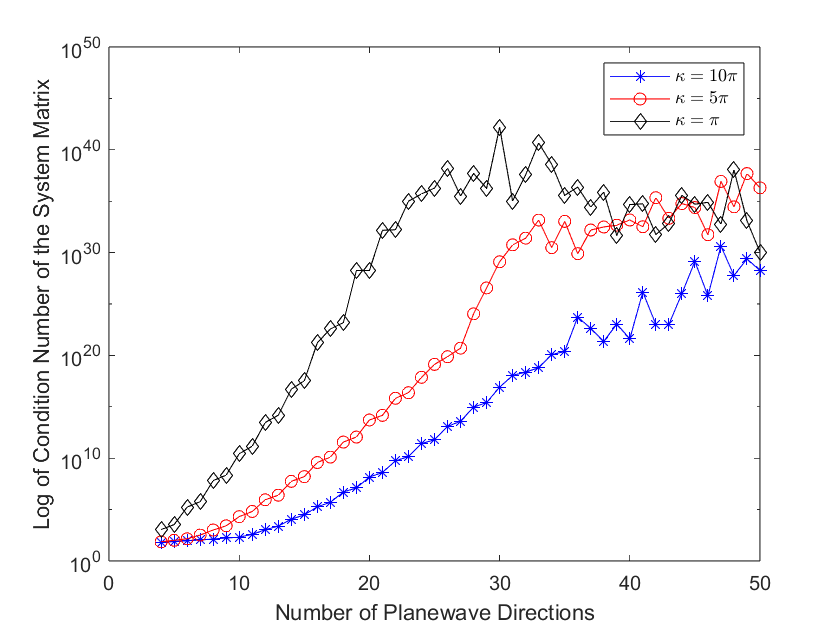} &
		\includegraphics[width=0.45\linewidth]{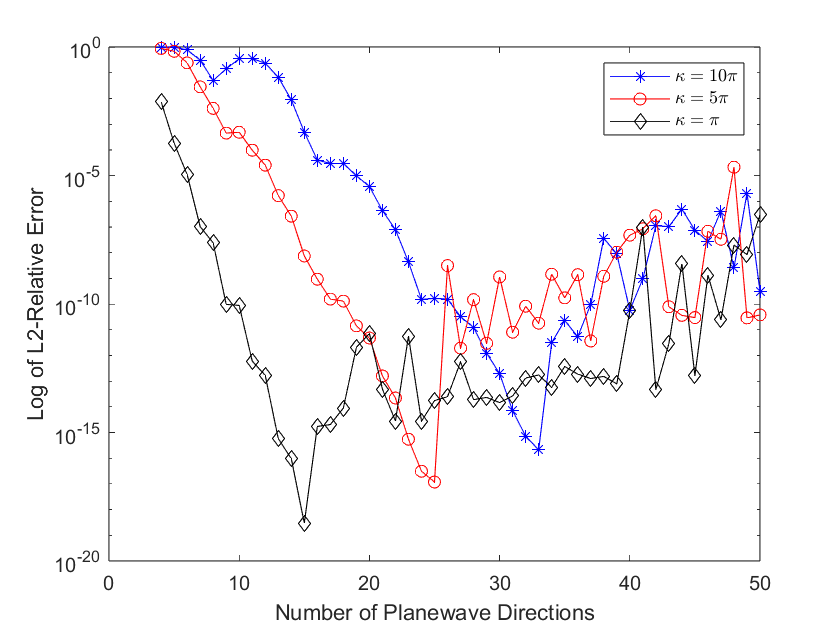} \, \\
		\mbox{(c) Condition Number} & \mbox{(d) $L^2$ Relative Error}
	\end{array}$
	\caption{The mesh of the domain $[-1,1]^2$ is shown in (a) and in (b) we see the real part of the computed solution for $p= 26$ and $\kappa = 5\pi.$  The condition number (c) and $L^2$-relative error (d) are given versus $p$ for $\kappa$ values of $\pi, 5 \pi,$ and $10\pi.$ }
\end{figure}

Clearly the issues of poor conditioning seen in Figure \ref{fig1} are related to the local condition number on a single element. In this paper we therefore focus on  questions around conditioning and the impact of preconditioning choices for the system matrix associated with solving the PDE in the specific case where $\Omega$ is a {\it single} element, $ K$. A successful preconditioner on a single element can then be used as part of a block-preconditioning strategy when computing on a mesh with many elements.

   Of particular interest is the conditioning of 3 common matrices: the mass, cross, and stiffness matrices. In particular, the entries of the local matrices on an element $K$ are 
\begin{eqnarray}
	\text{(mass matrix)}\, \, [\M]_{m,\ell}&:=&\int_{\partial K} \varphi_m\overline{\varphi_{\ell}} ds, \label{eq:mass} \\
	\text{(cross matrix)}\, \, [\S]_{m,\ell} &:=&\int_{\partial K} \nabla\varphi_m \cdot \nbm \, \overline{\varphi_{\ell}} ds, \label{eq:cross} 
\end{eqnarray}
and
\begin{eqnarray}
	\text{(stiffness matrix)}\, \, [\D]_{m,\ell}&:=&\int_{\partial K} (\nabla \varphi_m\cdot \nbm)(\overline{\nabla \varphi_{\ell}\cdot \nbm}) ds \label{eq:stiff}
\end{eqnarray}
where $\partial K$ is the boundary of the physical element and $\nbm$ is the unit outward normal.  Note that $\M$ and $\D$ are Hermitian.  In this case, we no longer have interior edges to consider and ${\cal E}^I$ is empty and the resulting system \eqref{eq_A} we solve is
\begin{equation}\label{eq:systemlinear} \mathtt{SyS}\, {\tt x}:= (\kappa^2  \M +  i\kappa (\S-\S^*) + \D){\tt x} = {\tt f}
\end{equation}
where $\S^*$ denotes the complex conjugate transpose of $\S$ and ${\tt f}$ is constructed from (\ref{eq_ell}).

Our goal in this paper is to provide an analysis of the conditioning of these local matrices relative to  the mesh size $h$, the wave number $\kappa$, or the number of directions $p$ on a single element $K$.
To motivate this, we observe that if we employ a plane-wave basis, then we do not have the same notion of a {\it reference} element as we typically do for polynomial basis functions. That is to say, defining a local basis on a physical element by way of affine transformation from a reference element is not possible. 
However, we observe that the {\it size} of the physical element has been observed to have an important effect on the conditioning of the matrices, even though  one of the advantages of the PWB method is that the {\it shape} of the physical element can be quite general. Consequently, we begin our presentation with a simple shape for $K,$ the disk of radius $h.$  As we shall see going forward, the conditioning also depends in a slightly less intuitive way on the geometry of the element.
The related analysis leads to a natural progression first to regular, then cyclic, and finally general polygons. We provide some numerical experiments which highlight the theoretical results, and then suggest some preconditioning strategies.

We point out some previous works which are of particular relevance to the scope of this paper.
In \cite{CongreveMGS}, the authors propose a preconditioning strategy based on orthogonalization via the modified Gram-Schmidt method, and how such a strategy would improve the performance of a GMRES-based solution strategy. SVD-based and QR-based approaches for plane-wave Trefftz methods are studied carefully in \cite{Barucq}. In parallel, the elegant theory of frames and overcomplete bases \cite{adcock1,adcock2} has been used to investigate the {\it inherent} instability of using plane wave methods; in a very recent paper by \cite{Parolin}, the use of evanescent plane waves is shown to provide stable approximation for the Helmholtz equation in a disk. Our approach in this paper is somewhat different: we first study the behaviour of important matrices arising in the DG-plane wave discretization of the Helmholtz equation on a single element, as in (say) \cite{Parolin} or \cite{Perrey}. Section 3 of the present paper is closely linked to results in Section 2.2 of \cite{Perrey}, where the conditioning of the mass matrix on a disk is discussed. We also study the eigenvalues and conditioning on the so-called cross and stiffness matrices on the disk there. We next study these matrices on {\it cyclic polygons}. On polygonal domains $K$  these matrices are not even Toeplitz. However, we show they are {\it close} to Toeplitz, and therefore conjecture preconditioning strategies based on circulant preconditioners may be effective.

 There is a considerable literature on preconditioning Toeplitz matrices (see, e.g. \cite{strang,chan, HANKE1998137,Gray,DIBENEDETTO199335,potts}) as well as nearly Toeplitz matrices (see, e.g. \cite{ku}). Based on these ideas a number of possible circulant preconditioners suggest themselves, and we compare their impact on improving the conditioning of the relevant matrices. We also assess their behaviour as part of both a direct solve and a GMRES computation.

In this paper, we address the following points:
\begin{itemize}
	\item Is there a precise characterization of the spectrum and condition numbers of the mass, stiffness, and cross matrices on a disk? How do these depend on the wave number $\kappa$, the element radius $h,$ and the number of plane waves $p$? (Section 3)
	\item How is the spectrum of the system matrix $\SyS$ on a regular/cylic polygon related to that on a disk? (Section 4)
	\item Can simple (e.g. circulant) preconditioners improve the conditioning of the system matrix $\SyS$ on a single element? (Section 5)
	\item Which of these preconditioners improves accuracy for the iterative solution of (\ref{eq:systemlinear}) via GMRES? (Section 6) 
\end{itemize}
In subsequent work, we investigate preconditioning strategies for meshes comprised of several elements.
\section{Notation and background}
We first set notation, and review some results which will be used frequently.  Throughout we consider $p$ plane wave directions of unit length given by  
\begin{equation}\label{eq:directions}
	\d_m:= \langle \cos(\theta_m), \sin(\theta_m)  \rangle; \qquad \theta_m := \frac{2\pi}{p}m, \quad m = 0,...,p-1
\end{equation}
and a point $\x(t)$ in the ${\mathbb R}^2 $ plane is represented as
\begin{equation}\label{eq:point}
	\x(t) = |\x(t)|\left \langle \cos(t), \sin(t)\right \rangle.
\end{equation}
Noting that 
\begin{eqnarray}
	\x(t) \cdot\d_m &=& |\x(t)| (\cos (t) \cos(\theta_m) + \sin (t) \sin(\theta_m))  = |\x(t)| \cos( \theta_m -t), \nonumber
\end{eqnarray}
we can then write
\begin{eqnarray}
	\x(t)\cdot(\d_m-\d_\ell) 
	&=&  2|\x(t)|\sin \left (\frac{\theta_m - \theta_{\ell}}{2} \right )\sin \left ( t-  \frac{\theta_m + \theta_{\ell}}{2}\right ).   \label{eq:xdotdm}
\end{eqnarray} 
Defining
\begin{equation}\label{eq:defineab}
	a(n):=2 \sin \left (\frac{\theta_n}{2} \right ) \qquad \mbox{ and } \qquad b(n,t):= \sin\left (t - \frac{\theta_n}{2} \right ),\end{equation}
leads to 
\begin{equation}\label{eq:prop1}
	\x(t)\cdot(\d_m-\d_\ell) = |\x(t)| a(m-\ell) b(m+\ell,t).
\end{equation}
Using properties of the sine function, we get
\begin{eqnarray}
	a(m-\ell) = -a(\ell-m), &\quad
	a(x\pm p) = -a(x), \label{eq:amodp2} 
\end{eqnarray}
and
\begin{eqnarray}
	b(x\pm p,t) &=& - b(x,t). \label{eq:amodp3}
\end{eqnarray}
We will also take advantage of the fact that if $\nbm =\langle \cos(t), \sin(t) \rangle$ is a unit vector then
\begin{eqnarray}
	\nabla \varphi_m \cdot \nbm = i \kappa |\x| \varphi_m (\d_m \cdot \nbm)
	= i \kappa h\varphi_m \cos(\theta_m -t).\label{eq:grad}
\end{eqnarray}
\subsection{Identities involving Bessel functions}
The Bessel function of the first kind and of order $n$, $J_n(x)$, will appear frequently in the subsequent sections. We record some important facts which can be found in, for example, \cite{Lebedev,Watson} , or readily derived:
\begin{itemize}
	\item The Fourier expansion of the function $e^{i x \sin(t)}$ can be expressed via the Jacobi-Anger expansion:
	$$ e^{i x\sin(t)} = \sum_{n=-\infty}^{\infty} J_n(x)e^{in t}.$$
	Consequently, the Bessel functions of the first kind $J_n(t)$ are the Fourier coefficients of a periodic and analytic function, and must decay rapidly with $n$.
	
	For integer $n$ and $\alpha \in {\mathbb R},$ the Jacobi-Anger expansion provides an identity we use frequently:
	\begin{eqnarray}
		J_n(x) &=&  \frac{1}{2\pi}\int_{0}^{2\pi} e^{i x \sin(t) - in t} \, dt  
		=\frac{1}{\pi}\int_{0}^{2\pi} \cos( x \sin(t) - n t) \, dt. \label{BesselIntegral} 
	\end{eqnarray}
	
	\item The Bessel functions satisfy, for integer $n$,  $J_{n}(x) = (-1)^n J_{-n}(x)$ as well as $J_{n}(x) = (-1)^nJ_n(-x)$.
	\item We shall encounter integrals of the form 
	\begin{equation}\nonumber
		\int_0^{2\pi} e^{i A \sin(t-\beta_1)} \cos(qt-\beta_2) dt
	\end{equation}
	where $A, \beta_1,\beta_2$ are fixed and $q$ is an integer. Simple manipulations yield 
	\begin{equation}
		\int_0^{2\pi} e^{i A \sin(t-\beta_1)} \cos(qt-\beta_2) dt = 
		\begin{cases}
			2\pi J_q(A)\cos(q\beta_1-\beta_2),& q \,\, \mbox{even}\\
			&  \\
			2\pi i J_q(A)\sin(q\beta_1-\beta_2), & q \,\, \mbox{odd} 		
		\end{cases}.
		\label{eq:integralconvertp}
	\end{equation}
	We also recall that for nonnegative $n$,
	\begin{equation}\label{eq:besselsmallargument}
		J_n(z) \approx \frac{1}{\Gamma(n+1)} 
		\left(\frac{z}{2}\right)^n, \quad \text{when} \,\, 0 < z \ll \sqrt{n + 1}.
	\end{equation} 

\end{itemize}
Finally, we shall need the following identity in Section 3.
\begin{lemma} \label{lemma:bessel-coeffs} Let $B>0$ be constant. For $M,L \in \mathbb{Z}$, we have
	\begin{equation}
		I_1(M,L) :=	\int_0^{2\pi} J_{2M} \left( B\sin\left ( \frac{x}{2} \right ) \right )\, e^{-i Lx}\, dx =
		 \int_0^{2\pi} J_{2L}\left(B\sin(t)\right)\, e^{-i2Mt}\, dt.
	\end{equation} 
\end{lemma}
\begin{proof}
	We easily see that $I_1(M,L)$ is real, by exploiting the even behaviour of $J_{2M}$. Next, 	
	for fixed integers $M, L$, we calculate, using \eqref{BesselIntegral}
	\begin{align*}
		I_1(2M,L)&=	\int_0^{2\pi} J_{2M}\left( B\sin\left ( \frac{x}{2} \right ) \right ) \, e^{i Lx}\, dx 
		= 2\int_0^{\pi}J_{2M}(B\sin(s))\, e^{i2Ls}\, ds\\
		&= 	\int_0^{2\pi}J_{2M}(B\sin(s))\, e^{i2Ls}\, ds
		=  \frac{1}{2\pi}	\int_0^{2\pi}\int_0^{2\pi} e^{i B\sin(s)\sin(t)-i2Mt -2Ls}\, dt\, ds\\
		&=\int_0^{2\pi} J_{2L}\left(B\sin(t)\right)\, e^{-i2Mt}\, dt.
	\end{align*} 
	That is, $I_1$ is the (scaled) $2M$ Fourier (cosine) coefficient of $J_{2L}(B\sin(t)).$
\end{proof}
\subsection{Properties of circulant and Toeplitz matrices}
Throughout, if $\mathtt{A}$ is an $n\times n$ matrix with $n$ eigenvalues $\lambda_{i}(\mathtt{A}), i=0,1,...,n-1$, we denote 
\begin{equation} \label{eq:notationlambda}
	\Lambda(\mathtt{A}):= \mathtt{diag}[\lambda_0(\mathtt{A}),\lambda_1(\mathtt{A}),.., \lambda_{n-1}(\mathtt{A})]. 
\end{equation} 
If the context is clear, we will suppress the argument in $\lambda_{i}(\tt A)$. 

We enumerate well-known properties of Toeplitz and circulant matrices, see for example, \cite{Gray}. Consider the $n \times n$  matrices of form 
\begin{equation*}
\T_n=\begin{bmatrix}
	t_{0}&t_{1}&t_{2}&\cdots&t_{n-1}\\
	t_{-1}&t_{0}  &t_{1}&\cdots&t_{n-2}\\ 
	t_{-2}  &t_{-1}    & t_0   &     & \vdots \\  
	\vdots& \vdots& & \ddots&t_1\\
	t_{-(n-1)} & t_{-(n-2)}&\cdots&t_{-1} &t_0 \\
\end{bmatrix}\hspace{-0.04in},  \C_n=\begin{bmatrix}
c_{0}&c_{1}&c_{2}&\cdots&c_{n-1}\\
c_{n-1}&c_{0}  &c_{1}&\cdots&c_{n-2}\\ 
c_{n-2}  &c_{n-1}    & c_0   &     & \vdots \\  
\vdots& \vdots& & \ddots&c_1\\
c_{1} & c_{2}&\cdots& c_{n-1}&c_0 \\
\end{bmatrix}. 
\end{equation*}
The matrix $\T_n$ with constant diagonals is a Toeplitz matrix, and $\C_n$ is a special form of the Toeplitz matrix in which each row is a right cyclic shift of the row above. Circulant matrices are uniquely specified by their first row and oftentimes denoted by  $\C_n = \mathtt{circ}[c_0,c_1,...,c_{n-1}].$
Multiplication by a circulant matrix is equivalent to a circular convolution.  It should be noted that an $n \times n$ matrix $\C$ is circulant if and only if 
\begin{equation}\label{iscirc}
[\C]_{m,\ell} = [\C]_{i,j} \mbox{ when } i-j = (m- \ell) \mod{n}.
\end{equation}
The eigenvalues of a circulant matrix $\C = {\rm circ}[c_0,c_1,...,c_{n-1}]$  are given by
\begin{equation}\label{eq:eigofC}
	\lambda_q(\C)=\sum_{j=0}^{n-1}c_j e^{-i j \theta_q }, \qquad q=0,1,...,n-1, \qquad \theta_q:=\frac{2\pi}{n}q, \end{equation}
and the associated eigenvectors are 
\begin{equation}\label{eq:eigenvectofC}
	v_q = [1,e^{-i \theta_q},...,e^{- i (n-1) \theta_q}]^{\sf T}, \qquad q=0,1,...,n-1.
\end{equation}
We denote by $\U$ the $n \times n$ unitary matrix whose columns are the eigenvectors $v_q$ above:
\begin{equation}\label{eq:defineU}
	\U:=[v_0|v_1|...|v_{n-1}].
\end{equation}

If $c_j =  f\left ( \theta_j \right )$ for a continuous periodic  {\it generating function} $f:[0,2\pi]\rightarrow [0,2\pi]$ then 
\begin{eqnarray}
	\lambda_q (\C) &=& \sum_{j=0}^{n-1} f(\theta_j) e^{-i q \theta_j } 
	\approx \frac{n}{2\pi} \int_0^{2\pi} f(x) e^{-i qx} \, dx, \qquad |q|\leq n-1.	
\end{eqnarray}
The sum above can be interpreted in terms of a quadrature rule, in which case the integral term defines the  Fourier transform of $f$ up to $n-1$. The quadrature is exact if $f\in {\rm span}\left \{e^{i jx}: \, |j|\leq p-1 \right \}.$
 
Several of the matrices we discuss in this paper are not only circulant, but also real. We refer elsewhere (e.g., \cite{circulantreal}) for a detailed exposition of the properties of such matrices, and record some facts we use.
 The $n \times n$ matrix ${\tt C}$ is real circulant if and only if \begin{equation}\label{eq:eigenvalue_real_circulant}
		\lambda_0(\C)\in \mathbb{R}, \quad \text{and} \quad \lambda_q(\C) = \bar{\lambda}_{q-n}(\C), \quad q=1,2,...,n-1.
	\end{equation}	
If the $p \times p$ matrix $\C$ is real, circulant and symmetric, then all the eigenvalues are real. Moreover, from \eqref{eq:eigenvalue_real_circulant} we see that several eigenvalues have multiplicity. In fact, we see \cite{circulantreal}
	\begin{itemize}
		\item If $n$ is odd, then $\lambda_0(\C)$ has odd multiplicity, and all other eigenvalues have even multiplicity
		\item If $n$ is even,  either $\C$ has precisely two eigenvalues of  odd multiplicity, or none. All others eigenvalues have even multiplicity. 
	\end{itemize}

Any circulant matrix admits the decomposition
\begin{equation}\label{eq:cdecomp}
	\C := \U^*\mathtt{\Lambda}\U
\end{equation}
where $\U$ is the unitary matrix defined in \eqref{eq:defineU}, and 
	$\mathtt{\Lambda} = \mathtt{diag}[\lambda_0, \lambda_1,...,\lambda_{n-1}]$. 
Application of circulant matrices is therefore fast (via the FFT).  Matrix inversion is also fast, 
\begin{equation}\label{eq:cinvdecomp}
	\C^{-1} = \U^*{\mathtt \Lambda^{-1}}\U.
\end{equation}


\section{The behaviour of the PWB on a disk}

To develop some understanding of the theoretical issues around the conditioning of the mass, cross, and stiffness matrices on general polygons, we first examine the situation of a single element $K$ which is a disk of radius $h$ centered at the origin. As mentioned in the Introduction, the PWB behaves differently from traditional polynomial basis methods, in that affine-mapped physical elements do not map the PWB to another one. Nonetheless, the size of the physical element plays an important role in the conditioning of the PWB. Motivated by these facts, we first examine the behaviour of the PWB on an element $K$ which is a {\it disk.} This allows us to obtain a quantitative description of the impact of size as well as wave number on an isotropic physical element. These results can then be used to provide insight into the impact of both the mesh size and the number of plane wave basis functions on the conditioning of a more general-shaped physical element. A similar analysis was conducted in Section 2 of \cite{Perrey}, with a focus on the mass matrix on the disk.

Let $K$ denote a disk centered at the origin with radius $h$, and consider a PWB centered at $\x_K=(0,0)$ with equally-spaced planewave directions. If $\x\in \partial K$ then $|\x| = h$ and employing \eqref{eq:prop1} we obtain 
\begin{equation}\label{eq:prop2}
	\x \cdot (\d_m-\d_\ell) = h a(\ell-m)b(m+\ell,t).
\end{equation}

We first arrive at closed-form expressions of the matrix entries (the generating functions) in terms of Bessel functions  
\begin{theorem}\label{thm:disk_entry} Let $K$ be a disk of radius $h$ centered at the origin. For $p$ evenly spaced plane waves
	the general entries of the mass, cross, and stiffness matrices are, for $\ell = 1,...,p,$
	\begin{eqnarray} 
		[\M^{disk}]_{m,\ell} &=& 2 \pi h J_0(\kappa h a(m-\ell)), \label{eq:Mdiskentry} 
	\end{eqnarray}
	\begin{eqnarray}
		[\S^{disk}]_{m,\ell} &=& \pi \kappa h^2 a(m-
		\ell)J_1(\kappa h a(m-\ell)),\label{eq:Sdiskentry}
	\end{eqnarray}
	and
	\begin{eqnarray}
		[\D^{disk}]_{m,\ell} &=& \pi\kappa^2h^3  J_0(\kappa h a(m-\ell))\cos \left (\frac{2(m-\ell)}{p}\pi \right)  \nonumber \\ 
		&&+ \, \pi\kappa^2h^3 J_2(\kappa h a(m-\ell)),\label{eq:Ddiskentry} 
	\end{eqnarray}
	respectively.
\end{theorem} These expressions will lead to the generating functions of these matrices.
\begin{proof} To establish \eqref{eq:Mdiskentry}, we employ (\ref{eq:prop2}) and the defnintion of $b$ given in (\ref{eq:defineab}) to get 
	\begin{eqnarray*}
		[\M^{disk}]_{m,\ell}&=& \int_0^{2 \pi} e^{\iota \kappa h (\d_m-\d_\ell)\cdot \x} h dt \\ \nonumber 
		&=&h \int_0^{2\pi} e^{\iota \kappa h a(m-\ell)b(m+\ell,t)} dt\\ 
		&=& h \int_0^{2\pi} e^{\iota \kappa h a(m-\ell)\sin \left (t- \frac{\pi(m+\ell)}{p}\right )}dt \\ \nonumber
		&=& h \int_{- \frac{\pi(m+\ell)}{p}}^{2\pi -  \frac{\pi(m+\ell)}{p}} e^{\iota \kappa h a(m-\ell) \sin(t)} dt \\ \nonumber
		&=&2\pi h J_0(\kappa h a(m-\ell)),
	\end{eqnarray*}
	where we have used the integral representation in \eqref{BesselIntegral}. Similarly, for the cross matrix we use  (\ref{eq:grad}) to compute
	\begin{eqnarray*}
		[\S^{disk}]_{m,\ell}
		&=&\int_0^{2\pi} i \kappa h  e^{i \kappa h a(m-\ell)b(m+\ell,t) }\cos\left (t-\frac{2\pi}{p}m \right ) h dt\\
		&=& i \kappa h^2 (2\pi i \sin \left (\frac{(m+\ell)}{p}\pi - \frac{2m}{p}\pi \right )J_1(\kappa h a(m-\ell )) \\
		&=&  \pi \kappa h^2 a(m-\ell)J_1(\kappa h a(m-\ell)) 
	\end{eqnarray*} 
	where we have taken advantage of \eqref{eq:integralconvertp} with $A=\kappa h a(m-\ell),q=1, \beta_1=\frac{(m+\ell)}{p}\pi,$ and $ \beta_2= \frac{2m}{p} \pi$. Finally, in the case of the stiffness matrix,
	\begin{eqnarray*}
		[\D^{disk}]_{m,\ell}
		&=&(\kappa h)^2\int_0^{2\pi}e^{i\kappa h a(m-\ell)b(m+\ell,t)}\cos \left ( t-\frac{2m}{p} \pi\right )\cos \left (t-\frac{2\ell}{p} \pi \right )hdt\\
		&=& \frac{1}{2}\kappa^2h^3\int_0^{2\pi}e^{i\kappa h a(m-\ell)b(m+\ell,t)} \cos \left (\frac{2(m-\ell)}{p}\pi \right )  \\
		&&+ \frac{1}{2}\kappa^2h^3\int_0^{2\pi}e^{i\kappa h a(m-\ell)b(m+\ell,t)} \cos \left (2t-\frac{2(m+\ell)}{p}\pi \right ) dt.\\
	\end{eqnarray*}	
	We treat the two integrals on the right side separately. 
	First, we use \eqref{eq:integralconvertp} with  $A=\kappa h a(m-\ell), q=0, \beta_1= \frac{(m+\ell)}{p}\pi,$ and $\beta_2 = -\frac{2(m-\ell)}{p} \pi$ to get
	\begin{eqnarray}
\frac{1}{2}\kappa^2h^3\int_0^{2\pi}e^{i\kappa h a(m-\ell)b(m+\ell,t)} \cos \left (\frac{2(m-\ell)}{p}\pi \right ) dt =\pi\kappa^2h^3 \cos \left (\frac{2(m-\ell)}{p}\pi \right)  J_0(\kappa h a(m-\ell)). \nonumber
	\end{eqnarray}
	Next, we see with $A=\kappa h a(m-\ell), q=2, \beta_1= \frac{(m+\ell)}{p}\pi,$ and $\beta_2 = \frac{2(m+\ell)}{p}\pi$, 
	\begin{eqnarray}
		\frac{1}{2}\kappa^2h^3\int_0^{2\pi}e^{i\kappa h a(m-\ell)b(m+\ell,t)} \cos \left (2t-\frac{2(m+\ell)}{p}\pi \right ) dt &=& \pi\kappa^2h^3 (2\pi) J_2(\kappa h a(m-\ell))\cos(2\beta_1-\beta_2) \nonumber \\
		&=& \pi\kappa^2h^3  J_2(\kappa h a(m-\ell)). \nonumber
\end{eqnarray}
This yields the result in (\ref{eq:Ddiskentry}).\end{proof}

Theorem \ref{thm:disk_entry} allows us, by inspection and using the identities in the previous section, to readily conclude the following result.

\begin{theorem} \label{thm2} Let $K$ be a disk of radius $h$ centered at the origin.  The mass, cross, and stiffness matrices in (\ref{eq:mass}) - (\ref{eq:stiff}) resulting from  $p$ evenly spaced plane waves are real, circulant and symmetric. 
\end{theorem}
\begin{proof}
	The fact that the entries for each matrix are real is evident from the results of the previous theorem.  The symmetry of each matrix follows from follows by noting $a(n)$ and $J_1$ are odd in their argument while $J_0$ and $J_2$ are even.  It is also clear that each of these matrices is Toeplitz: the $m,\ell$ entry of these matrices only depends on the difference $m-\ell$.
	
	It remains to show the matrices are circulant, i.e., that each row is a right-shift of the row above it. If $(m-\ell) = (i-j)  \mod p$, then $a(m-\ell) = -a(i-j)$  and hence
	\begin{equation*}	
		[\M^{disk}]_{m,\ell} =  2 \pi h J_0(\kappa h -a(i-j)) = [\M^{disk}]_{i,j}\quad \text{if}\,\, (m-\ell) =(i-j) \, \mod p.
	\end{equation*}
	Similar arguments hold for the cross and stiffness matrices.
\end{proof}

We are now in a position to explicitly compute the spectra of the mass ($\M^{disk}$), cross ($\S^{disk}$), and stiffness ($\D^{disk}$) matrices.

In each, if the number of plane waves $p$ is odd, then $\lambda_0$ has multiplicity one, and the other eigenvalues have multiplicity 2. If the number of plane waves is even, $\lambda_0$ and $\lambda_{p/2}$ have multiplicity one and the rest have multiplicity 2.

\begin{theorem}\label{thm:mass_spectrum}
	Let $K$ be a disk of radius $h$ centered at the origin, and consider the mass matrix (\ref{eq:mass}) resulting from  $p$ evenly spaced plane waves. Then, the following hold for $L=0,...,p-1$:		 
	\begin{enumerate}
		
		\item The generating function for $\M^{disk}$ is   $F_{\M}^{disk}:=2\pi h J_0\left(2\kappa h \sin\left(\frac{x}{2} \right )\right )$. 
		\item There is a simple eigenvalue of $\M^{disk}$ given by
		$$ \lambda_0(\M^{disk})  =\sum_{j=0}^{p-1}F_{\M}^{disk} \approx p h \int_0^{2\pi}J_0(2\kappa h)\sin(x)\, dx $$
		\item The multiplicity-two eigenvalues of $\M^{disk}$ are given by
		\begin{align}
			\lambda_L(\M^{disk}) &= \sum_{j=0}^{p-1}F_{\M}^{disk}\left (\frac{2j}{p} \pi \right ) e^{i \frac{2j}{p}L \pi} \nonumber \\
			&\approx hp \int_0^{2\pi}J_{2L}(2\kappa h \sin(t))\, dt, \qquad L=1,...,\left\lfloor \frac{p-1}{2} \right\rfloor. \nonumber 
		\end{align}
		\item 	If $p$ is even, there is an additional simple eigenvalue \begin{align}
			\lambda_{p/2}(\M^{disk}) = \sum_{j=0}^{p-1}(-1)^jF_{\M}^{disk}\left (\frac{2j}{p} \pi \right ) \approx hp \int_0^{2\pi} J_{p}(2\kappa h \sin(t))\, dt
	\end{align} \end{enumerate}
\end{theorem}
\begin{proof} We have already shown $\M^{disk}$ is real, symmetric, and circulant. To establish (1), recall from the previous theorem that $\M^{disk}$ is circulant, and hence
	$$ \M^{disk} = \texttt{circ}[{\mathrm m_0,m_1,...,m_{p-1}} ], \qquad \mathrm{m}_j =  2 \pi h J_0\left (2\kappa h \sin \left (\frac{j \pi}{p}\right)\right ) \equiv F_{\M}^{disk}(\xi_j)$$
	where $\xi_j:=\frac{2\pi}{p}j$ and 
	$ F_{\M}^{disk}(x):= 2 \pi h J_0\left(2\kappa h \sin\left( \frac{x}{2} \right ) \right )$ on $[0,2\pi]$ is the generating function of this circulant matrix.

	The eigenvalues of the $[\M^{disk}]$ are therefore given as
	\begin{align}
		\lambda_L(\M^{disk})&:=\sum_{j=0}^{p-1}\mathrm{m}_j {\rm e}^{-i\frac{2j}{p}L\pi} 
		=\sum_{j=0}^{p-1}F_{\M}(\xi_j) 
		{\rm e}^{-iL \xi_j}, \quad L=0,1,...,p-1 	\label{eq:theoremMassDCT}
	\end{align}
	It is easy to see that \begin{align*} \lambda_0(\M^{disk}) &= \sum_{j=0}^{p-1}F_{\M}(\xi_j)\\
		& \approx\frac{p}{2\pi} \int_0^{2\pi}F_{\M}(x)\, dx\\
		&= hp \int_0^{2\pi}J_0\left (2\kappa h \sin \left(\frac{t}{2} \right )\right )\, dt = hp \int_0^{2\pi}J_0(2\kappa h \sin(s))\, ds,\end{align*}
	where the approximation follows via the Trapezoidal rule.
	
	Recall that  $\lambda_s(\M^{disk})=\lambda_{p-s}(\M^{disk}),  s=1,...,\left \lfloor \frac{p-1}{2} \right \rfloor$. It suffices to  consider eigenvalues \\ $\lambda_0(\M^{disk}),...,\lambda_{\left\lfloor\frac{p-1}{2}\right\rfloor}(\M^{disk}),$ and $\lambda_{p/2}(\M^{disk})$ if $p$ is even. 
	
	From \eqref{eq:theoremMassDCT} it is clear that $\lambda_L(\M^{disk})$ are precisely the {\it discrete Fourier coefficients} of the function $F_{\M}^{disk}.$ We now examine the relationship between these eigenvalues $\lambda_j(\M^{disk})$ and the Fourier (cosine) coefficients of $F_{\M}^{disk}(x)$ on $[0,2\pi]$. From the Trapezoidal rule
	\begin{align}
		I_{L,p}:=\sum_{j=0}^{p-1} F_{\M}^{disk}(\xi_j){\rm e}^{-iL \xi_j}&\approx 	\frac{p}{2\pi}\int_0^{2\pi}F_{\M}^{disk}(x) e^{-iLx} dx =:I_{L}, \quad L=0,...,\left \lfloor \frac{p-1}{2} \right \rfloor.\nonumber
	\end{align}	
	Since $\M^{disk}$ is real and symmetric, $\lambda_j(\M^{disk})_j$ are real, and therefore so is the sum on the left.  We see
	\begin{align*} 
		\lambda_L(\M^{disk}) &=I_{L,p}
		\approx  \frac{p}{2\pi}\int_{0}^{2\pi}F_{\M}(x) \cos(Lx) dx = I_L,\qquad  L=0,...,\left \lfloor \frac{p-1}{2} \right \rfloor.
	\end{align*} 
	The quantity $I_L-I_{L,p}$ measures the difference between the true and discrete Fourier coefficients of the (analytic and periodic) function $F_{\M}^{disk}$. Concretely, $I_L -I_{L,p} = \sum_{j=1}^\infty \hat{f}_{L+jp}$ where $\hat{f}_k$ is the $kth$ Fourier coefficient of $F_{\M}^{disk}$.
	
	Using the first integral identity from Lemma \ref{lemma:bessel-coeffs} with $M=0$ and $B= 2\kappa h$, we immediately get the approximation
	
	\begin{align*}
		\lambda_L(\M^{disk}) &\approx 	
		hp\int_0^{2\pi} J_{2L}(2\kappa h\sin(t)) dt, \quad L=1,...,\left \lfloor \frac{p-1}{2} \right \rfloor.
	\end{align*}
	Moreover, if $p$ is even, \begin{align}
		\lambda_{p/2}(\M^{disk}) = \sum_{j=0}^{p-1}(-1)^jF_{\M}^{disk}\left (\frac{2j}{p} \pi \right ) \approx hp \int_0^{2\pi} J_{p}(2\kappa h \sin(t))\, dt
	\end{align} 
	
\end{proof}
Now for a given fixed $t\in [0,2\pi)$, 
\begin{equation}
	\max{|J_{2L}(2\kappa h \sin(t))|} \rightarrow 0
\end{equation}
rapidly as $L$ increases. This follows from the fact that $J_{2L}(2\kappa h \sin(t))$ are precisely the cosine coefficients in the Fourier series expansion of an analytic and periodic function ${\rm e}^{i 2\kappa h \sin(t)\sin(\theta)}$. 

Suppose $\kappa h>0$ is fixed. When the number of plane waves $p \gg 2\kappa h$, we observe using the asymptotic relation (\ref{eq:besselsmallargument}) that the eigenvalues $\lambda_s, 2\kappa h <s\leq \lfloor p/2 \rfloor $ of $\mathtt{M}_D$ will behave as \begin{align}|\lambda_{s}(\M^{disk})| &\approx  \left\vert
	hp\int_0^{2\pi} J_{2s}(2\kappa h\sin(t)) dt \right \vert \nonumber \\
	&\approx\frac{2\pi hp}{\Gamma(2s+1)} (\kappa h)^{2s}.\quad  \label{eq:distrend}
\end{align} 
Clearly, if $\kappa h <1$ this asymptotic behaviour is achieved very  rapidly even for small wave numbers. This behaviour is illustrated in Figure \ref{fig:mass}, and demonstrates both the effectiveness of the asymptotic estimate and the exponential decay of the larger-index eigenvalues. In this figure we also plot the integral expressions from \autoref{thm:mass_spectrum}, which are excellent approximations of the computed eigenvalues.

\begin{figure}
	\centering
	\includegraphics[width=0.7\linewidth]{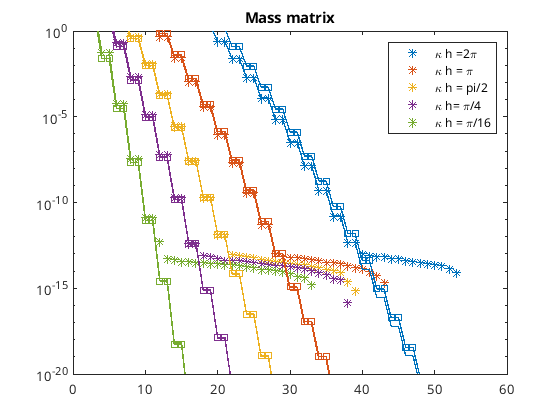}
	\caption{Eigenvalues $\lambda_s(\M^{disk})$ $v/s s, s=0,1,...p-1$  of mass matrix $\M^{disk}$ on the disk (marked by asterix), for different choices of wave number $\kappa$. In each of these experiments, the domain is a disk of radius $h=2\pi$, and the total number of plane waves is $p=61$. The wave number $\kappa$ is chosen to be $\kappa = 1,1/2,...1/32.$. Also plotted: integral approximations from \autoref{thm:mass_spectrum}(line) and asymptotic expression  $\frac{ hp}{\Gamma(2s+1)} (\kappa h)^{2s}$ (line with squares).  }
	\label{fig:mass}
\end{figure}

More detailed asymptotics are possible, but the point is clear:   As soon as $L\gg \kappa h, \lambda_L(\M^{disk})$ becomes small and the conditioning of each of the three matrices will deteriorate.  For {\it any} fixed wave number $\kappa$ and any fixed element size $h$, the eigenvalues of the mass matrices $\M^{disk}$ decay to zero as  $\frac{2\pi hp}{\Gamma(2s+1)} (\kappa h)^{2s}$.  This means that provided the Trapezoidal approximation is of small error,  $\lambda_j(\M^{disk}) \rightarrow 0$ for large enough $j$, i.e., when the number of plane wave $p$ become sufficiently large, and consequently the mass, stiffness and cross matrices become nearly singular. 

Let $p$ be even. Based on the results of the previous theorem, one would expect the spectral condition number of the mass matrix behaves as
\begin{align}
	{\rm Cond} \left( \M^{disk}\right )  &= \left \vert \frac{\lambda_{max}}{\lambda_{min}}\right\vert \approx  \left\vert\frac{\int_0^{2\pi} J_0(2\kappa h \sin(t))\, dt}{\int_0^{2\pi} J_{p}(2\kappa h \sin(t))\, dt}\right\vert  &\approx  C\,\left(\frac{\Gamma(p+1)}{2\pi hp} (\kappa h)^{-p}\right)	
	\label{eq:CondMDisk} 
\end{align} The quality of this approximation is based on how well the eigenvalues $\lambda_j(\M^{disk})$ are approximated by the exact Fourier coefficients of $F_M^{disk}$.
The error in the Trapezoidal rule will depend on the number $p$ of plane waves used. Recall that we make an approximation error 
\begin{equation*}
 \left\vert \frac{p}{2\pi}\int_0^{2\pi} F_{\M}^{disk}(x) \cos(jx) dx - \lambda_j(\M^{disk}) \right\vert =  \left\vert \frac{p}{2\pi}\int_0^{2\pi} F_{\M}^{disk}(x) \cos(jx) dx - \sum_{k=0}^{p-1}F_{\M}^{disk}(\xi_k) 
{\rm e}^{-i j \xi_k} \right\vert.
\end{equation*}
Following Weidemann, or Trefethen (2014), since the function $F_M$ is periodic and analytic on $[0,2\pi]$, the Trapezoidal rule is {\it supergeometric:}
\begin{equation}
	\left\vert \frac{p}{2\pi}\int_0^{2\pi} F_{\M}^{disk}(x) \cos(jx) dx - \sum_{k=0}^{p-1}F_{\M}^{disk}(\xi_k) \exp(i j \xi_k) \right\vert \leq \frac{2\pi \mathcal{A_M}}{e^{ap-1}} \nonumber 
\end{equation}
where $|F_{\M}(z)\cos(jz)|\leq \mathcal{A_M}$ in the strip $-a<Im(z)<a$. 

In Figure \ref{fig:circle-conditioning}, we present the largest and smallest eigenvalues of the mass matrix $\M^{disk}$ on a circle corresponding to different choices of the number of plane waves $p$. We also present the condition number of the mass matrix for these cases. Finally, we consider the minimum DFT of the first row of the mass matrix. Standard theoretical arguments predict the DFT coefficients of the first row of the (circulant) mass matrix are exactly the eigenvalues of the matrix; this is seen in the graphs. We also see, as predicted, that if $\kappa h$ is small, then the conditioning of the mass matrix degenerates even with a small number of plane waves.

\begin{figure}
	\centering
	\includegraphics[width=0.4\linewidth]{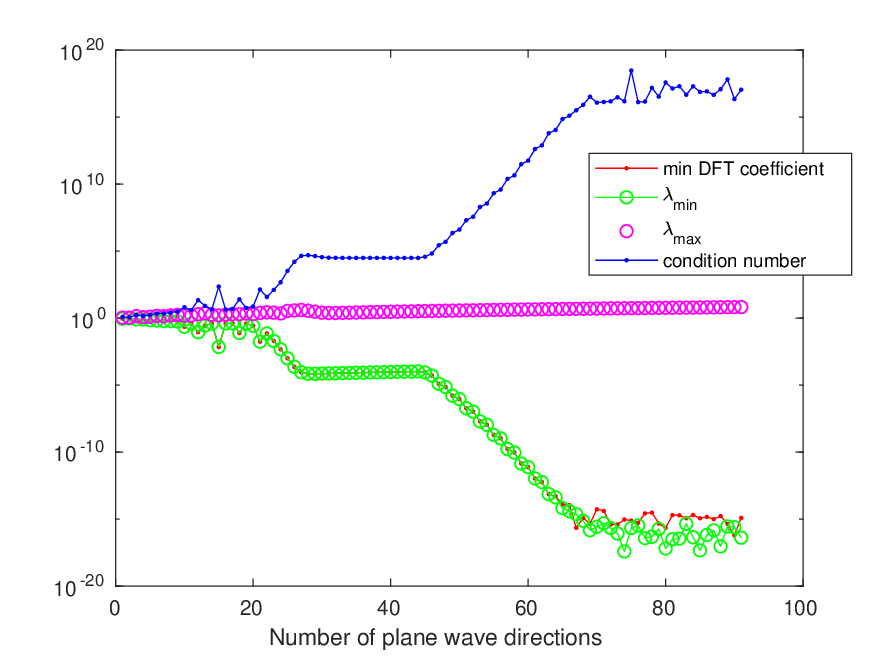}
	\includegraphics[width=0.4\linewidth]{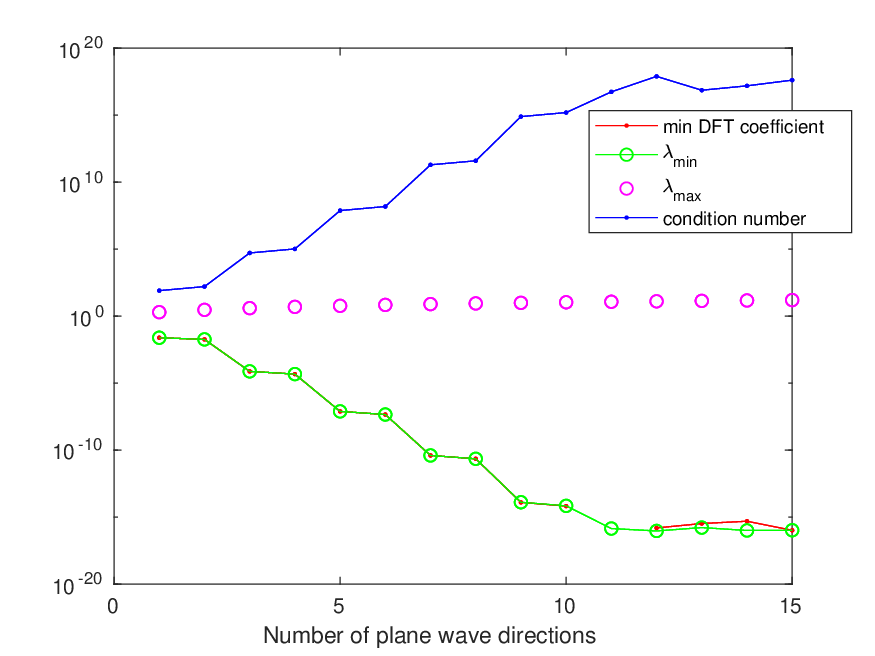}
	
	\caption{The condition number, smallest eigenvalue, largest eigenvalue, and minimum DFT coefficient corresponding to first row of the mass matrix $\M^{disk}$ on a disk. Left: $\kappa h = 10\pi.$ Right: $\kappa h = 0.1\pi$}
	\label{fig:circle-conditioning}
\end{figure}

We next examine the spectrum of the cross matrix using a very similar argument.
\begin{theorem}\label{thm:cross_spectrum}
	Let $K$ be a disk of radius $h$ centered at the origin, and consider the  cross matrix in (\ref{eq:cross})resulting from  $p$ evenly spaced plane waves. Then, the following hold for $L=0,...,p-1$:		 
	\begin{enumerate}
		\item 	The generating function of the cross matrix $\S^{disk}$is   $$F_{\S}^{disk}:=2\pi \kappa h^2 \sin\left(\frac{x}{2} \right)J_1\left (2\kappa h \sin\left( \frac{x}{2}\right ) \right).$$
		\item There is a simple eigenvalue,
		\begin{align*}
			\lambda_0(\S^{disk})&= \sum_{j=0}^{p-1}F_{\S}^{disk}\left (\frac{2j}{p} \pi \right) \approx p \kappa h^2 \int_0^{2\pi} \sin\left (\frac{x}{2} \right)J_1\left (2\kappa h \sin \left(\frac{x}{2} \right) \right)\, dx\nonumber \\ 
		\end{align*}
		\item The eigenvalues $\lambda_L(\S^{disk}),L=1,...,\left\lfloor \frac{p-1}{2} \right\rfloor$  of $\S^{disk}$ have multiplicity 2 and are given by 
		\begin{align} 
			\lambda_L(\S^{disk})&= \sum_{j=0}^{p-1}F_{\S}^{disk}\left (\frac{2j}{p} \pi \right) e^{-i \frac{2j}{p}L \pi}  \nonumber \\ 
			&=(2p\kappa h^2)\int_0^{\pi/2}  \left[J_{2L+1}(\sin(2\kappa h \sin (x))) - J_{2L-1}(2 \kappa h \sin(x)\right]\, dx.
			\nonumber
		\end{align}
		\item If $p$ is even, $\S^{disk}$ has an additional simple eigenvalue,
		\begin{align}
			\lambda_{p/2}(\S^{disk})&= \sum_{j=0}^{p-1}(-1)^jF_{\S}^{disk}\left (\frac{2j}{p} \pi \right) \nonumber\\ &\approx  2p\kappa h^2\int_0^{\pi/2}  \left[J_{p+1}(\sin(2\kappa h \sin (x))) - J_{p-1}(2 \kappa h \sin(x)\right]\, dt.
		\end{align}
		
\end{enumerate}\end{theorem}

\begin{proof}
	
	The proofs for the spectra of $\S^{disk}$ and $\M^{disk}$ are similar, and we omit a few details.
	We see that for all $L=0,...,p-1$ (ignoring multiplicities) the eigenvalues of the cross matrix are
	\begin{align*}
		\lambda_L(\S^{disk}) &= \sum_{j=0}^{p-1} \kappa h^2 a(j) J_1(\kappa h a(j)){\rm e}^{-i\frac{2 j }{p}\pi L}
		\equiv \sum_{j=0}^{p-1} F_{\S}(\xi_j){\rm e}^{-i \xi_j L},
	\end{align*} 
	where we have defined the $2\pi$-periodic function \begin{align}
		F_{\S}(x)&:= \kappa h^2a(x)J_1(\kappa h^2 a(x))= 2\kappa h^2 \sin\left(\frac{x}{2}\right)J_1\left (2\kappa h \sin\left(\frac{x}{2} \right ) \right ). 
	\nonumber
	\end{align}
	Via the Trapezoidal rule, the discrete Fourier coefficient of $F_{\S}$ can be approximated by the true Fourier coefficient:
	\begin{align}
		\lambda_L(\S^{disk}) &\approx  \frac{p}{2\pi} \int_0^{2\pi} F_{\S}(x){\rm e}^{-i Lx}\, dx 
		= p\kappa h^2 \int_0^{2\pi}  \sin\left ( \frac{x}{2} \right )J_1\left (2\kappa h \sin\left(\frac{x}{2} \right )\right) {\rm e}^{i xL}\, dx.\nonumber
	\end{align}
	The eigenvalues of $\S^{disk}$ are real, and both $\sin$ and $J_1$ are odd functions. We obtain
	\begin{align*}
		\lambda_L(\S^{disk}) &\approx  p\kappa h^2 \int_0^{2\pi}  \sin\left(\frac{x}{2}\right )J_1\left (2\kappa h \sin\left (\frac{x}{2} \right ) \right) \cos (xL)\, dx\\
		&=p\kappa h^2 \int_0^{2\pi}  \sin\left(s\right)J_1\left (2\kappa h \sin(s)\right) \cos (2Ls)\, ds .\end{align*}
	Once again, we have eigenvalues $\lambda_L$ of multipliticity 2 for  $L=1,..., \left \lfloor \frac{p-1}{2} \right \rfloor.$ Using the odd behaviour of $J_1$
	\begin{align*}
		\lambda(\S^{disk})_L&\approx   p\kappa h^2\int_0^{2\pi}\sin(s)J_1(2\kappa h \sin(s)) \cos(2Ls)\, ds \\
		&=  \frac{1}{2}p\kappa h^2\int_0^{2\pi}\left[\sin\left((2L+1)s\right)+\sin\left((1-2L)s\right) \right]J_1(2\kappa h \sin(s)) \cos(2Ls)\, ds \\
		&\equiv Q_1+Q_2		
	\end{align*}
	We now use \eqref{BesselIntegral} and \eqref{eq:integralconvertp} to compute, for any integer $A$,
	\begin{align*}
		\int_0^{2\pi}e^{iAs}J_1(2\kappa h \sin(s))\, ds &=  \frac{1}{2\pi}\int_0^{2\pi} \int_0^{2\pi}
		{\rm e}^{iAs + i (2\kappa h \sin(s)\sin(t)) - it} \, ds\, dt\\
		&= \int_{0}^{2\pi}e^{-it} J_{-A}(2\kappa h \sin(t))\, dt\\
	\end{align*}	Therefore
	\begin{align*}
		Q_1+Q_2 &= \frac{1}{2}p \kappa h^2\int_0^{2\pi}\sin(-t) J_{-(2L+1)}(2\kappa h \sin(t))\, dt + \frac{1}{2}p \kappa h^2\int_0^{2\pi}\sin(-t) J_{-(1-2L)}(2\kappa h \sin(t))\, dt\\
		&= 2p\kappa h^2 \int_0^{\pi/2} \sin(t) \left[J_{2L+1}(2\kappa h \sin(t))- J_{2L-1}(2\kappa h \sin(t))\right]\, dt,
	\end{align*}
	which gives (2) and (3) in the statement of the theorem.
\end{proof}

We note that further manipulations using recurrance relations  give
\begin{align*}
	\lambda_L(\S^{disk})	&\approx 2p\kappa h^2 \int_0^{\pi/2} \sin(t) \left[J_{2L+1}(2\kappa h \sin(t))- J_{2L-1}(2\kappa h \sin(t))\right]\, dt,\\
	&= \frac{2p}{\kappa} \int_0^{2\kappa h} \frac{2\kappa h z}{\sqrt{4\kappa^2h^2-z^2}}(2J'_{2L}(z))\, dz\\
	&= -(2ph)\int_0^{2\kappa h} \left(\sqrt{4\kappa^2h^2-z^2}\right) J''_{2L}(z)\, dz.
\end{align*}
We remark that this approximation is less useful for computation, since most softwares use recurrance relations instead.

If we use the asymptotic behaviour for $J_{p}(z)$ for large order, and $0<z\ll p$, we can estimate
\begin{align*} \lambda_{p/2}(\S^{disk}) 
	&\leq p\kappa h^2  [J_{p+1}(2\kappa h) - J_{p-1}(2\kappa h ]\\
	& \sim  p\kappa h^2 \left [\frac{(\kappa h)^{p+1}}{\Gamma(p+2)} - \frac{(\kappa h)^{p-1}}{\Gamma(p)} \right ]
\end{align*} In practice, this asymptotic expression should be used with care in finite-precision computations, since it involves subtraction of nearly equal (and very small) numbers. 

From this, the spectral condition number clearly becomes large with the number of plane waves; the ill-conditioning sets in rapidly if $\kappa h<1$.
\begin{align*}
	{\rm Cond} \left( \S^{disk}\right )  &= \left \vert \frac{\lambda_{max}}{\lambda_{min}}\right\vert \approx  \left\vert\frac{\int_0^{2\pi}\sin(s)J_1(2\kappa h \sin(t))\, dt}{|\lambda_p(\S^{disk})|}\right\vert\\ &\sim C\,\left(\frac{\Gamma(p-1)}{p\kappa h^2} (\kappa h)^{1-p}\right).	
	\label{eq:CondsDisk} 
\end{align*} 

The stiffness matrix spectrum can be characterized in a very similar way.
\begin{theorem}\label{thm:stiffness}
	Let $K$ be a disk of radius $h$ centered at the origin, and consider the stiffness matrices in (\ref{eq:stiff}) resulting from  $p$ evenly spaced plane waves. Then, the following hold for $L=0,...,p-1$:		
	
	\begin{enumerate}
		\item The generating function of the stiffness matrix is $$F_{\D}^{disk}:=\pi h(\kappa h)^2 \left[\cos(x)J_0\left(2\kappa h \sin\left(\frac{x}{2} \right ) \right) + J_2\left(2\kappa h \sin\left(\frac{x}{2}\right ) \right )\right].$$
		\item There is a simple eigenvalue
		$$ 	\lambda_0(\D^{disk})\approx\frac{1}{2}p \kappa^2h^3\int_0^{2pi}\pi \left[J_0\left (2\kappa h \sin\left (\frac{x}{2} \right )\right )\cos(x)  + J_2\left(2\kappa h \sin\left(\frac{x}{2}\right)\right)\right]\, dx. \\ \nonumber$$
		\item The multiplicity-2 eigenvalues of the stiffness matrix $\D^{disk}$ are given by 
		\begin{align}
			\lambda_L(\D^{disk})&= \sum_{j=0}^{p-1}F_{\D}^{disk}\left (\frac{2j}{p} \pi \right ) {\rm e}^{-i \frac{2j}{p}L \pi}  \qquad L=1,...,\left\lfloor \frac{p-1}{2} \right\rfloor \nonumber\\
			& \approx  \frac{1}{4}p \kappa^2 h^3 \int_0^{2\pi}\left[J_{2L+2}(2\kappa h\sin(t)) + J_{2L-2} (2\kappa h\sin(t))\right]+2J_{2L}(2\kappa h\sin(t)) \cos(2t)\, dt. \nonumber
		\end{align}
		\item If $p$ is even then there is an additional simple eigenvalue
		\begin{align*}
			\lambda_{p/2}(\D^{disk}) = \sum_{j=0}^{p-1} (-1)^jF_{\D}^{disk}\left (\frac{2j}{p}\pi \right ). 
		\end{align*}		
	\end{enumerate}
\end{theorem}
\begin{proof}
	It immediately follows that
	$$\lambda_0(\D^{disk}) \approx \frac{1}{2}p\kappa^2 h^3\int_0^{2\pi}  \left[\cos(x)J_0\left (2\kappa h \sin \left(\frac{x}{2} \right) \right) + J_2\left(2\kappa h \sin \left(\frac{x}{2} \right ) \right )\right]\, dx, \nonumber. $$
	Next, there are real eigenvalues $\lambda_L(\D^{disk}), \, L= 1,...,\left \lfloor \frac{p-1}{2} \right \rfloor$ with multiplicity 2. 
	\begin{align}
		\lambda_L(\D^{disk})
		&\approx \frac{1}{2}p\kappa^2 h^3\int_0^{2\pi}\pi  \left[\cos(x)J_0\left (2\kappa h \sin\left(\frac{x}{2} \right)\right) + J_2\left (2\kappa h \sin\left(\frac{x}{2}\right )\right)\right]{\rm e}^{iLx}\, dx, \nonumber \\
		&= \frac{1}{2}p\kappa^2 h^3 \int_0^{2\pi}J_{2L}(2\kappa h\sin(t)) {\rm e}^{-i2t}\, dt + p\kappa^2 h^3 \int_0^{2\pi}J_0(2\kappa h \sin(s))\cos(2s)\cos(2Ls)\, ds\nonumber 
	\end{align}
	
	We can rewrite the second integral using the sum of cosines and \eqref{eq:integralconvertp}
	\begin{align}
		\frac{1}{2}	p\kappa^2 h^3 \int_0^{2\pi}J_0(2\kappa h \sin(s))\cos(2s)\cos(2Ls)\, ds  \hspace{-1.4in} && \nonumber \\
		&=\frac{1}{4} p \kappa^2 h^3 \int_0^{2\pi}J_{2L+2}(2\kappa h\sin(t)) + J_{2L-2} (2\kappa h\sin(t)) \, \, dt, \nonumber 
	\end{align} leading to
	\begin{align*}
		\lambda_L(\S^{disk})	\approx 	\frac{1}{4}p\kappa^2 h^3 \int_0^{2\pi} 2J_{2L}(2\kappa h\sin(t))\cos(2t) +J_{2L+2}(2\kappa h\sin(t)) + J_{2L-2} (2\kappa h\sin(t))  \, ds, 
	\end{align*} for $\,\, L=1,...,\left \lfloor \frac{p-1}{2} \right \rfloor.$
	Finally if $p$ is even
	\begin{align*}
		\lambda_{p/2}(\S^{disk}) \approx \frac{1}{4} p\kappa^2 h^3 \int_0^{2\pi} 2J_{p}(2\kappa h\sin(t))\cos(2t)+ [J_{p+2}(2\kappa \sin(t))+ J_{p-2}(2\kappa h\sin(t))]\, dt,
	\end{align*}
	proving the theorem.
\end{proof}
One could use recurrence relations to rewrite the expressions in terms of derivatives of Bessel functions. 

As the number of plane waves $p$  grows, the spectral condition number of the stiffness matrix will grow as
$$ {\tt Cond}(\D^{disk}) =  C(\lambda_{p/2})^{-1} \approx C (p+2)! (\kappa h)^{-2-p}.$$

%

\section{PWB on  polygons}
The mass and systems matrices on the disk are circulant, and therefore their properties could be described analytically.   We next study the behaviour of the mass matrix on  cyclic polygons, and notice that they are close to Toeplitz. This suggests the use of circulant preconditioners. In this section we identify two simple-to-construct circulant matrices which are close to the mass matrix on a polygon.

\subsection{Cyclic polygons}
Consider the situation of a single element $K$ consisting of a {\it cyclic} polygon with $L$ sides. The vertices of an $L$-sided cyclic polygon  are given by 
\begin{equation}
	\x_j=h(\cos(\theta_j),\sin(\theta_j))
\end{equation}
where $0 \leq \theta_0 < \theta_1 < \cdots < \theta_{L-1}< 2 \pi.$  Note if the angles $\theta_i$ are equally spaced then the polygon is regular.

In what follows we assume the plane wave directions are evenly spaced as in (\ref{eq:directions}). The entries of the mass matrix will be given as $$\,[\\M^{reg}]_{m,\ell}:=\int_{\partial K} \varphi_m\overline{\varphi_{\ell}} ds.$$ The conditioning of this matrix deteriorates both as the number of plane wave directions increases, and if $\kappa h$ becomes small. 

We expect the mass matrix on the disk and  the mass matrix on a $L-sided$ polygon to be close for large $L$, for a fixed number of PWB $p$ and a fixed wavenumber $\kappa$. This is illustrated in \autoref{fig:fig4} 
The explicit expression derived in \cref{thm:mass_spectrum} for the eigenvalues of the mass matrix on the disk initially provides an excellent estimate (see \cref{fig:circle-conditioning}), and the conditioning of $\M^{disk}$ is governed by the smallest eigenvalue $\lambda_{p/2}(\M^{disk})$. The comparison breaks down when the computed condition number of the matrices on the polygon exceeds $10^{20}$. This suggests that we may use the (computable) quantities in \eqref{eq:CondMDisk} as a heuristic guide to decide how many plane wave directions to use for a given problem and mesh size.

\begin{figure}
	\centering
	$\begin{array}{cc}
	\includegraphics[width=0.47\linewidth]{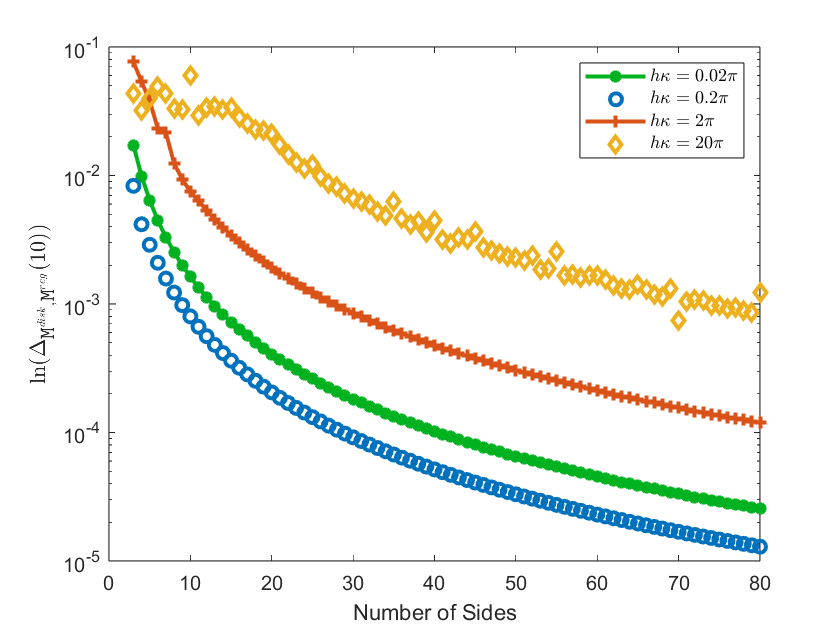} &
	\includegraphics[width=0.47\linewidth]{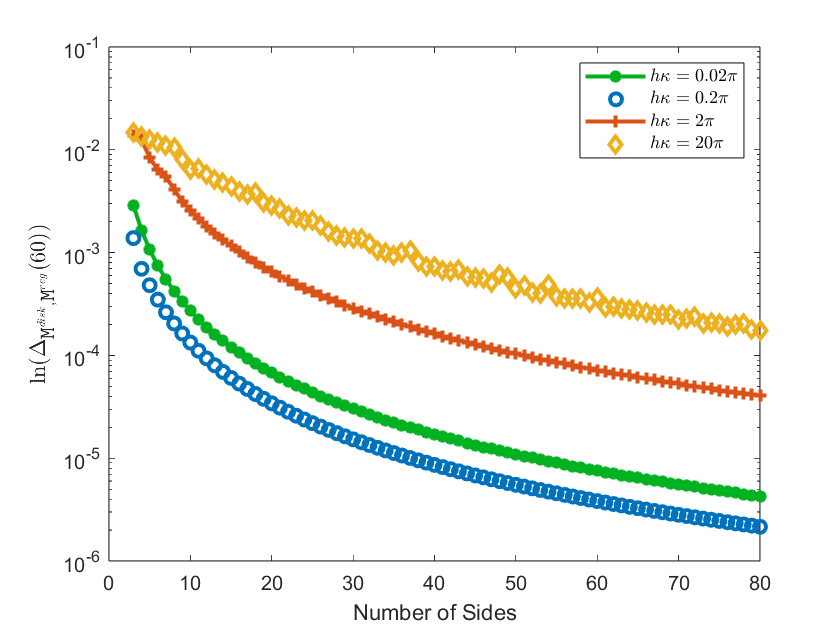} \nonumber \\
	(a) \, \, \,  p = 10 &  (b) \, \, \, p = 60 \nonumber
	\end{array}$
	\caption{Comparison of $\M^{disk}$ and $\M^{reg}$ on a $L$-sided regular polygon as $L$ increases.  Shown is the natural log of $\Delta_{\texttt{M}^{disk},\texttt{M}^{reg}}(L)$  for $p = 10$ (a) and $p = 60$ (b) plane waves. In each case, results are shown for $\kappa h = 0.02 \pi, 0.2 \pi, 2 \pi,$ and $20 \pi.$ }
	\label{fig:fig4}
\end{figure}


In order to facilitate the proceeding analysis, we parametrize the edge, $\Gamma_j,$ between subsequent 
vertices, $x_j$ and $x_{j+1},$ in terms of the angle  parameter $t$. That is, we write
\begin{equation}
	\x(t) = r(t) \langle \cos(t), \sin(t)\rangle, \qquad  \theta_j \leq t \leq \theta_{j+1}
\end{equation}
where
\begin{equation}
	r(t) = h \cos \left (\frac{\theta_{j+1}-\theta_j}{2} \right )\sec \left (\frac{\theta_{j+1}+\theta_j}{2}-t \right ).
\end{equation} 
The element of arclength on the edge between $x_j$ and $x_{j+1}$ is given by 
\begin{eqnarray}
	g_j(t) &=& \left | \frac{d}{dt}\x(t) \right |=  \left | \langle r^{\prime}(t) \cos(t) - r(t)\sin(t), r^{\prime}(t) \sin(t) + r(t) \cos(t) \rangle \right | \nonumber \\
	&=& h \left |\cos\left (\frac{\theta_{j+1}-\theta_j}{2} \right ) \right | \sec^2 \left (\frac{\theta_{j+1}+\theta_j}{2}-t \right ). \label{eq:g}
\end{eqnarray}
Consequently, 
\begin{equation*}
	\int_{\gamma_j} \varphi_m \overline{\varphi_{\ell}}ds = \int_{\theta_j}^{\theta_{j+1}} e^{i \kappa |r(t)| a(m-\ell )b(m+\ell,t) } g_j(t)ds 
\end{equation*}
and adding contributions over all edges yields
\begin{equation}\label{eq:regular_poly_M}
	[\M^{cyclic}]_{m\ell} =  \sum_{j=0}^{L-1} \int_{\theta_j}^{\theta_{j+1}} e^{i \kappa |r(t)| a(m-\ell )b(m+\ell,t) } g_j(t)dt.
\end{equation}
Note that $\M^{cyclic}$ is Hermitian. 
Unlike the case when $K$ is a disk, the mass matrix on a cyclic polygon is not circulant.
\begin{lemma}
	The mass matrix on the cyclic polygon is {\it not} Toeplitz, and hence not circulant. We show the details for the case of a regular polygon, and the situation on a general cyclic polygon follows a similar computation. \end{lemma}
\begin{proof} We easily see $\M^{reg}_{m+1,\ell+1}\not= \M^{reg}_{m,\ell}$:
	\begin{eqnarray*}
		[\M^{reg}]_{m+1,\ell+1} &=&  \sum_{j=0}^{L-1} \int_{\theta_j}^{\theta_{j+1}} {\rm e}^{i \kappa |r(t)| a(m-\ell )b(m+\ell+2,t) } g_j(t)dt\\
		&=&\sum_{j=0}^{L-1} \int_{\theta_j}^{\theta_{j+1}} {\rm e}^{i \kappa |r(t)| a(m-\ell )\sin(t - \frac{m+\ell}{p}\pi +\frac{2\pi}{p}) } g_j(t)dt\\
		&=&\sum_{j=0}^{L-1} \int_{\theta_j+\frac{2\pi}{p}}^{\theta_{j+1}+\frac{2\pi}{p}} {\rm e}^{i \kappa |r(s+\frac{2\pi}{p})| a(m-\ell )\sin( s- \frac{m+\ell}{p}\pi ) } g_j\left (s+\frac{2\pi}{p} \right ) \,ds\\
		&\not=& [\M^{reg}]_{m,\ell}.
	\end{eqnarray*}
This shows that the diagonals of $\M^{reg}$ are not constant.\end{proof}
Nonetheless, we observe that by letting $p\rightarrow + \infty$, the mass matrix on the polygon becomes Toeplitz. 

We aim to quantify how 'far' from Toeplitz the mass matrix $\M^{reg}$ is.To this end, we use the average of the diagonals of $\M^{reg}$ to  construct a Toeplitz matrix $T^{reg}$. That is, define
\begin{equation*}
	T^{reg}_{ij} := \overline{m}_{i-j}, \quad \text{where} \quad 	\overline{m}_k := \frac{1}{p - |k|} \sum_{i-j = k} \M^{reg}_{i,j}, \quad \text{for}\,  k = 1-p,...,p-1.
\end{equation*}
A measure of how close to Toeplitz the mass matrix $M^{reg}$ is depicted in Figure \ref{Fig:departurefromToeplitz}. Here we use the relative measure
\begin{equation}\label{eq:deltatoeplitz}
	\Delta_{A,B}(p):=\|A - B \|/(p\|B\|).
\end{equation}

As we see, as the number of plane waves increases, $\Delta_{\M^{reg},\T^{reg}}(p)$ decreases. As the frequency $\kappa$   increases, so does the departure from Toeplitz.

\begin{figure}
	\centering
	$\begin{array}{ccc}
		\includegraphics[width=0.3\linewidth]{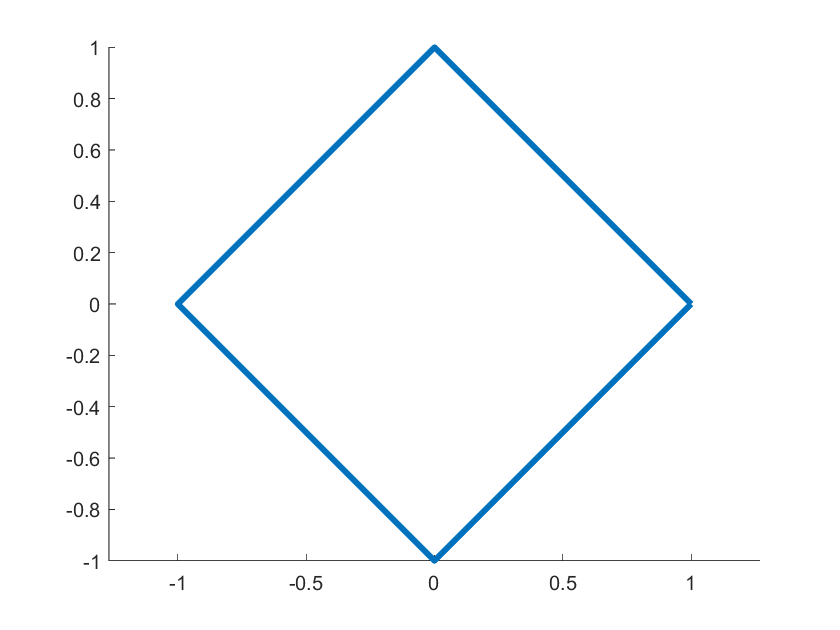} & 	\includegraphics[width=0.3\linewidth]{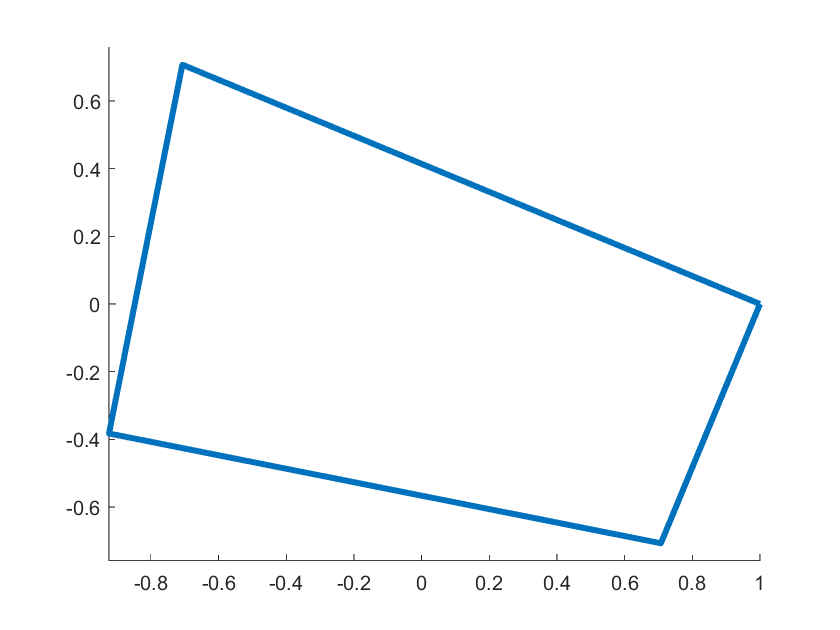}&
		\includegraphics[width=0.3\linewidth]{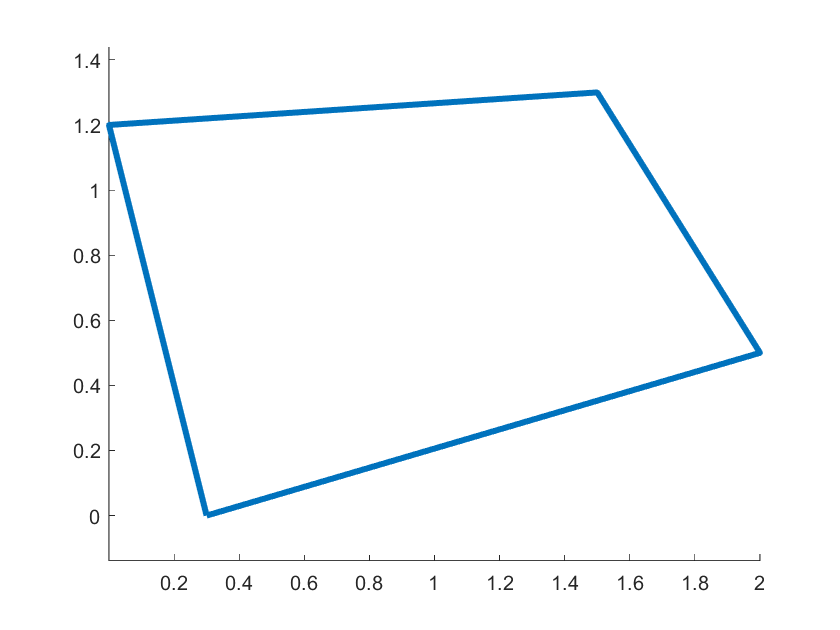}
		 \nonumber \\	
			(a) \, \, \,  \mbox{Regular Polygon} &  (b) \, \, \, \mbox{Cyclic Polygon} & \, \, \, (c) \, \, \, \mbox{Polygon General}\nonumber \\
		\includegraphics[width=0.3\linewidth]{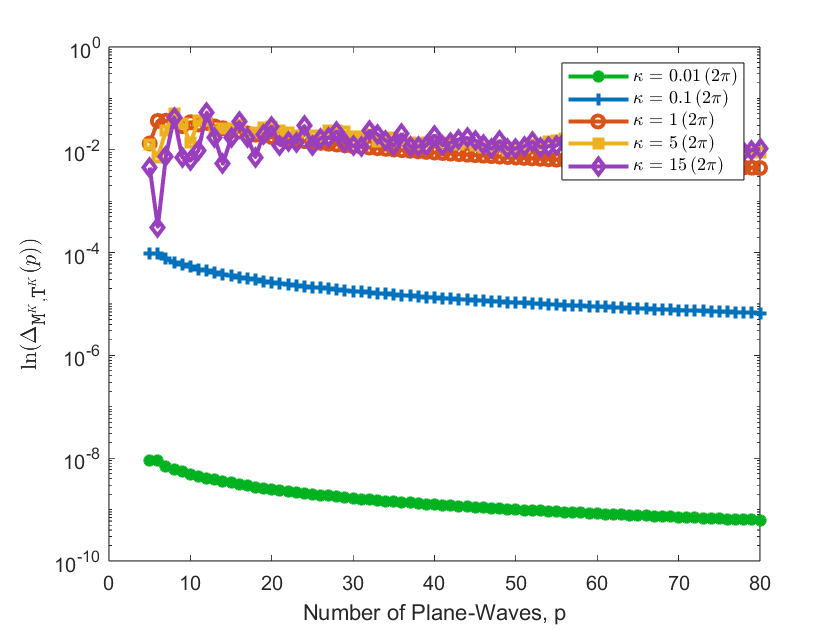}&
		\includegraphics[width=0.3\linewidth]{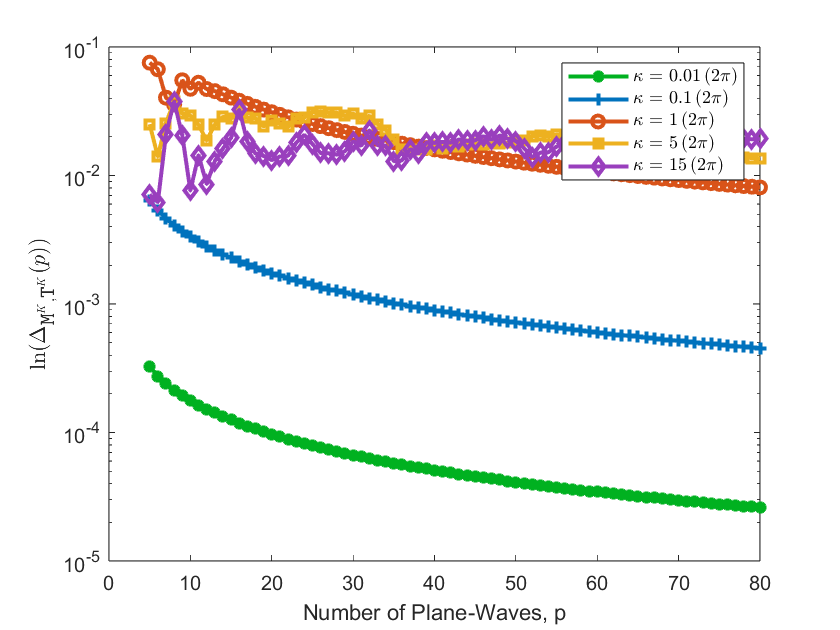} & 
		\includegraphics[width=0.3\linewidth]{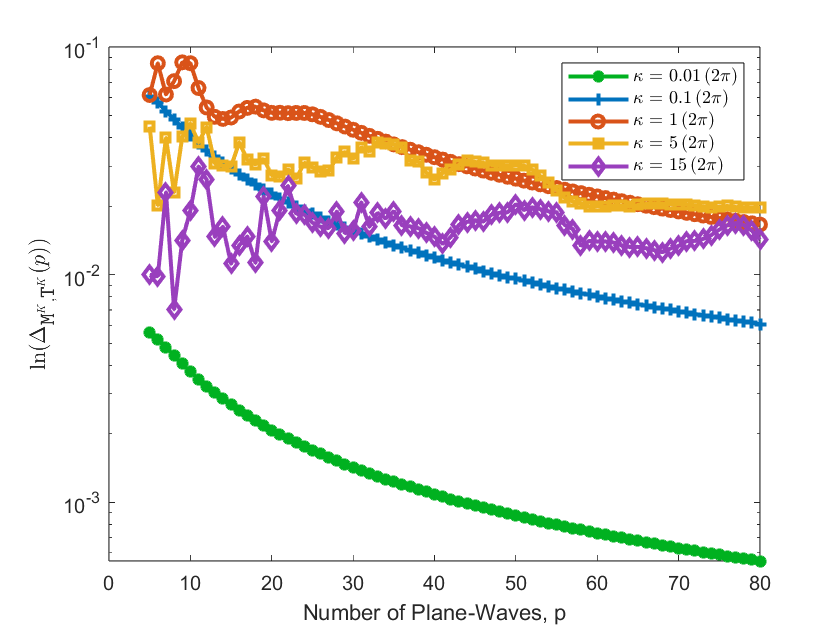} \nonumber \\	
			(d)  &  (e) & (f) \nonumber \\	
	
	\end{array}$\label{Fig:departurefromToeplitz}
	\caption{The mass matrices $M^{K}$ on polygonal elements $K$ related to $\T^{K}.$  The three polygons we consider are regular $\left \{\theta = m\frac{\pi}{2}, m = 0,1,2,3; \, h = 1 \right \} $ in (a), cyclic $\left (\theta = 0,\frac{3 \pi}{4},\frac{9 \pi}{8},\frac{7 \pi}{4}; h = 1 \right )$ in (b), and the general polygon $\left \{ \mbox{vertices at } (0.3,0), (2,0.5), (1.5,1.3), \mbox{ and } (0,1.2) \right \}$ in (c). The plots of   $\ln \left ( \Delta_{\M^{K}, \T^{K}}(p) \right )$ as a function of the number of plane-waves, $p,$ for $\kappa = 0.01 (2\pi), 0.1 (2\pi), 1 (2\pi), 5 (2 \pi),$ and $15 (2\pi)$ below each associated shape in (d), (e), and (f), respectively.  }
\end{figure} 

We see from these experiments that the mass matrix is {\it close} to Toeplitz. This motivates the question: can we use a circulant preconditioner for $\M^{reg}$, using the knowledge that such preconditioners are effective on Toeplitz matrices\cite{DIBENEDETTO199335,Gray,KARNER2003301,chan}? If so, how do we design such a preconditioner?

\subsection{Preconditioning on a polygonal element}
Circulant preconditioners are attractive for their ease of application. Ideally, one seeks a circular matrix $\C$ so that $\C^{-1} \SyS$ is nearly the identity matrix; the inverse $\C^{-1} = \U^*\D^{-1}\U$ where $\U$ is unitary and $\D$ has the (easily computable) eigenvalues of $\C$. 

Motivated by these ideas, we seek simple choices of circulant preconditioners which have desirable properties. One candidate choice may be constructed as follows: use the first row of the mass matrix on a polygon $\M^{K}$ to generate a circulant matrix, $\C_R^{K}$. That is, define
\begin{equation}
	\C_R^{K} = \text{circ}[
		M_{11},M_{12},M_{13},...,M_{1,p}]	\label{eq:rcirc}
\end{equation}
This matrix is circulant, and has the eigenvalue decomposition $\C_R^{K}:=\U\Lambda(\C^{K}_{R}) \U^*$ where $\Lambda(\C^K_R)$ is a diagonal matrix containing the eigenvalues of $\C_R^{K}$. However, $\C_R^{K}$ is {\it not} Hermitian. This can be seen by considering $K$ to be a regular polygon:
$$ [\M^{reg}]_{12}=  \sum_{j=0}^{L-1} \int_{\theta_j}^{\theta_{j+1}} e^{i \kappa |r(t)| a(1 )b(3,t) } g_j(t)dt \not=\overline{[\M^{reg}]_{1,p}},$$
The eigenvalues of $\C^K_R$ may be therefore be complex. This need not be problematic, particularly if $\C_R^{K}$ is better conditioned than $\M^{K}$; since it is easy to invert, the application of $\C^{K}_R$ can be effective in reducing the error in an iterative solution.

	We define the so-called {\it best circulant matrix} $\C_B^{K}$ given by 
	\begin{equation}
		[\C^K_B]_{i, \,j} = \frac{1}{p} \sum_{m-n = i-j \, \, mod(p)} [\M^K]_{m n} \label{eq:bcirc}
	\end{equation}
	where recall that $p$ is the number of planewaves.
	
	\begin{figure}
		\centering
		$\begin{array}{ccc}
			\includegraphics[width=0.3\linewidth]{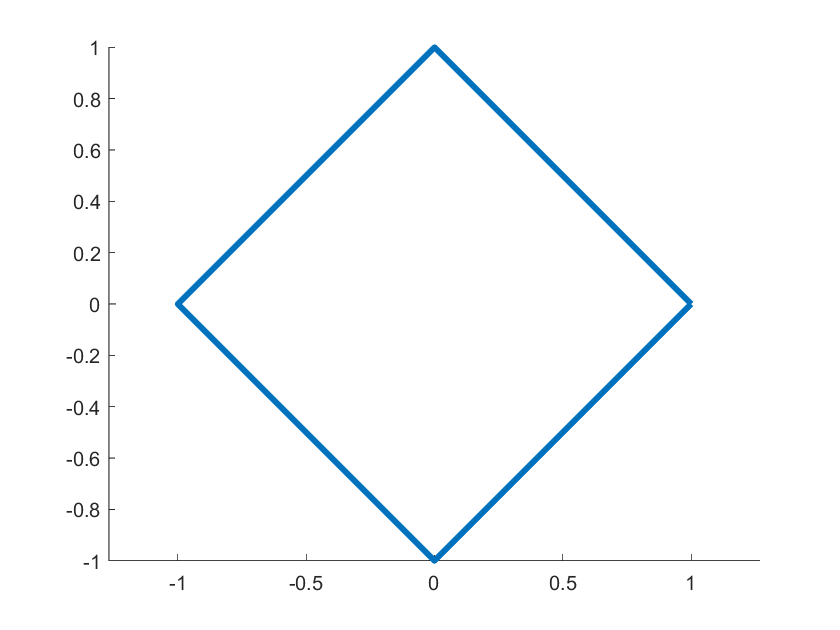} & 	\includegraphics[width=0.3\linewidth]{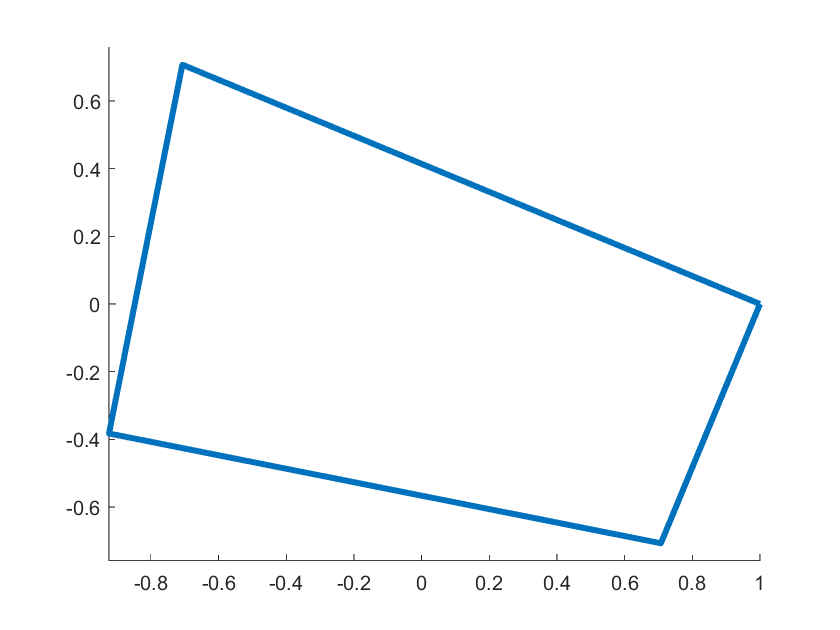}&
			\includegraphics[width=0.3\linewidth]{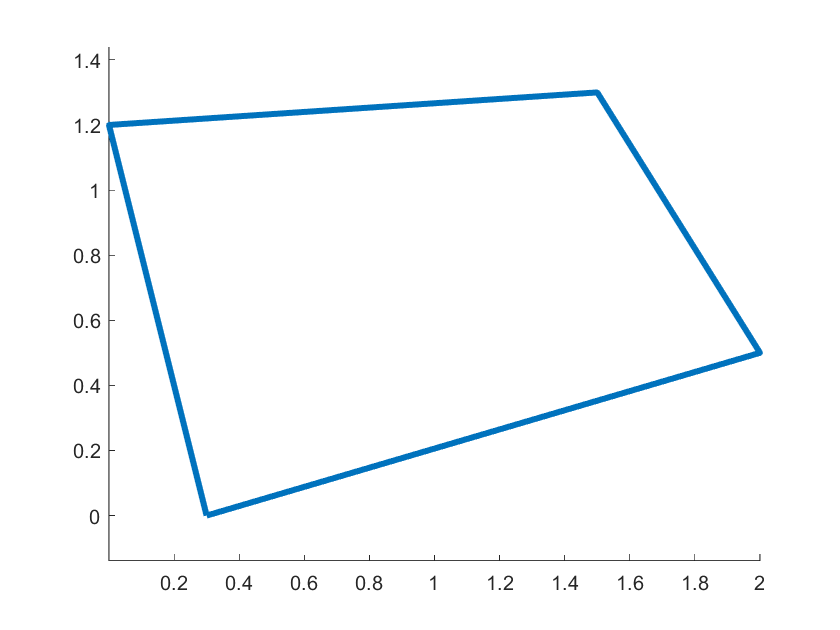} \nonumber \\
			(a) \, \, \,  \mbox{Regular Polygon} &  (b) \, \, \, \mbox{Cyclic Polygon} & \, \, \, (c) \, \, \, \mbox{Polygon General}\nonumber \\
			\includegraphics[width=0.35\linewidth]{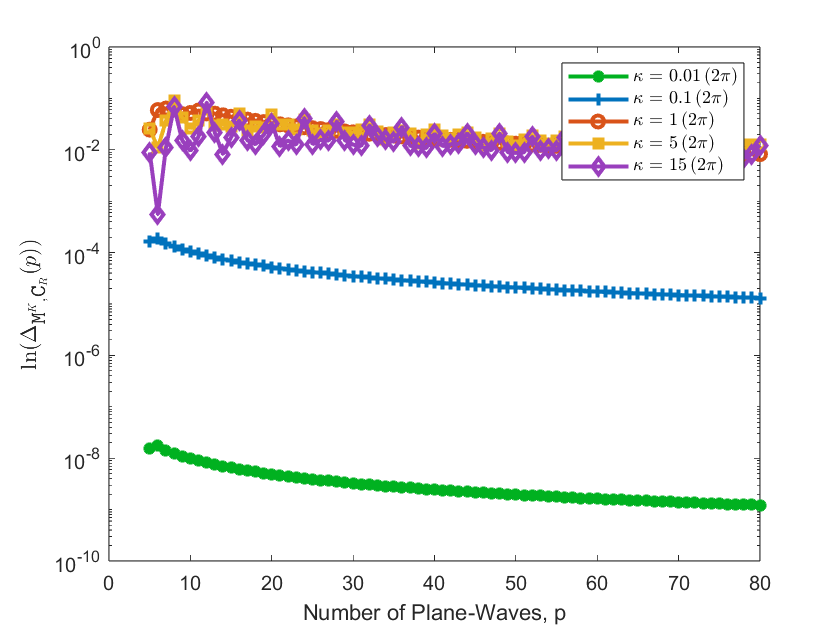}&
			\includegraphics[width=0.35\linewidth]{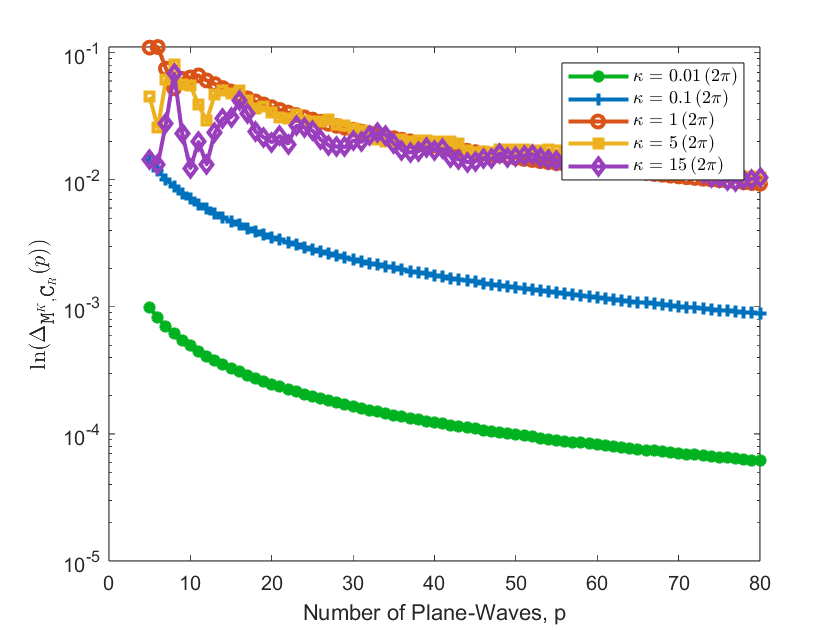}&
			\includegraphics[width=0.35\linewidth]{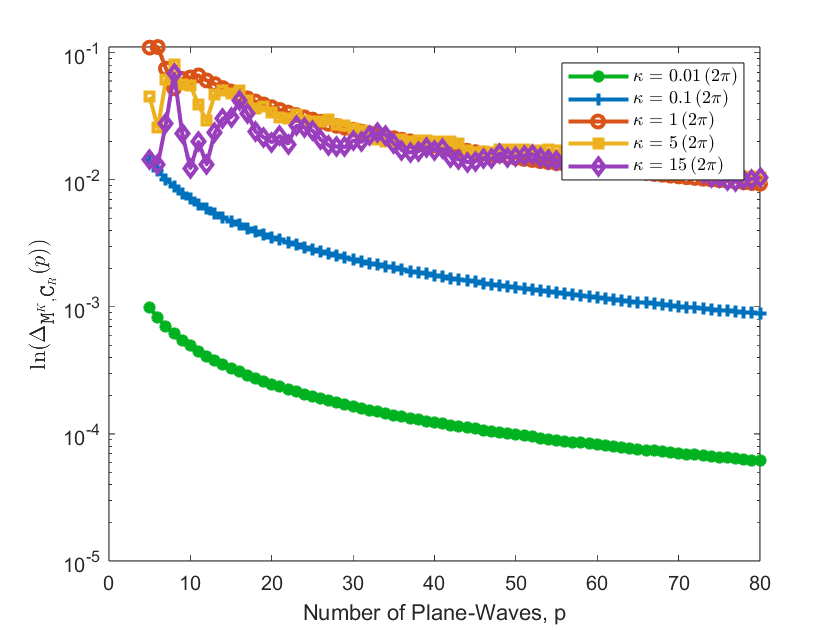} \nonumber \\
			(d) & (e) & (f) \nonumber \\
			\includegraphics[width=0.35\linewidth]{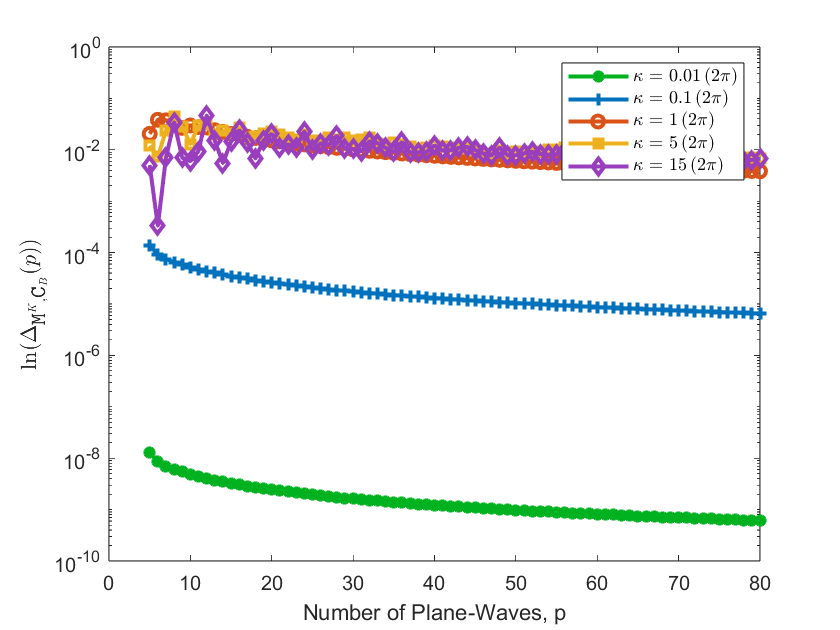}&	\includegraphics[width=0.35\linewidth]{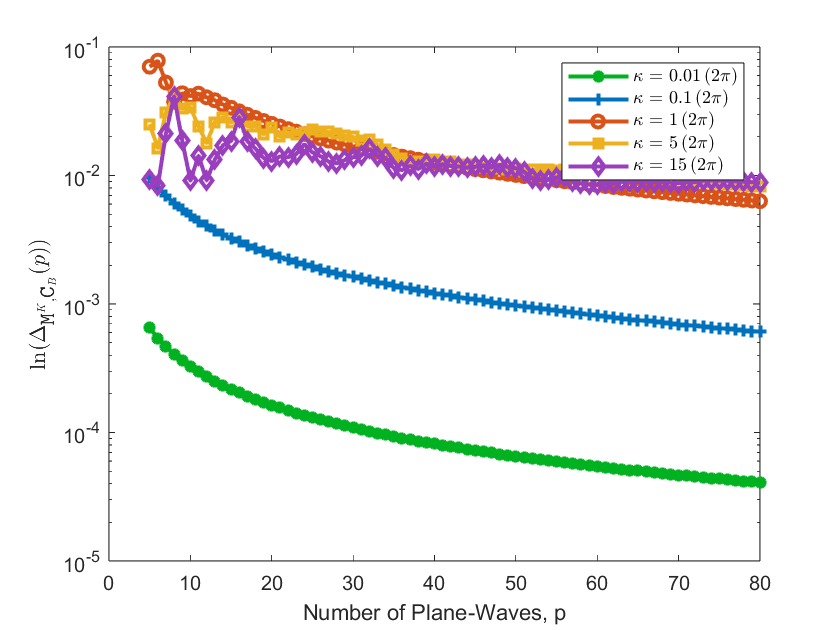} &
			\includegraphics[width=0.35\linewidth]{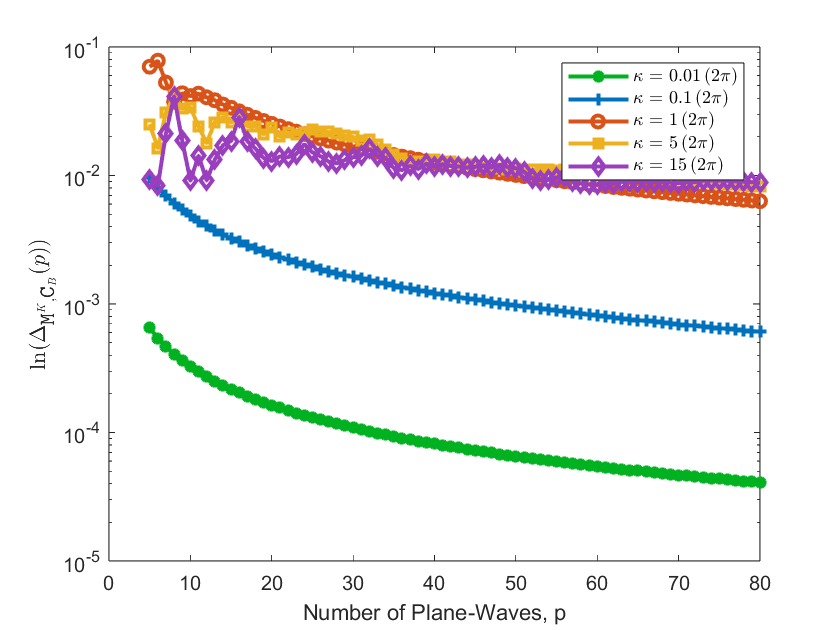} \nonumber \\
			(g) & (h) & (i) \nonumber
		\end{array}$\label{Fig:departurefromCirculant}
		\caption{The mass matrices $M^{K}$ on polygonal elements $K$ related to $\C^{K}_R$ and $\C^{K}_B.$  The three polygons we consider are regular $\left \{\theta = m\frac{\pi}{2}, m = 0,1,2,3; \, h = 1 \right \} $ in (a), cyclic $\left (\theta = 0,\frac{3 \pi}{4},\frac{9 \pi}{8},\frac{7 \pi}{4}; h = 1 \right )$ in (b), and the general polygon $\left \{ \mbox{vertices at } (0.3,0), (2,0.5), (1.5,1.3), \mbox{ and } (0,1.2) \right \}$ in (c). The plots of   $\ln \left ( \Delta_{\M^{K}, \C^{K}_R}(p) \right )$ and $\ln \left ( \Delta_{\M^{K}, \C^{K}_B}(p) \right )$as a functions of the number of plane-waves, $p,$ for $\kappa = 0.01 (2\pi), 0.1 (2\pi), 1 (2\pi), 5 (2 \pi),$ and $15 (2\pi)$ below each associated shape in (d), (e), (f), and (g),(h),(i),  respectively.   }
	\end{figure}

	\subsection{A closer look at $\C_B^{cylic}$}
	The behaviour of the best circulant matrix $\C_B^K$ is easier to analyze in the case $K=$ cyclic polygon (in this scase denoted $\C_B^{cyclic}$).
	The diagonal entries of this matrix are given by \eqref{eq:bcirc}:
	\begin{equation}
		[\C^{cyclic}_B]_{i, \,i}= \frac{1}{p} \sum_{s = 1}^p M_{s, \, s} \mbox{ for } i = 1,2...p-1.
	\end{equation} The off-diagonal terms can be written as 
	\begin{equation}
		[\C^{cyclic}_B]_{i, \,j} = \frac{1}{p} \left ( \sum_{s = 1}^{p-(i-j)} M_{s+(i-j), \, s} +  \sum_{s =p -(i-j) + 1}^{p} M_{s-p+(i-j), \, s} \right ) \mbox{ for } i > j, \label{eqc1}
	\end{equation}
	and
	\begin{equation}
		[\C^{cyclic}_B]_{i, \,j} = \frac{1}{p} \left ( \sum_{s = 1}^{P+(i-j)} M_{s, \, s-(i-j)} +  \sum_{s = P +(i-j) + 1}^{p} M_{s, \, s-P-(i-j)} \right ) \mbox{ for } i < j, \label{eqc2}
	\end{equation}
	When $i > j,$ in (\ref{eqc1})
	\begin{eqnarray}
		M_{s+(i-j), \, s} &=&  \sum_{\ell=0}^{L-1}\int_{\theta_\ell}^{\theta_{\ell+1}} e^{i \kappa |r(t)| a(i-j )b(2s + (i-j),t) } g_\ell(t)dt \nonumber
	\end{eqnarray}
	and
	\begin{eqnarray}
		M_{s-P+(i-j), \, s} &=&  \sum_{\ell=0}^{L-1}\int_{\theta_\ell}^{\theta_{\ell+1}} e^{i \kappa |r(t)| a(-P+(i-j) )b(2s - P + (i-j),t) } g_\ell(t)dt \nonumber \\
		&=&  \sum_{\ell=0}^{L-1}\int_{\theta_\ell}^{\theta_{\ell+1}} e^{i \kappa |r(t)| a((i-j) )b(2s + (i-j),t) } g_\ell(t)dt \nonumber 
	\end{eqnarray}
	
	Consequently, when $i > j,$
	\begin{equation}
		[\C^{cyclic}_B]_{i, \,j} = \frac{1}{p} \left ( \sum_{s = 1}^{p} \sum_{\ell=0}^{L-1}\int_{\theta_\ell}^{\theta_{\ell+1}} e^{i \kappa |r(t)| a((i-j) )b(2s  + (i-j),t) } g_\ell(t)dt \right ). \label{eqc3}
	\end{equation}
	Similarly, when $i < j,$ a similar calculation leads to 
	
	\begin{equation}
		[\C^{cyclic}_B]_{i, \,j} = \frac{1}{p} \left ( \sum_{s = 1}^{p} \sum_{\ell=0}^{L-1}\int_{\theta_\ell}^{\theta_{\ell+1}} e^{i \kappa |r(t)| a(i-j )b(2s - (i-j),t) } g_\ell(t)dt \right ).\label{eqc4}
	\end{equation}	
	We now observe
	\begin{equation}
		b(2s + (i-j),t) = \sin \left (t -|i-j|\frac{\pi}{p} - \frac{2s \pi}{p} \right ). \label{eqsum1}
	\end{equation}
	We can then write (\ref{eqc3}) and (\ref{eqc4}) as
	\begin{equation}
		[\C^{cyclic}_B]_{i, \,j} =   \sum_{\ell=0}^{L-1}\int_{\theta_\ell}^{\theta_{\ell+1}} \left (\frac{1}{p}\sum_{s = 1}^{p} e^{i \kappa |r(t)| a((i-j) )b(2s  + |i-j|,t) } \right ) g_\ell(t)dt. \label{eqc5}
	\end{equation}
Recalling that $a(\ell)$ is an odd function, it is obvious that $\C^{cyclic}_B$ is Hermitian. Moreover in light of (\ref{eqsum1}), 
	\begin{eqnarray}
		\frac{1}{p}\sum_{s = 1}^{p} e^{i \kappa |r(t)| a((i-j) )b(2s  + |i-j|,t) }
		&=& \frac{1}{p}\sum_{s = 1}^{p} e^{i \kappa |r(t)| a(i-j )\sin \left (t - (i-j)\frac{\pi}{p} - \frac{2s\pi}{p} \right ) } \nonumber \\
		& \approx & \frac{1}{2\pi} \int_0^{2 \pi} e^{i \kappa |r(t)| a(i-j )\sin \left (t - |i-j|)\frac{\pi}{p} - \phi \right ) } d \phi\nonumber \\
		& = & \frac{1}{2\pi} \int_{t - |i-j|\frac{\pi}{p} }^{t - |i-j|\frac{\pi}{p} +2 \pi} e^{i \kappa |r(t)| a(i-j )\sin \left (\phi \right ) } d \phi \nonumber \\
		&=& J_0\left ( \kappa |r(t)| a(i-j)\right ).
	\end{eqnarray}
	Finally,
	\begin{equation}
		[\C^{cyclic}_B]_{i, \,j} \approx   \sum_{\ell=0}^{L-1}\int_{\theta_\ell}^{\theta_{\ell+1
		}}  (J_0\left (  \kappa |r(t)| a(i-j)\right ) g_\ell(t)dt \label{eqc6}
	\end{equation}
	as the number of PWB $p$ increases. It is also clear that as the number of sides $L$ increases, $[\C^{cyclic}_B]_{i, \,j} \rightarrow \M_{i,\,j}^{disk}$.

		We compare the spectra of $\C_R^{reg}, \C^{reg}_B$ and $\M^{reg}$ on a regular triangle, a hexagon and a 20-gon in Figure \ref{fig:regular-20}.\begin{figure}
		\centering
		$\begin{array}{ccc}
		\includegraphics[width=0.3\linewidth]{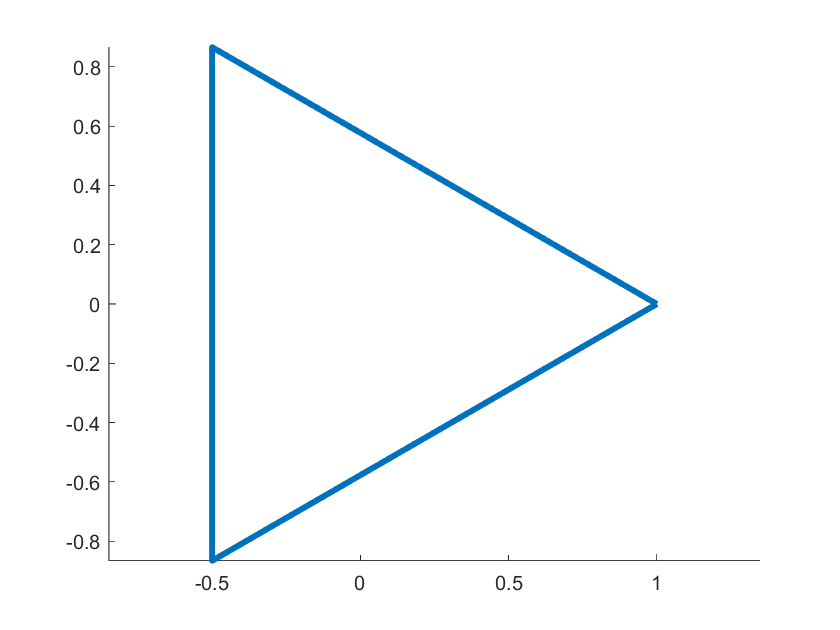}&
		\includegraphics[width=0.3\linewidth]{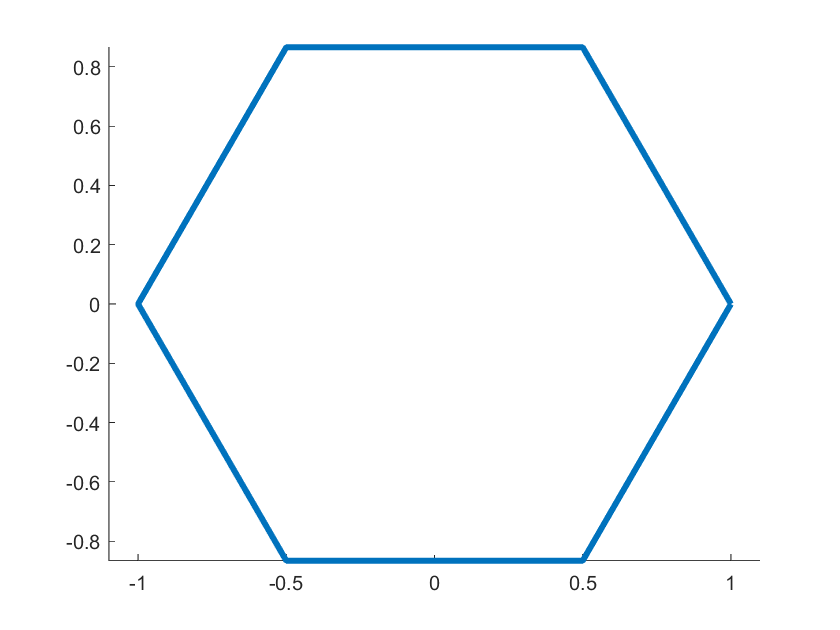}&	
		\includegraphics[width=0.3\linewidth]{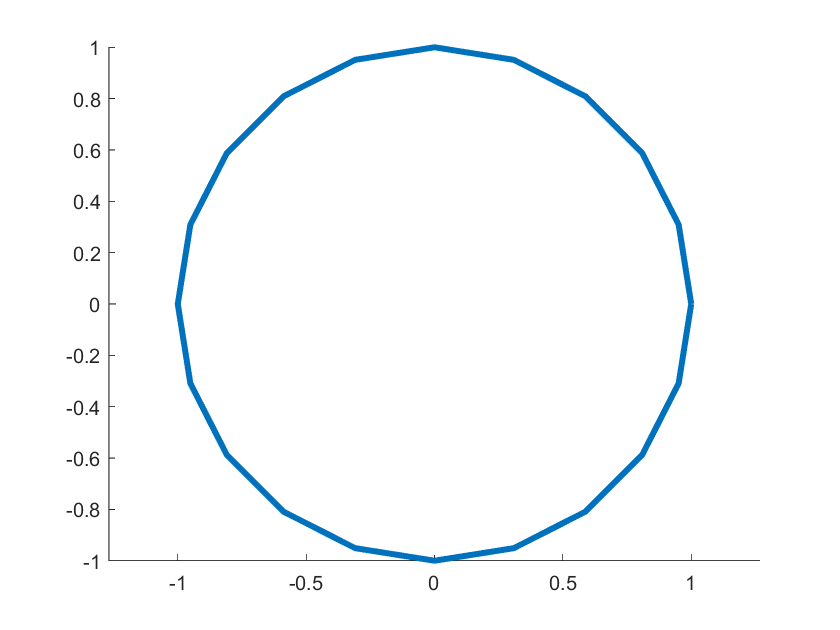} \nonumber \\	
			(a) \, \, \,  L = 3 &  (b) \, \, \, L = 6 & \, \, \, (c) \, \, \, L = 20\nonumber \\
		\includegraphics[width=0.3\linewidth]{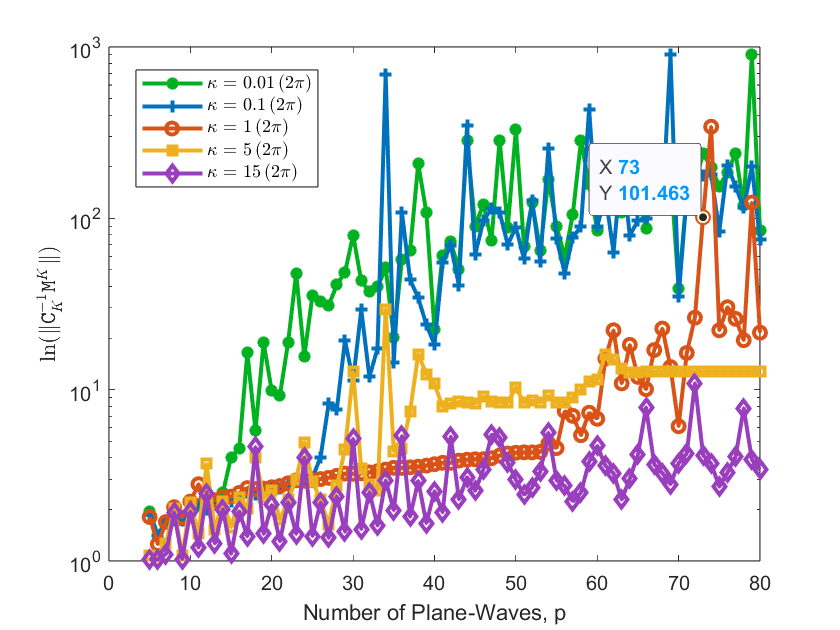}&
		\includegraphics[width=0.3\linewidth]{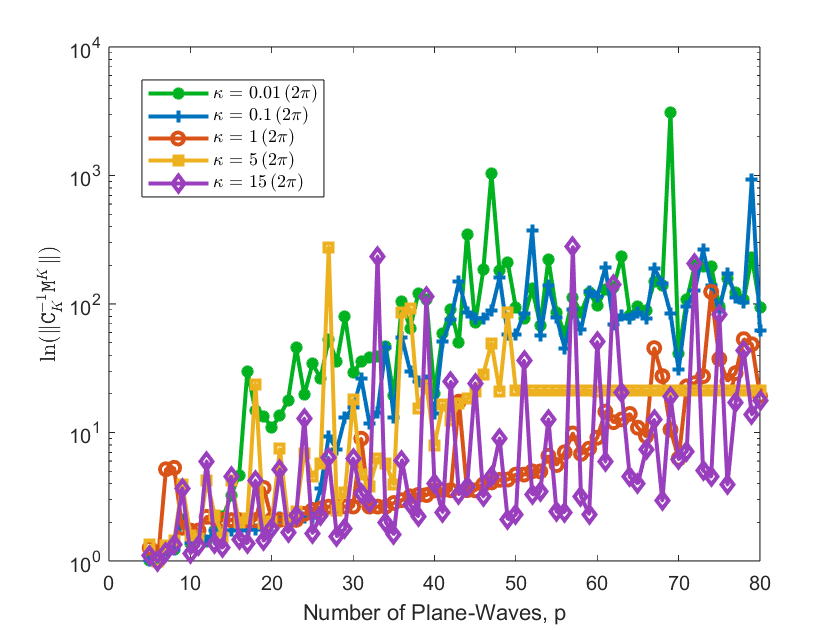} &
		\includegraphics[width=0.3\linewidth]{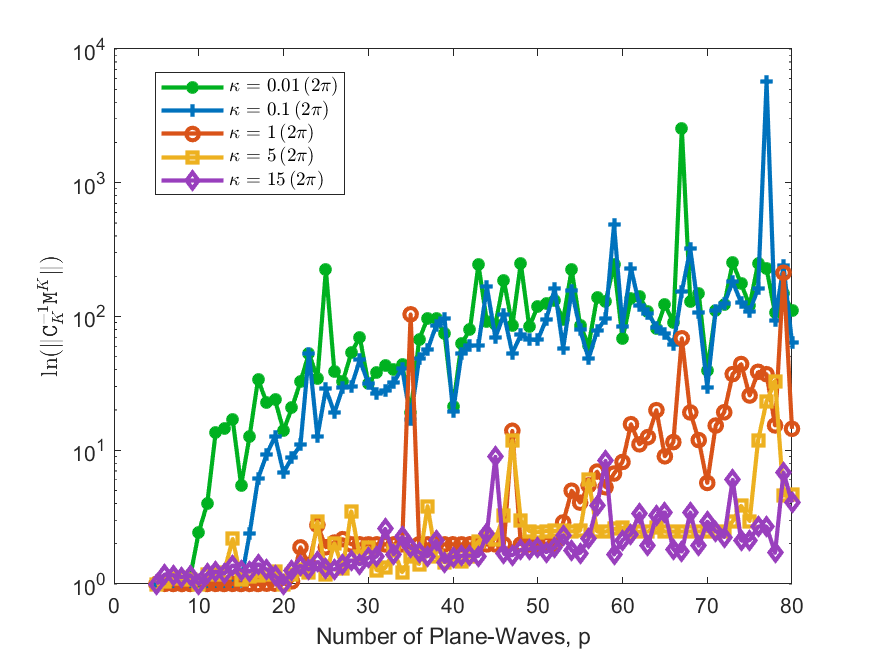} \nonumber \\
			(d)  &  (e)  &  (f) \nonumber \\
		\includegraphics[width=0.3\linewidth]{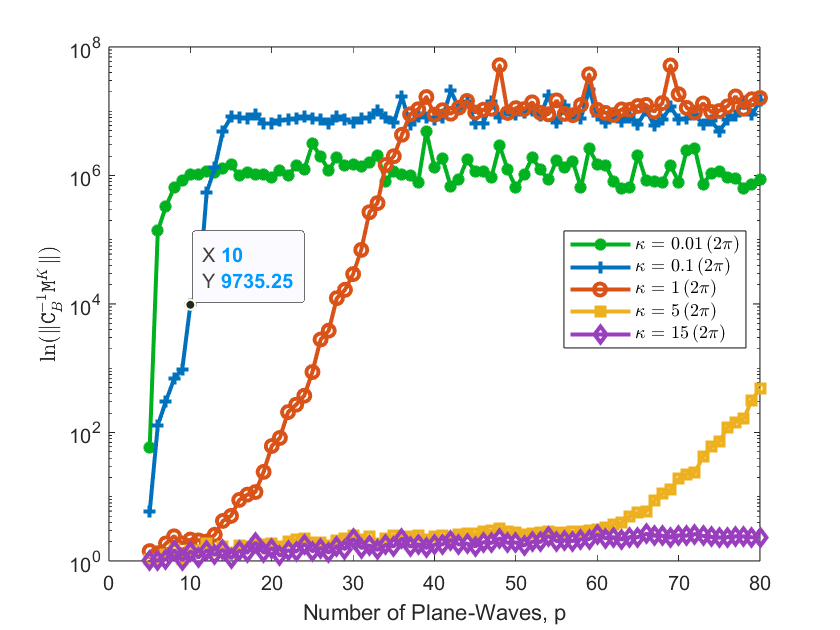}&
		\includegraphics[width=0.3\linewidth]{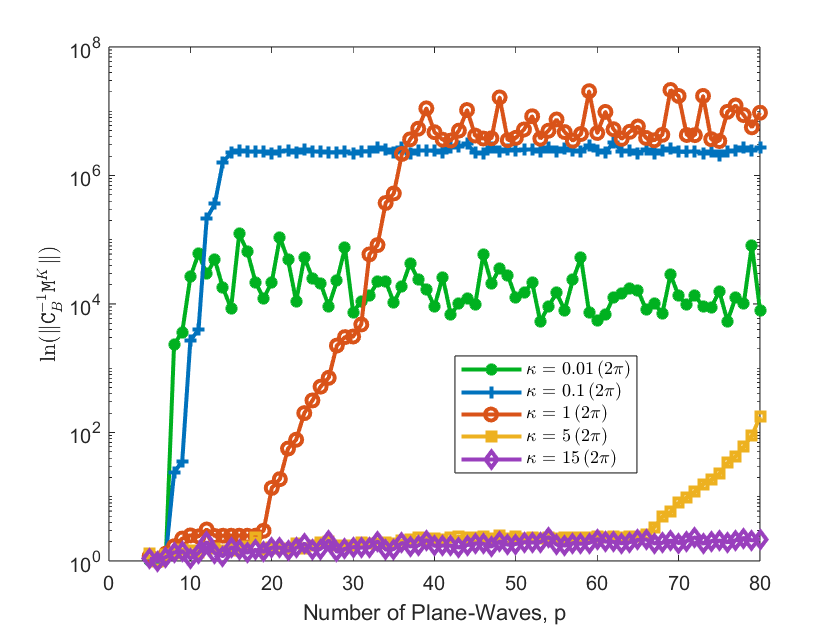} &
		\includegraphics[width=0.3\linewidth]{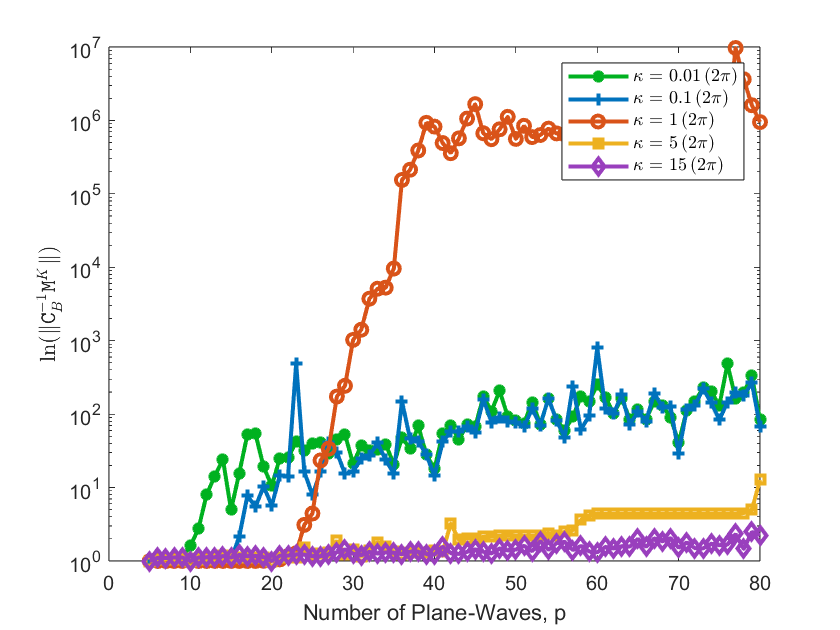} \nonumber \\
			(g)  &  (h) & (i) \nonumber
		\end{array}$
		\caption{Comparison of $\M^K$ with both ${\tt  C }^{reg}_R$ and ${\tt C}^{reg}_B$ on a regular polygons with sides $L = 3$ (a), $L = 6$ (b), and $L = 20$ (c). Results are shown for $\kappa$ values of $0.01 \pi, 0.1\pi,  \pi, 5 \pi,$ and $15 \pi$ for $\ln\left ( \left({\tt C}_R^{reg}\right)^{-1}\M^K\right )$ in the second row, (d) - (f), and for  $\ln\left ( \left ({\tt C}_B^{reg}\right )^{-1}\M^K \right)$ in the second row, (g) - (i) for plane-waves $p = 5,...,80.$ The results are grouped by column for each shape.}	\label{fig:regular-20}
	\end{figure}


\section{Preconditioners and condition numbers}

In  previous sections, we established that the mass, stiffness and cross matrices $\M^K,\S^K,\D^K$  for the plane-wave basis are {\it necessarily} ill-conditioned as the size of the system increases. To address this, one typically constructs a preconditioner for the system matrix $\SyS^K$. The {\it choice} of preconditioner is dictated by the source of ill-conditioning, as well as the iterative method one will use to solve the linear system. A good preconditioner is one which is efficient to implement and improves accuracy.

We document in Figure \ref{fig:allhk} the dependance of the condition number of {\tt SyS} on both element size $h$ and wavenumber $\kappa$. This is consistent with the behaviour of the condition number of the system matrix on the entire mesh in \autoref{fig1}(c), where the conditioning rapidly deteriorates with the number of plane-waves.
\begin{figure}
	\centering
	\includegraphics[width=0.4\linewidth]{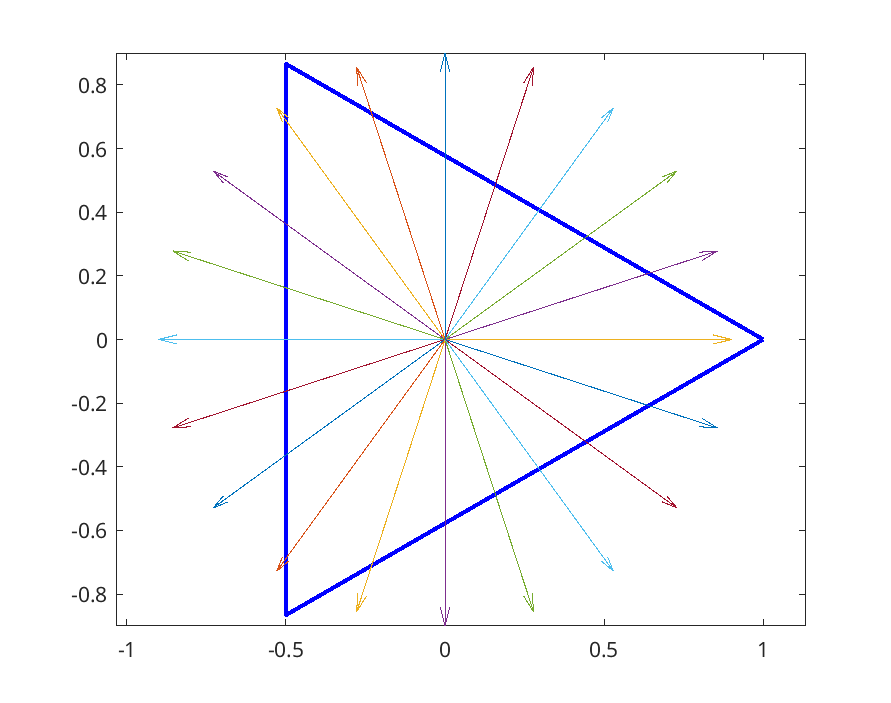}
	\includegraphics[width=0.4\linewidth]{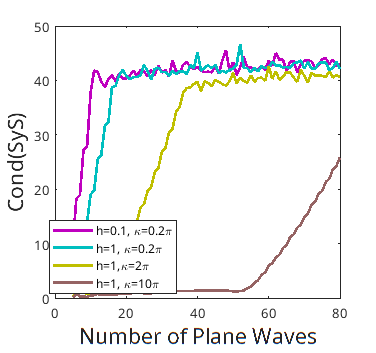}
	\caption{Dependance of condition numbers of the full system matrix, on element size h and frequency $\kappa$. As the number of plane waves increases, the conditioning deteriorates. Computations on an equilateral triangle with vertices on disk of radius $h$, equispaced PWB (left). }
	\label{fig:allhk}
\end{figure}

In the concrete case of the PWB, we present several choices of  preconditioners. Each of these is based on the properties of the PWB, either on the polygonal element $K,$ or a disk of similar diameter. Consequently, each of these preconditioners (apart from the regularized ones) will {\it also} be poorly conditioned. Nonetheless, as we shall see, some dramatically improve the conditioning of the system matrix, while others have good properties when used as part of a direct solve or  iterative strategy such as GMRES.

In previous work on a mesh (instead of a single element) \cite{CongreveMGS}, a modified Gram Schmidt method is used to construct an orthogonal change of basis, leading to a better conditioned system matrix. Concretely, if $\M$ denotes the mass matrix on the whole mesh, then
$ \M = {\tt Q} \tilde{\M} {\tt Q}^*,$ where ${\tt Q}$ is obtained from a QR decomposition of $(\M+\M^*)$. The matrix $\tilde{\M}$ is then somewhat  better conditioned, and to solve $ {\tt Mx=f}$ we instead solve $${\tt \tilde{M} y} = {\tt Q}^* {\tt f}, {\tt x = Qy}.$$ 

Another related approach would be to start with the spectral decomposition of the mass matrix on a single element: 
$$ {\tt \M^K = X \Lambda({\M^K}) X^*},$$ and to use $\Lambda({\M^K})^{-1}$ as a preconditioner to solve 
$ \Lambda({\M^K})^{-1}\SyS^{K} {\tt x} = \Lambda(\M^{K})^{-1} {\tt f}.$
While the columns of ${\tt X}$ are orthogonal (eigenvectors), this approach has two problems. First, we would need to compute the eigendecomposition of the matrix $\M^K$ - this preconditioner would not be efficient in practice. Second, the $\Lambda(\M^K)$ is severely ill-conditioned as $p$ increases.

A compromise based on the observations of the previous sections, is to exploit the `nearly Toeplitz' nature of the mass matrix as the number of plane-waves $p$ increases. In the previous section, we noted that the mass, cross, and stiffness matrices associated with polygonal elements were neither Toeplitz nor circulant. Nonetheless, as the number of sides increases, these matrices tended towards a Toeplitz behaviour. It has been well-documented \cite{chan,circulantreal,strang,Gray} that circulant preconditioners are very effective for Toeplitz matrices. An important criterion for the selection of a preconditioner is its ease of application. To this end, we focus on circulant preconditioners in three categories:
\begin{enumerate}
	\item {\it Disk-based.} These strategies are based on the observation that on the disk of radius $h$, the same as the radius of the circumcircle of polygonal element $K$, matrices $\M^{disk},\S^{disk},$ and $\D^{disk}$ are circulant, Theorem \ref{thm2}. On a polygon, the matrices are {\it close} to circulant, though note the eigenvectors of $\mathtt{SyS}$ will not be the same as those on the disk. 
	For this class of preconditioners, we can use the exact expressions for the mass matrix $\M^{disk}$ or the system matrix $\SyS^{disk}$ on a disk of radius $h$, derived in Section 3; their eigenvalues are efficiently computed using a DFT of their first row. In other words, the eigenvalue decompositions
	$$ \M^{disk} = \U \Lambda({\M^{disk}})\U^*, \qquad \SyS^{disk} = \U \Lambda({\SyS^{disk}}) \U^*$$ can be computed very efficiently, and preconditioners based on these can be applied fast.
	\item {\it Polygon-based.} In this approach, we directly work with the mass and system matrices $\M^K, \SyS^K$ on the physical element $K$ itself. We know (from Section 4) that these matrices are neither Toeplitz nor circulant. Nonetheless, we can {\it build} circulant matrices from these matrices from their first rows, as in Section 4. Once again, the application of such preconditioners will be fast.
	
	We denote by $\C^{K}_{R}$ the circulant matrix constructed by using the first row of $\M^K$  respectively. Recall also that $\C^K_{best}$ is a circulant matrix generated from the average value of the  diagonals of the mass matrix $\M^{K}$, (see \eqref{eq:bcirc})
	\begin{equation}
	\C_{best}^K:=\text{circ}[c_0,c_1,...c_{p-1}],\quad 	c_{i-j} = \frac{1}{n} \sum_{p-q = i - j (mod n) } [\M^{K}]_{pq}.
	\end{equation} 
	
	\item {\it Regularization-based.}
	As we have seen previously, the system matrix $\SyS$ becomes severely ill-conditioned with many eigenvalues of extremely small magnitude. A regularization - explicitly setting eigenvalues smaller than a threshold to 0 - could be a viable preconditioning strategy for such matrices. Such ideas have been fruitfully applied to the solution of ill-posed Toeplitz problems, \cite{chan, elden}.
	However, to avoid an eigendecomposition of $\M^K$, we use instead the (easily computed) eigenvalues of $\M^{disk}$, corresponding to an element with similar size as $K$. 
	
	In order to solve ${\tt A x} = {\tt f}$, suppose ${\tt A}$ is close to a circulant matrix ${\tt Q} = \U \begin{bmatrix} \Lambda_1 & 0\\
		0 &\Lambda_{\delta}\end{bmatrix} \U^* $ where $\Lambda_1$ is a diagonal matrix with entries of magnitude  $\geq \delta > 0$ and $\Lambda_{\delta}$ is diagonal with entries smaller than this regularization parameter $\delta$. In this case, one can define either 
	$ {\tt Q}^{regularized} :=  \U \begin{bmatrix} \Lambda_1 & 0\\
		0 & I\end{bmatrix} \U^* $ and use its inverse as a preconditioner, or one can directly use the singular matrix ${\tt Q}^{\dagger}:=\U \begin{bmatrix} \Lambda_1^{-1} & 0\\
		0 & {\tt Z}\end{bmatrix} \U^* $ where ${\tt Z}$ is a zero matrix. 
	
	In the concrete case of the Trefftz methods, we choose ${\tt Q}:= (\kappa^2\M^{disk})^{1/2}$, whose eigenvalues are readily computed. 	
\end{enumerate}

To compare the performance of these preconditioners, we shall fix the element shape $K=$ an equilateral triangle, $h=0.1,$ the radius of its circumcircle, and  $\kappa=0.2\pi$, unless otherwise specified. The condition numbers of the preconditioned system will be compared. 
The idea we intend to exploit is that as the number of plane-waves becomes large, the deviation of the mass matrix $\M^{reg}$ from the circulant matrix $\M^{disk}$ becomes small, as in Figures \ref{Fig:departurefromToeplitz} and \ref{Fig:departurefromCirculant}. We investigate the behaviour of several preconditioners for the matrix ${\tt SyS}$ defined in \eqref{eq:systemlinear}, with the eventual goal of solving ${\tt SyS x}={\tt f}$.

We anticipate square-root type preconditioners $\P$ which approximate $\SyS^{-1/2}$ will outperform preconditioners which approximate $\SyS^{-1}$ while solving linear systems, since they will be better conditioned; the singular values of $\SyS$ include both very small and very large values as the number of plane waves increases.

To test this, we solve $\SyS {\tt x = f}$ using the polygon-based circulant matrix $\kappa^2 \C^K_{{\tt R},\M}$ as a  preconditioner. That is, to solve ${\tt SyS}^K {\tt x} = {\tt  f}$, we solve instead
$(\kappa^2\M^{disk})^{-1}{\tt SyS^K {x}} = (\kappa^2\M^{disk})^{-1}{\tt  {f}}$ (left preconditioning). We compare the results with using the square root of this matrix,  $(\kappa^2 \M^{disk})^{1/2}$, instead.  From Figure \ref{fig:compareconditioning} we clearly see that this preconditioner leads to well-conditioned system matrices in comparison with the `square-root' preconditioners $\P_1,\P_3,\P_5$ which approximate $({\M^K})^{-1/2}$. However, the $L^2$-errors in solving a point-source scattering PDE are substantially worse using this preconditioner. 

\begin{figure}
	\centering
	\includegraphics[width=0.3\linewidth]{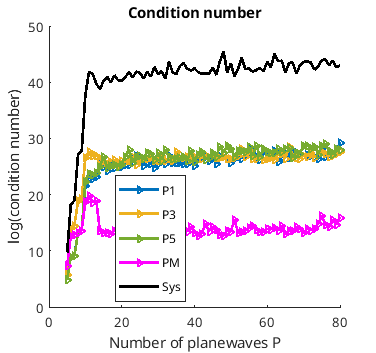}
		\includegraphics[width=0.3\linewidth]{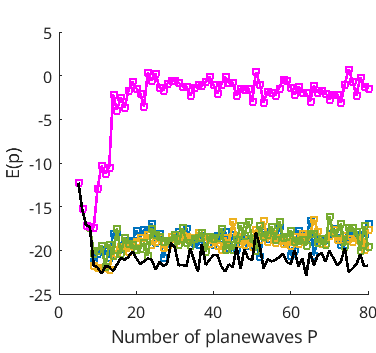}
			\includegraphics[width=0.3\linewidth]{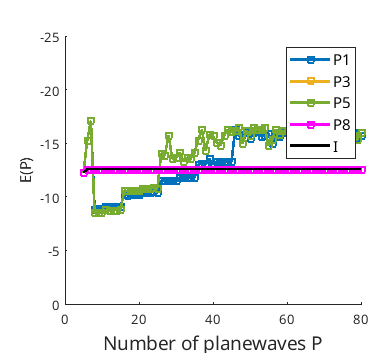}
	\caption{Impact of preconditioning with ${\tt P} \approx \M^{-1}$ versus  ${\tt P} \approx \M^{-1/2}$. Left: domain $K$ and plane wave directions. Middle: condition number of $\P_i\SyS$ for $\P_1,\P_3, \P_5$ and $PM:=(\kappa^2\M^{disk})^{-1}$. Middle: (direct),  Right(GMRES): Relative errors $E(p)$ for preconditioned systems. Black curves: unpreconditioned system. The figure legend is the same for all 3 subplots. }
	\label{fig:compareconditioning}
\end{figure}


Based on these results, in the remainder we focus our attention on 7 preconditioning matrices $\P_1-\P_7$ of square-root type, and their performance as left, two-sided, or right preconditioners. 
We group these into three broad categories as before.
\begin{itemize}
	\item{\it Disk-based preconditioners.}
Consider a disk with the same center and radius as the circumcircle of the element $K$. We use the (square root of) $\M^{disk}$ or $\SyS^{disk}$ as preconditioners. Concretely, define   
	\begin{equation}\label{eq:P1P2}
		\P_1:=(\kappa^2\M^{disk})^{-1/2} = \kappa \U \Lambda({\M^{disk}})^{-1/2} \U^*,\qquad \P_{2}:=\U (\Lambda({\mathtt{SyS}^{disk}}))^{-1/2}\U^* 
	\end{equation}
	\item {\it Polygon-based preconditioners.}
	We can generate circulant matrices from the first row of either $\kappa^2\M^K$ or $\SyS^K$, and use these as preconditioners to define 
	\begin{equation}\label{eq:P3P4}\P_3:= (\kappa^2 \C^{K}_{{\tt R},\M})^{-1/2} = \U (\kappa^2\Lambda({\C^{K}_{{\tt R},\M}}))^{-1/2} \U^*, \qquad  \P_4:= (\C^K_{{\tt R},\SyS})^{-1/2}.\end{equation}	
A Hermitian circulant matrix can be generated by an averaging process from $\M^{K}$ in  \eqref{eq:bcirc}, \begin{equation}\label{eq:P5}
		\P_5:=(\kappa^2\C^K_{best})^{-1/2} =  \U (\kappa^2\Lambda \left ( \C^K_{best} \right )^{-1/2}) \U^*.
	\end{equation}	
	\item {\it Regularization-based preconditioners.}
	A simple regularized preconditioner can be constructed as 
	$$ \P_6:=\kappa \U  \, {\tt diag}(\mu_1,...,\mu_{\delta},1,1,...,1)^{-1/2} \U^{*}, \qquad \mu_i :=\lambda_i(\M^{disk})$$
	The eigenvalues of $\M^{disk}$ are easy to compute. In $\P_6$ we retain the $\mu_i$ such that ${\mu_i}>\delta$ in magnitude, and set the others to 1. The matrix $\P_6$ has smaller norm than that of $\P_1$. The tolerance $\delta>0$ is defined by the user.  This idea is motivated by \cite{channg} and \cite{HANKE1998137} on preconditioning Toeplitz matrices, who recommend the use of (right) regularization-based preconditioners.  

Another form of regularization is the use of a singular preconditioner, motivated by the work of \cite{elden} who studied the behaviour of such preconditioners with GMRES while solving non-self-adjoint ill-posed problems. 
	We define
	\begin{equation} \label{eq:P7}\P_7:=\frac{1}{\kappa}\U \, {\tt diag}(1/\mu_1,...,1/\mu_{tol},0,0,...0)^{1/2}\U^{*}.\end{equation}
Here, again $\mu_i$ are the eigenvalues of the mass matrix on the disk $\M^{disk}$ so that  ${\mu_i}\geq \delta$ in magnitude.
\end{itemize}

In Figure \ref{fig:PreconditionerComparison}, we present results on the system condition numbers based on these choices, for $\kappa h = 0.01\pi.$ As can be seen, the best reduction on condition number is provided by the use of the left-right preconditioners. The preconditioners $\P_i$ we have chosen are of square-root type, and therefore $\P_i \SyS\P_i \approx {\tt I}$.

The regularized preconditioner $\P_{6}, \P_{7}$ offer a reduction in condition number which depends on the tolerance $\delta$ picked. These preconditioners should be used with care. 

\begin{figure}
	\centering
	\includegraphics[width=0.3\linewidth]{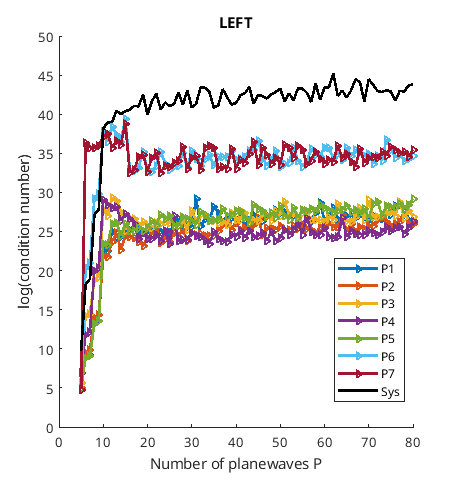}
	\includegraphics[width=0.3\linewidth]{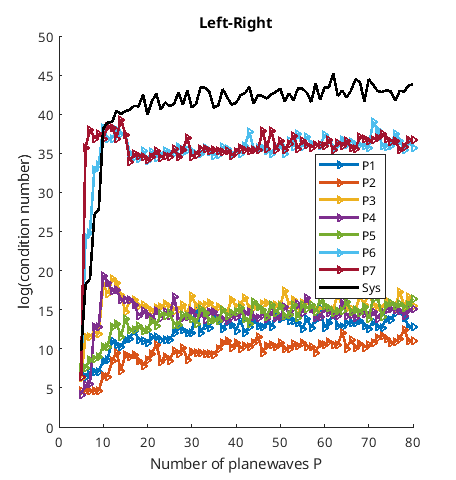}
	\includegraphics[width=0.3\linewidth]{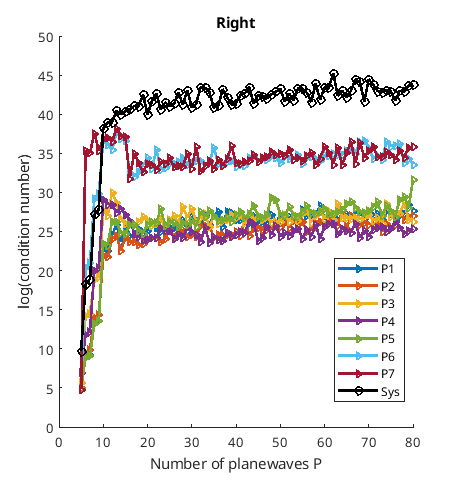}	
	\caption{A comparison of the condition number of the (preconditioned) system matrix on an equilateral with vertices on a disk of radius $h=0.1, \kappa =0.2\pi.$ Left panel: $cond(\P_i\SyS)$. Middle: $cond(\P_i\SyS\P_i)$. Right: $cond(\SyS\P_i)$.Parameters used for the regularization-based preconditioners are $\delta =1e-10$. Condition number of $\SyS^K$ shown in black. }
	\label{fig:PreconditionerComparison}
\end{figure} 
\section{Preconditioners and $L^2$ errors}

One deploys  preconditioners $\P$ while using an iterative solver such as GMRES to solve the problem $\SyS^{K} {\tt x} = {\tt f}$. One could also use a direct approach (such as \MATLAB's backslash ) on the preconditioned system. We examine the action of the matrices $\mathtt{P}_i$ as {\it left preconditioners}, {\it two-sided preconditioner} or {\it right preconditioners} while solving the system $\SyS^K {\tt x}  = {\tt f}$. That is, we solve for ${\tt x}$ as
$${\mathtt P}\SyS^K {\tt x} = {\tt P}{\tt f},\qquad \P{\tt SyS^K }\P {\tt y} = \P{\tt  f}, \,\, {\tt x} = \P{\tt y}, \, \, \text{or} \, \,\qquad {\tt SyS^K }\P {\tt y} = {\tt f}, \,\, {\tt x} = \P{\tt y}.$$ 

For convenience we remind the reader of the PDE of interest: find $u\in H^1(K)$ such that 
	\begin{eqnarray} 
		\Delta u + \kappa^2 u = 0 \, \mbox{ in } K, \label{eq:helmeqt}\qquad 
		\frac{\partial u }{\partial \nbm } + i \kappa u = g_K\,\, \mbox{ on } \partial K.
	\end{eqnarray}
The associated system matrix on a single element, $\Omega = K$ is $\SyS^K = (k^2  \M^K +  ik (\S^K-\bar{\S}^K) + \D^K)$.
In order to test ideas, we use $g_K$ for which the solution $u$ is known. We
generate the related ${\tt {f}}\in \mathbb{C}^P$ and solve
\begin{equation}\label{eq:systemlineart} 
	{\tt \SyS^K \, {x}=  {f}, \quad x}=[x_1,x_2,...,x_p]
\end{equation}	
We can then compare the computed pde solutions $u_p:=\sum_{m=1}^p x_m \phi_m$ of \eqref{eq:helmeqt} with the exact solution. The measure we track is the relative $L^2$-error of the computed solution $u_p$ as a function of the number of PWB :
{\begin{equation} \label{eq:measure}
		E(p):=\left (\int_K |u-u_p|^2 \right)^{1/2}/\left(\int_K |u|^2 \right)^{1/2}
\end{equation} }
This allow us to directly study the impact of using the preconditioners $\P_1-\P_7$ with \MATLAB's GMRES preconditioner. We set $\M_i:=\P_i^{-1}$ for $i=1,...,6$ (these inverses are easily computed thanks to their structure), and the function call is then {\tt gmres($\SyS$, {f}, restart, tol, maxit, $\M_i$)} for left preconditioning. The preconditioner $\P_7$ is {\it not} invertible, and therefore the call to \MATLAB would be {\tt gmres($\P_7\SyS$, {f}, restart, tol, maxit)}

We address the following questions in this section.
\begin{enumerate}
	\item Which preconditioners $\P_1,...,\P_7$ work best with a direct solver such as \MATLAB's backslash?
	\item Which preconditioners work best with GMRES?	
	\item How sensitive is our preconditioning strategy to the size of the domain $K$ and the wavenumber $\kappa$?	
	\item Is there a universally-recommended preconditioning strategy? 
\end{enumerate}
\subsection{Plane wave scattering}

We begin by solving \eqref{eq:helmeqt} for the case of plane-wave scattering, with wavenumber $\kappa = 0.2\pi$, i.e., $g_K= {\rm e}^{\iota \kappa x}$ in \eqref{eq:helmeqt}. The domain $K$ is an equilateral triangle whose vertices are on a disk of radius $h=0.1$.
In Figure \ref{fig:l2errorspentagon1-5}, we report the relative $L^2$ errors using both direct solvers and GMRES, with the default tolerance $tol=[]$ and maximum number of iterations, with restart after 5 iterations. Concretely, for the left preconditioners $\P$, we report the $L^2$ error when solving
$ \P \SyS {\tt x}= \P {\tt f}$ directly, and by using \MATLAB's GMRES call ${\tt  x= gmres(SyS,f,5,[],[],P^{-1})}$. For the left-right preconditioners, we find the solution $x$ using $\P \SyS \P y= \P {\tt f}, {\tt x}=\P {\tt y}$ directly, or using the GMRES call ${\tt y= gmres(\SyS\P,R,restart,tol,maxit, \P^{-1});x=\P {\tt y}}$. The right preconditioners can be similarly studied. As discussed, the use of $\P_7$ as a two-sided or left preconditioner requires some care.

We observe in this experiment that the recommended choice of preconditioner clearly depends on whether we use a direct or indirect solver. At least for this experiment, using {\tt backslash} on the unpreconditioned system is optimal, followed by its use with left preconditioners. Not surprisingly, using the poorly-conditioned matrices $\P_i$ as right preconditioners with a direct solver leads to poor accuracy.
\begin{figure}[h]
	\centering
	\includegraphics[width=0.3\linewidth]{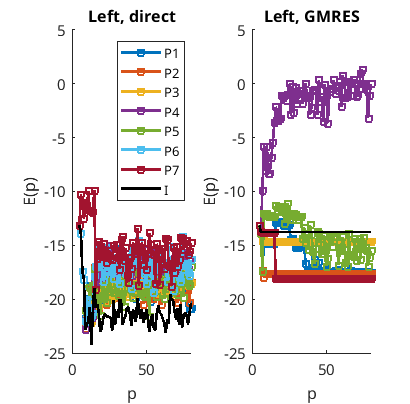}
	\includegraphics[width=0.3\linewidth]{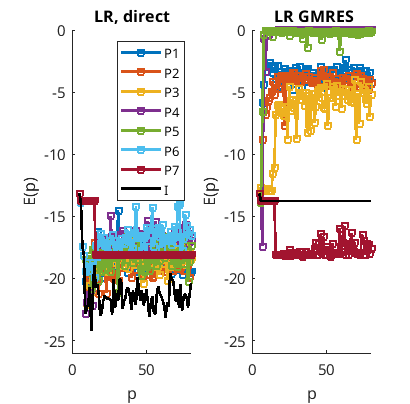}
	\includegraphics[width=0.3\linewidth]{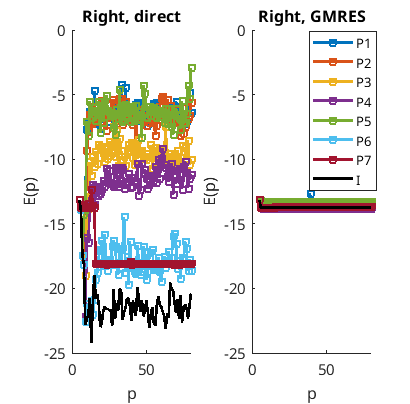}
	\caption{Plane-wave scattering. Effect of preconditioners on the solution of $\SyS x = f$. Left panel: Errors $E(P)$ (\eqref{eq:measure}) using direct or GMRES methods for left preconditioned systems. Middle: Left-right preconditioning. Right: Right preconditioning. Here, the domain $K$ is a regular triangle with vertices on a disk of radius 0.1, with wavenumber $\kappa = 0.2\pi$. Regularization parameter is $\delta=1e-10$, and GMRES tolerance is the \MATLAB default $tol=1e-6$. Also shown (black curves) are solves using the unpreconditioned system.}
	\label{fig:l2errorspentagon1-5}.
\end{figure}
From these experiments, it is easy to see some preconditioning strategies are terrible! We should {\it not} use preconditioner $\P_4$ \--- using $(\SyS^{K})^{1/2}$ \--- as a left preconditioner. Two-sided preconditioners performed poorly in comparison with left or right preconditioning. 

This highlights one of the important observations in this paper: improving the conditioning of the linear system does not directly improve the accuracy of the pde solution. GMRES, for instance, will minimize the residual of the {\it preconditioned system} and not the quantity $\|{\tt x_{true} - x_{computed}}\|_{L^{2}}.$

\subsection{Point source scattering}

Next we consider the situation of point-source scattering, with the source located outside the domain $K$.  Unlike the previous experiment, where it is possible for the incident plane-wave to be aligned very well with a basis function, for point source scattering this is not the case. Nonetheless, the behaviour of preconditioners $\P_i$ was very similar to the plane-wave scattering case. The results are summarized in Figure \ref{fig:fanpentagon}. Once again we see that the performance (in terms of error reduction) depends on whether one uses the preconditioner with  GMRES or a direct solver.

\begin{figure}
	\centering
	\includegraphics[width=0.3\linewidth]{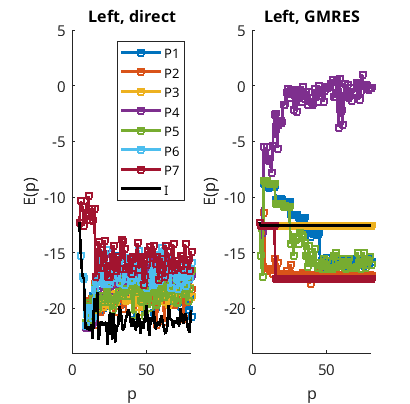}
	\includegraphics[width=0.3\linewidth]{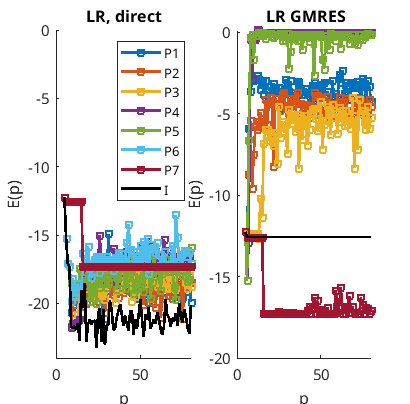}
	\includegraphics[width=0.3\linewidth]{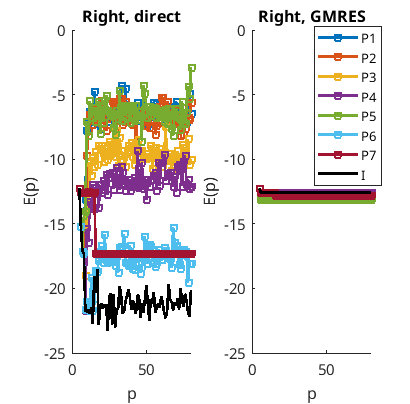}
	\caption{Point-source scattering from an equilateral triangle, $h=0.1, \kappa = 0.2\pi$. The point source is located at (2,-4). Left: Relative $L^2$ errors $E(p)$ v/s number of PWB $p$ using the selected left preconditioners as part of a direct or iterative solver. Middle: Left-right preconditioners. Right: right preconditioning. Regularization parameter is $\delta=1e-10$, and GMRES tolerance is $tol=1e-6$, with restart after 5 iterations. }
	\label{fig:fanpentagon}
\end{figure}

We see preconditioning does {\it not} improve the performance of the direct solver. Using  $\P_2$ or $\P_5$ as left preconditioners gives results which are comparable to the unpreconditioned system. On the other hand, left preconditioning of GMRES by the singular preconditioner $\P_7$, followed by circulant preconditioners $\P_2$ and $\P_5$ improves accuracy. The right preconditioner strategies all have similar levels of accuracy as the unpreconditioned GMRES.

Two-sided preconditioners perform worse than one-sided preconditioning, despite the superior conditioning of the system (apart from $\P_7$). This isn't surprising: multiplication of the intermediate vector ${\tt y}$ by $\P_i$ is not a well-conditioned operation.

Based on these observations, in the next subsections we omit $\P_4$ from consideration. We also focus on one-sided preconditioning strategies.

\subsection{Element size and wavenumber}

Both the size (measured by the radius of the circumcircle, $h$) and  the wavenumber $\kappa$ play a role in the conditioning of the matrices as well as preconditioners, as was seen in preceding sections. We study this effect in the concrete case of point-source scattering of wavelength $\kappa$ from a regular triangle $K$ whose size is $h$. The impact of size and wavenumber is clear both for the left preconditioning strategies (Figure \ref{fig:leftcondhsmall}). Moreover, it is again evident that the lowest reduction in condition number does not always translate into the best $L^2$-error reduction. Left preconditioning by $\P_2, \P_5,$ or $\P_7$ either improves accuracy in a GMRES solve, or gives comparable results.
\begin{figure}
	\centering
	\includegraphics[width=0.22\linewidth]{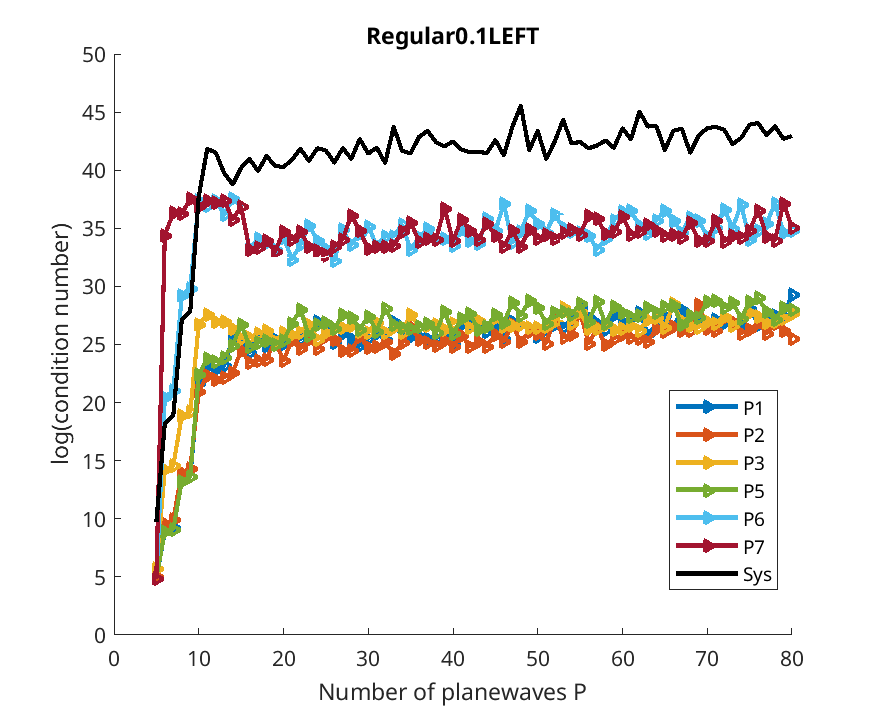}
	\includegraphics[width=0.22\linewidth]{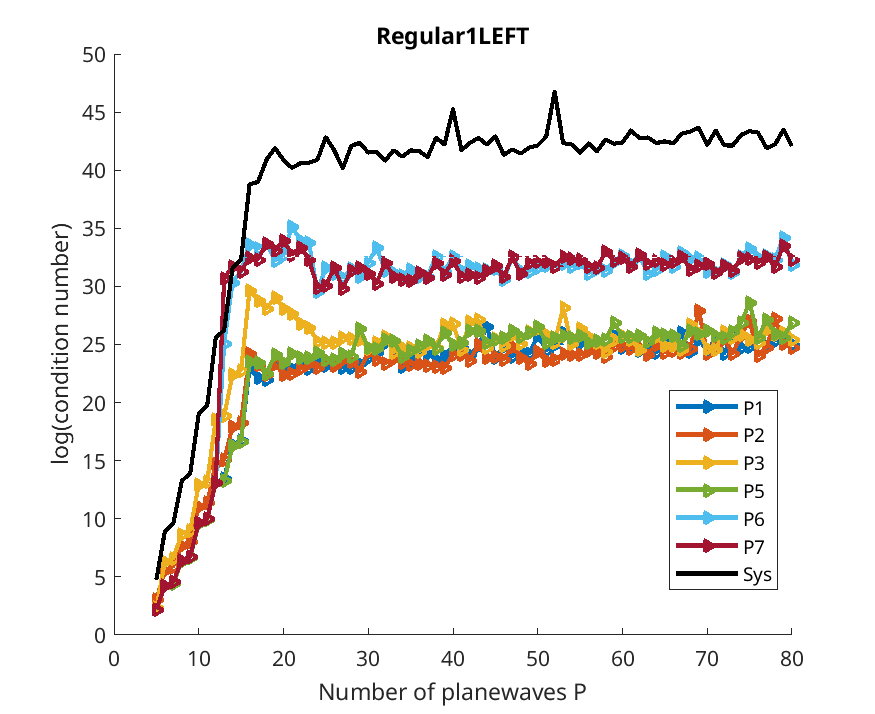}
	\includegraphics[width=0.22\linewidth]{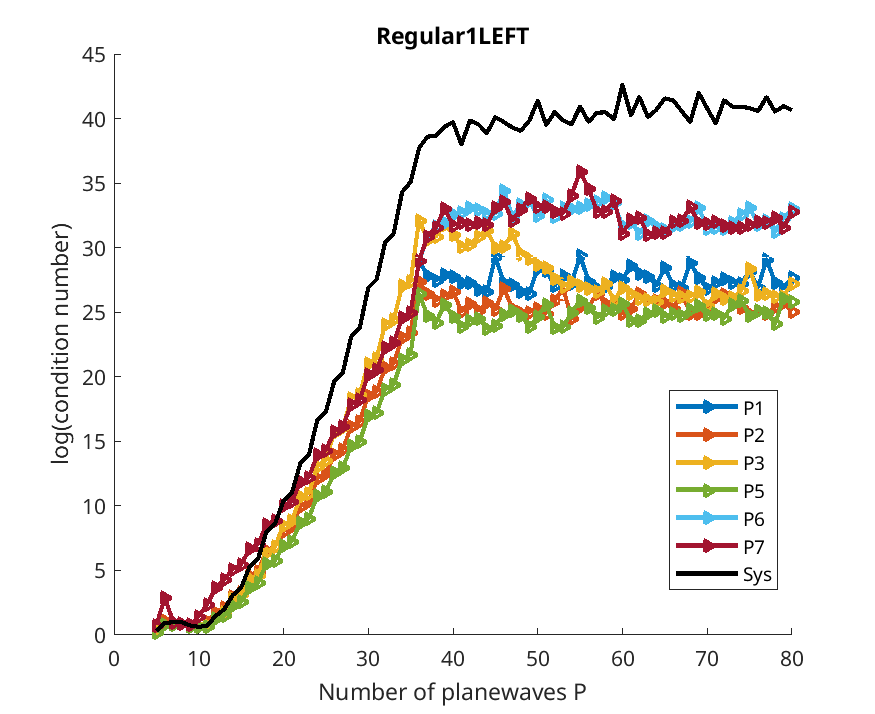}
	\includegraphics[width=0.22\linewidth]{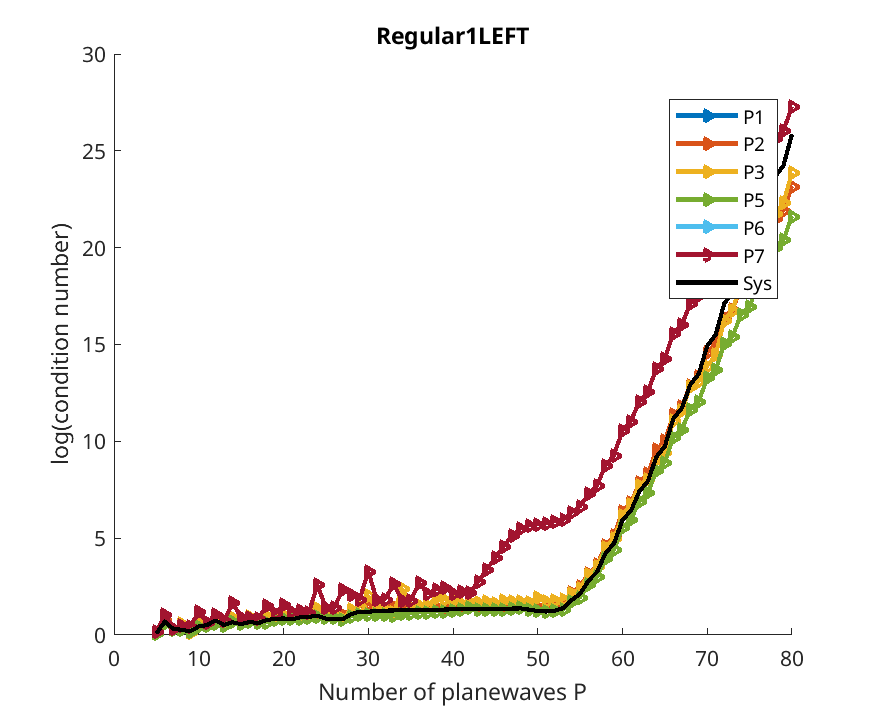}
	
	\includegraphics[width=0.22\linewidth]{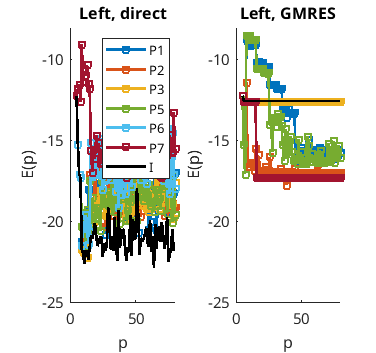}
	\includegraphics[width=0.22\linewidth]{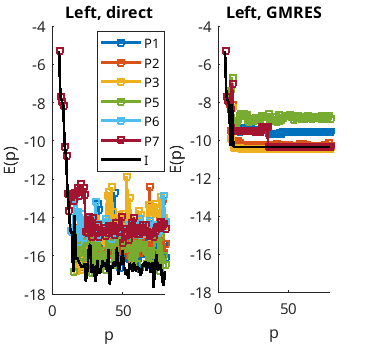}
	\includegraphics[width=0.22\linewidth]{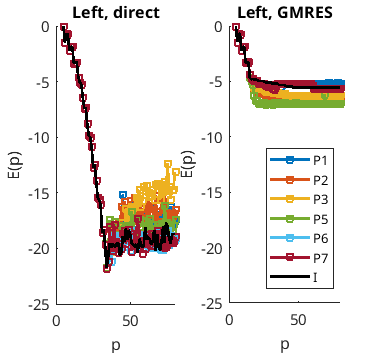}
	\includegraphics[width=0.22\linewidth]{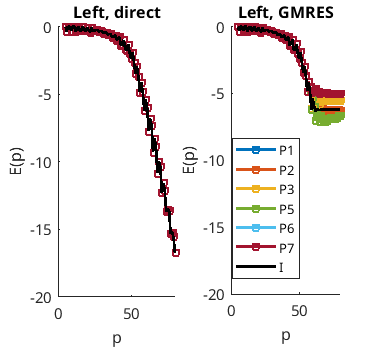}
	
	\includegraphics[width=0.22\linewidth]{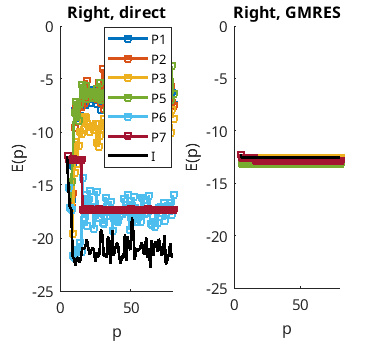}
	\includegraphics[width=0.22\linewidth]{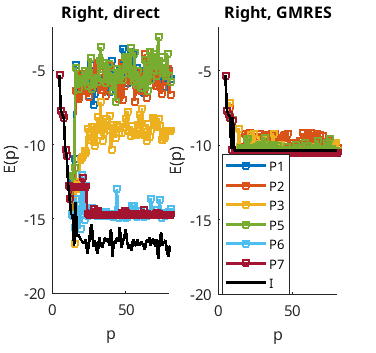}
	\includegraphics[width=0.22\linewidth]{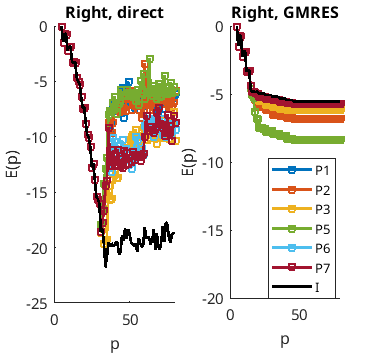}
	\includegraphics[width=0.22\linewidth]{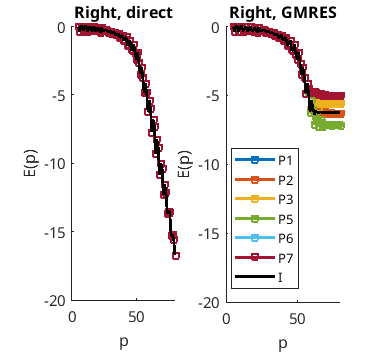}
	\caption{Behaviour of left and right  preconditioners on an equilateral triangle. Left to right: ($h=0.1, \kappa = 0.2\pi$),  ($h=1,\kappa=0.2\pi)$. ($h=1,\kappa=2\pi$), ($h=1,\kappa = 10 \pi$). Top row: condition numbers. Middle row: Errors $E(p)$ using left preconditioning. Bottom row: Right preconditioning. We have set $\delta=1e-10$ and $tol=1e-6$ in each of these.}
	\label{fig:leftcondhsmall}
\end{figure}

\subsection{Impact of regularization and GMRES tolerance parameters}

There are two tolerances we pick: $\delta$ for the preconditioners $\P_6,\P_7$ and the GMRES tolerance. To study their impact, we consider the experiment of point-wave scattering from an equilateral triangle with vertices on a circle of radius $h=0.1$.

We first explore whether the choice of $\delta$ in the regularization-based preconditioners $\P_6,\P_7$ matters. We set the wavenumber to be $\kappa = 0.2\pi$ and the GMRES tolerance to be $tol=1e-6$. The conditioning of the problem is known to be poor. As we can see from Figure \ref{fig:comparingconditionersleftdelta}, as $\delta$ decreases, the accuracy of the computed solution increases. We also see the 'singular' preconditioner $\P_7$ leads to better accuracy.

\begin{figure}
	\centering
	\includegraphics[width=0.3\linewidth]{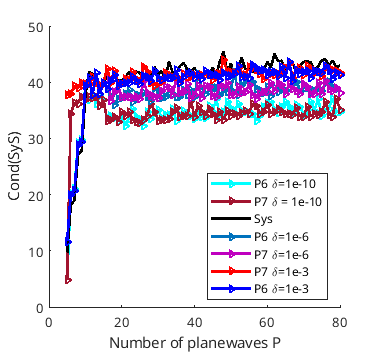}
	\includegraphics[width=0.3\linewidth]{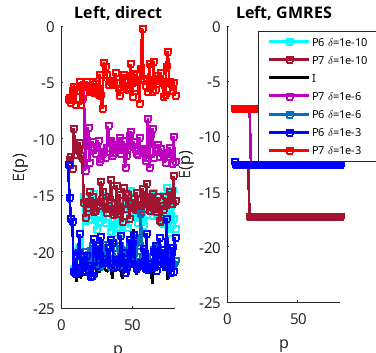}
	\includegraphics[width=0.3\linewidth]{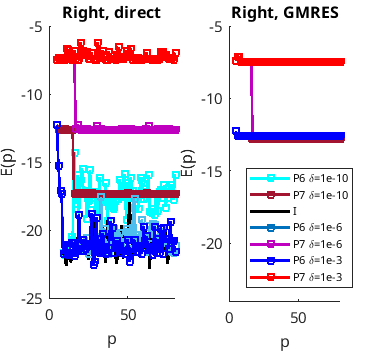}
	\caption{Impact of the regularization parameter $\delta$. Left: the condition number of the (left) preconditioned systems matrix using $\P_6$ and $\P_7$, using $\delta = 1e-3, 1e-6$ and $\delta=1e-10$ respectively. Middle: $L^2$ errors while left preconditioning. Right: $L^2$ errors with right preconditioning.  Point-wave scattering from an equilateral triangle, $h=0.1, \kappa=0.2\pi$. We have used GMRES with default tolerance ${\tt tol=}1e-6$ and restart after 5 iterations. }
	\label{fig:comparingconditionersleftdelta}
\end{figure}
We now study the impact of changing the GMRES tolerance, while keeping the regularization parameter $\delta=1e-10$ fixed. This is documented in Figure \ref{fig:triangletoll2errorsleft}. The only impact appears to be on right preconditioning, and that too in going from $tol=1e-4$ to $tol=1e-6$. We also see from this example that left preconditioning by $\P_5$ based on the 'best' circulant preconditioner provides good error reduction.

\begin{figure}
	\centering
	\includegraphics[width=0.3\linewidth]{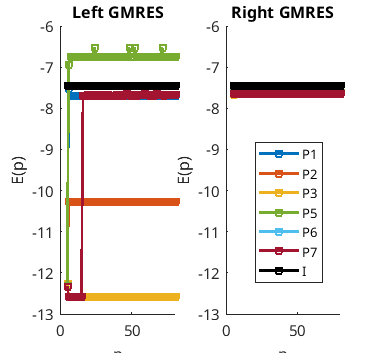}
	\includegraphics[width=0.3\linewidth]{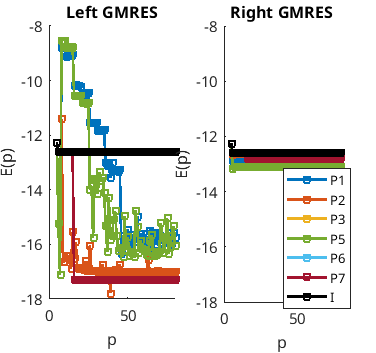}
	\includegraphics[width=0.3\linewidth]{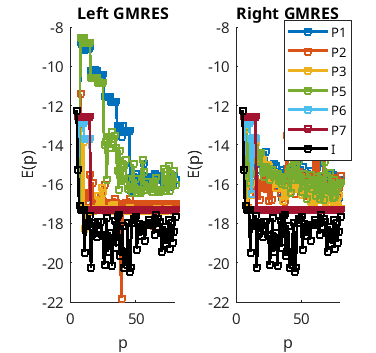}
	\caption{Impact of the GMRES tolerance on left and right preconditioning. Left: $L^2$ errors  using $tol = 1e-4$. Middle:  Using  $tol=1e-6$. Right: $L^2$ errors using $tol=1e-14$.  Point-wave scattering from an equilateral triangle, $h=0.1, \kappa=0.2\pi$. We have used $\delta=1e-10$ and restart after 5 iterations.}
	\label{fig:triangletoll2errorsleft}
\end{figure}

\subsection{Effect of element shape}
In this subsection we attempt to address a question posed by Prof. P. Monk \cite{peter}: how should one select plane-wave basis functions in a 'skinny' triangle?

We study this question for point-source scattering where the point source is located at $(2,-4)$. The element $K$ is an isosceles triangle with short side of length $2 \sin(\frac{\pi}{25})$, and the apex is located at distance $1+\cos(\frac{\pi}{25})$ from this side. The plane waves are 'centered' at $x_{K,\ell}$, where this point is chosen as the centroid of the triangle, a vertex of the short side, or the apex. The point $x_{K,\ell}=(0,0)$ in \eqref{eq:phi} and the plane wave directions are chosen as shown. The relative $L^2$-errors using the various preconditioners are shown in Figure \ref{fig:fan}.  For this experiment we see that selecting uniform-in-angle plane waves centered at the centroid is the best strategy,  and left preconditioners $\P_2, \P_5$ or $\P_7$ work well with GMRES. Choosing regularly-spaced planewave directions between 0 and $2\pi$, centered at the apex leads to somewhat improved accuracy (not shown), but is still not comparable to centering the PWB at the centroid.

\begin{figure}
	\centering
	\includegraphics[width=0.3\linewidth]{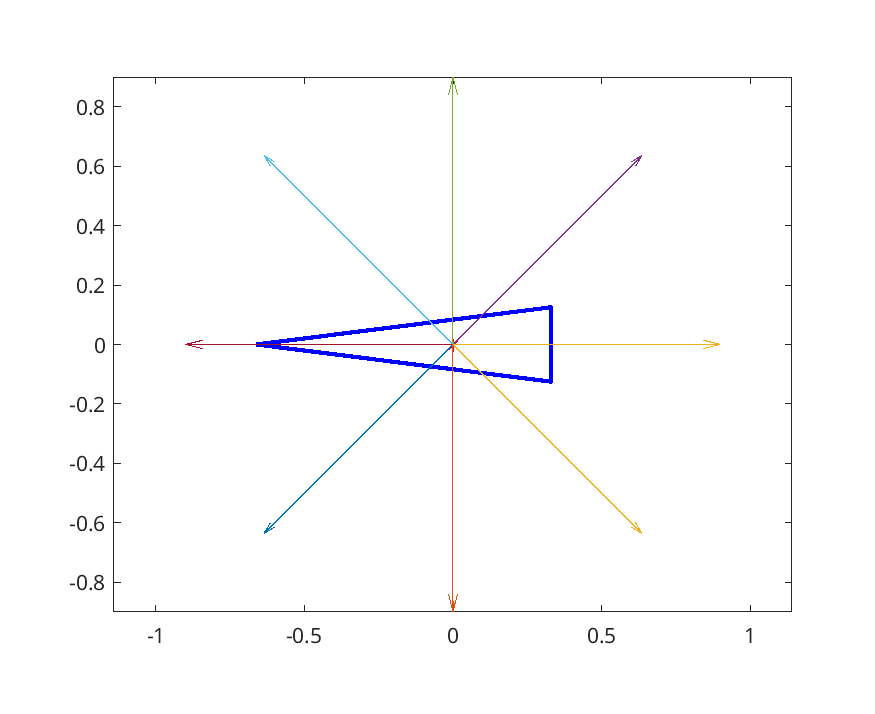}
	\includegraphics[width=0.3\linewidth]{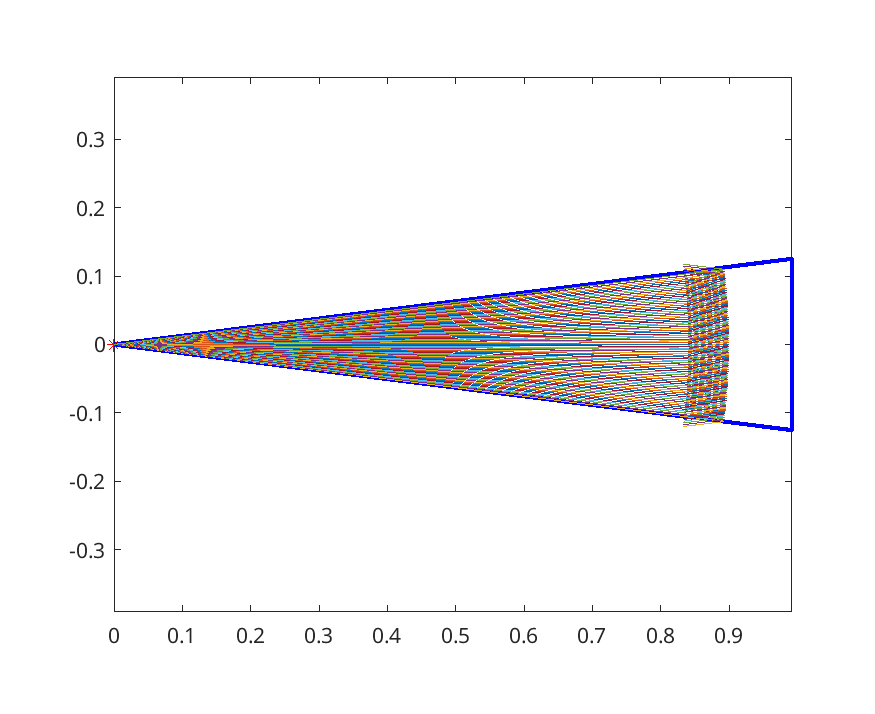}
	\includegraphics[width=0.3\linewidth]{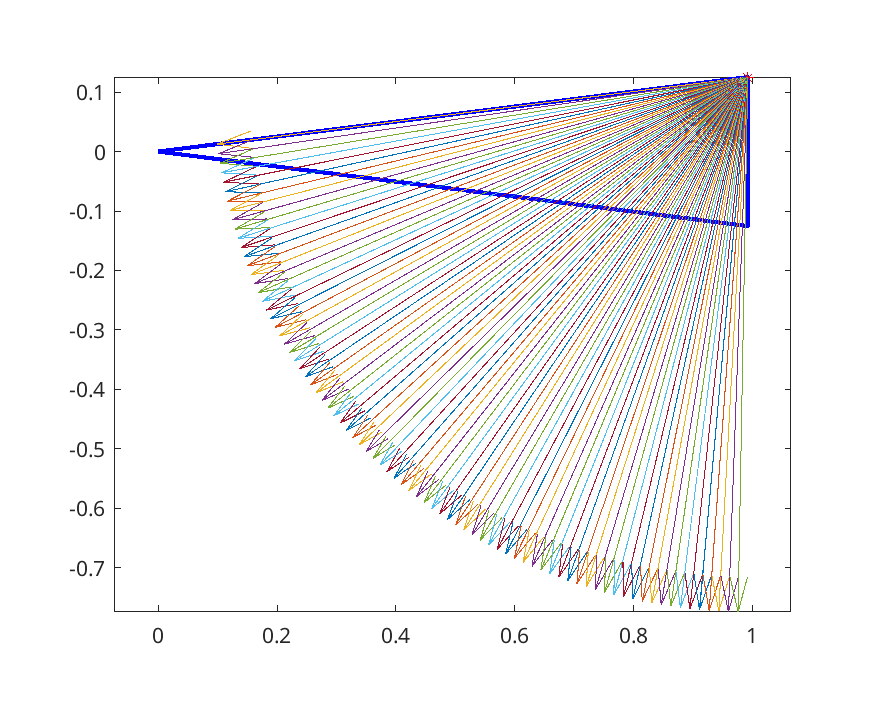}

	\includegraphics[width=0.3\linewidth]{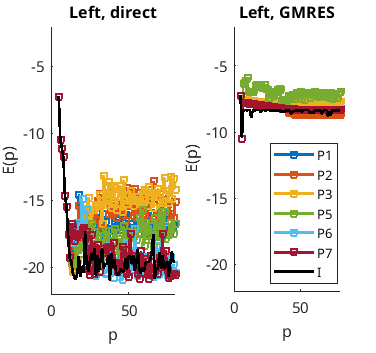}
	\includegraphics[width=0.3\linewidth]{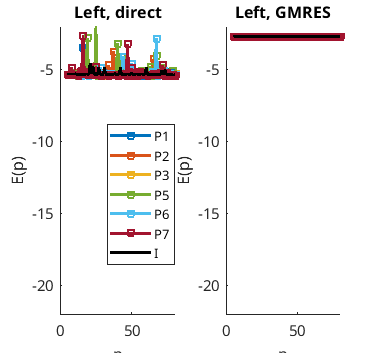}
	\includegraphics[width=0.3\linewidth]{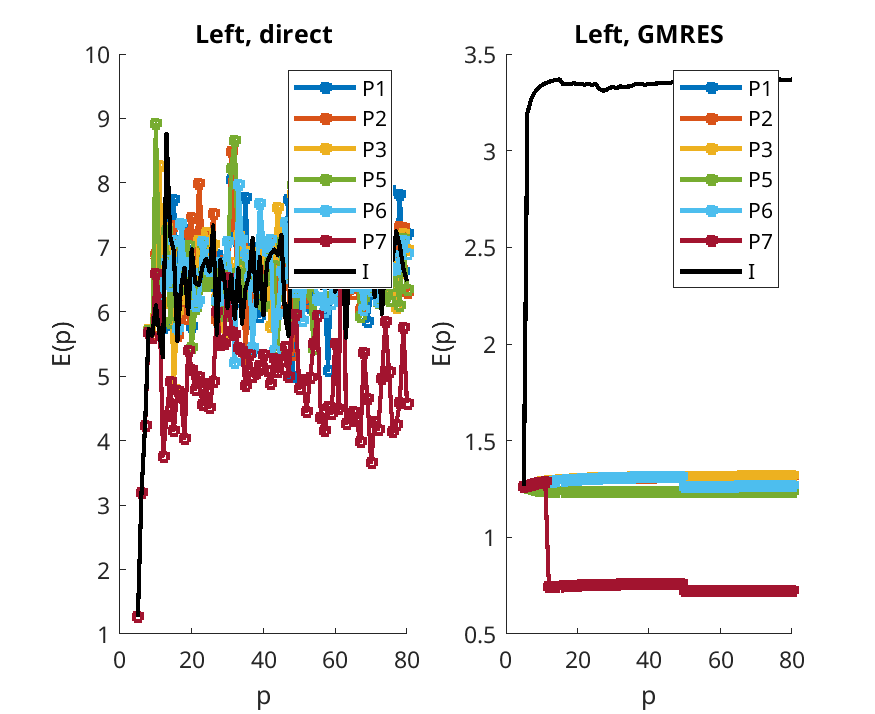}

	\includegraphics[width=0.3\linewidth]{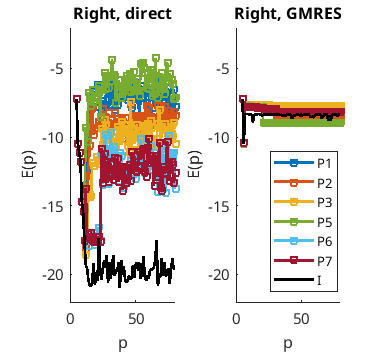}
	\includegraphics[width=0.3\linewidth]{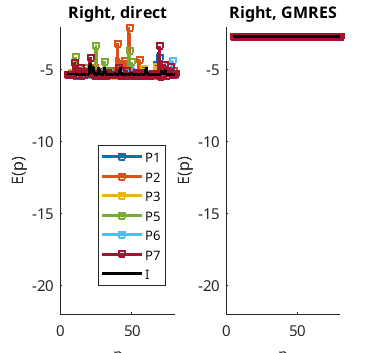}
	\includegraphics[width=0.3\linewidth]{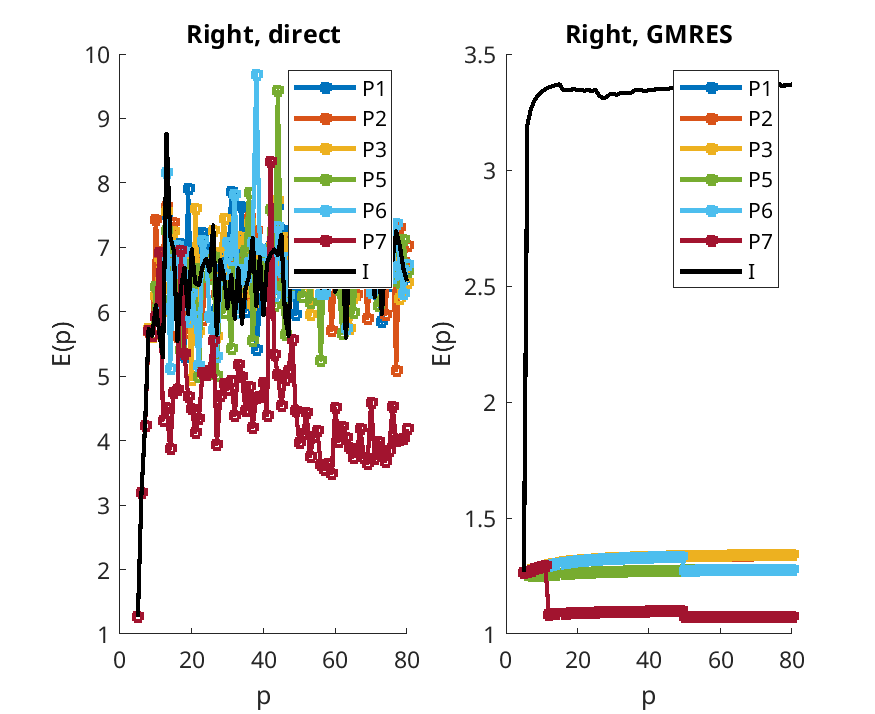}

	\caption{Where should we choose the center of the plane wave basis in skinny elements? Here, $h=1, \kappa = 0.2\pi$, and this is point-source scattering, source at (2,-4). Top row: configuration of plane waves. Second row: $E(P) v/s P$ using left preconditioners.  Third row: Left-right preconditioners. Final row: Right preconditioners. The GMRES tolerance used is $tol=1e-6$ and $\delta=1e-10$. }
	\label{fig:fan}	
\end{figure}
We end this section with a comparison of preconditioner performance on elements of different shapes in Figure \ref{fig:equilateralh1kappa01fan}. We expect the circulant preconditioners to perform well on regular polygons, for which the mass and stiffness matrices are close to Toeplitz. Interestingly, the preconditioners also improve accuracy on a triangular element with poor aspect ratio.

\begin{figure}
	\centering
	\includegraphics[width=0.22\linewidth]{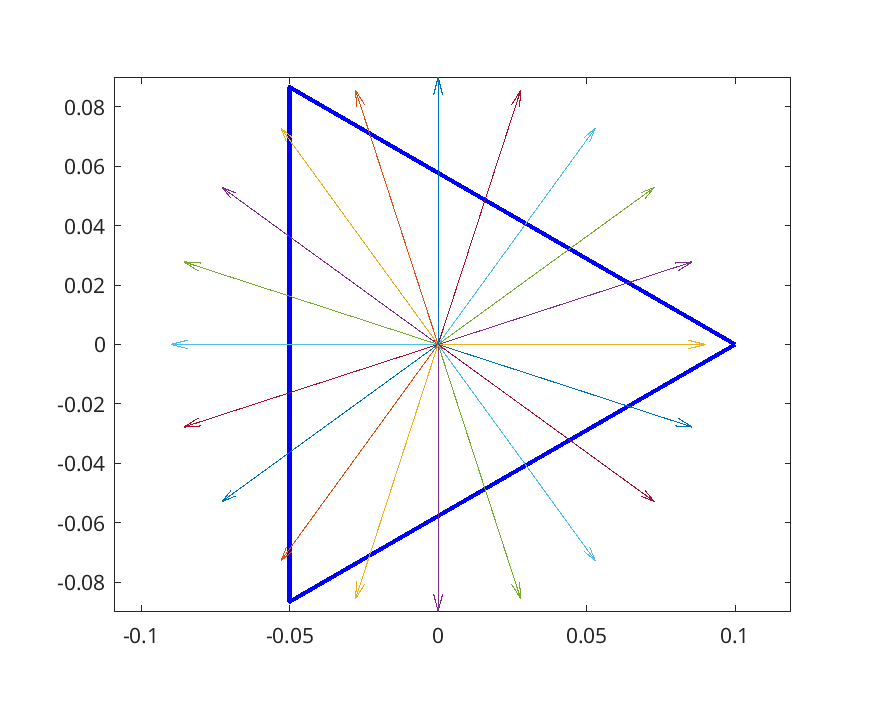}
	\includegraphics[width=0.22\linewidth]{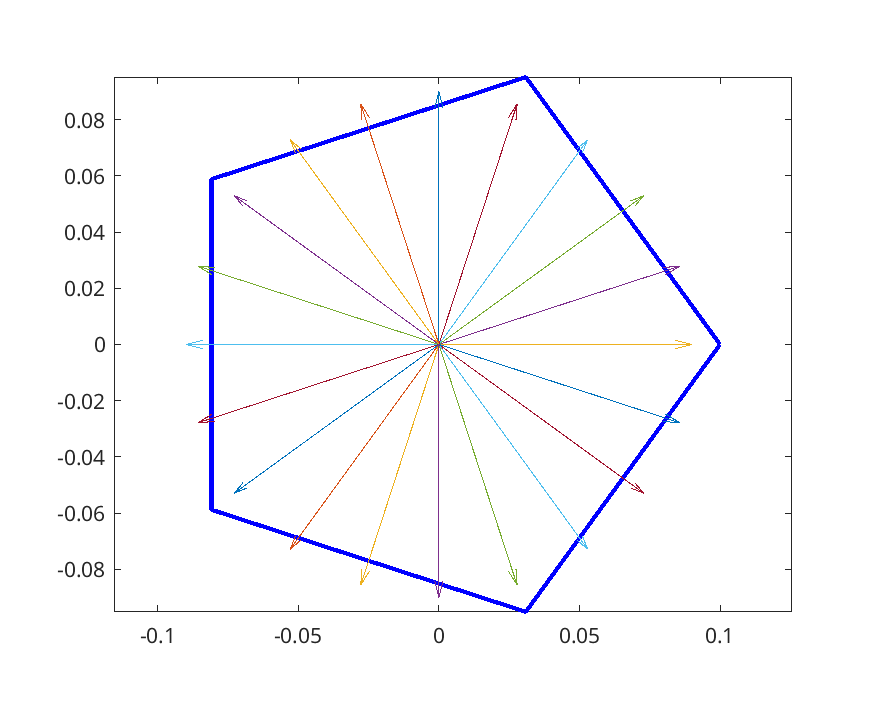}
	\includegraphics[width=0.22\linewidth]{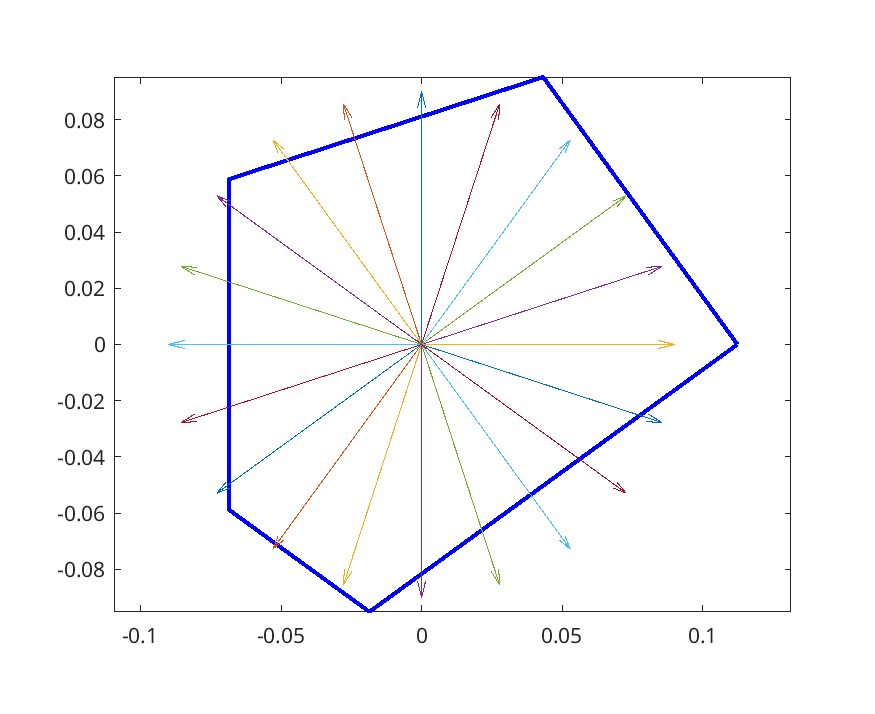}
	\includegraphics[width=0.22\linewidth]{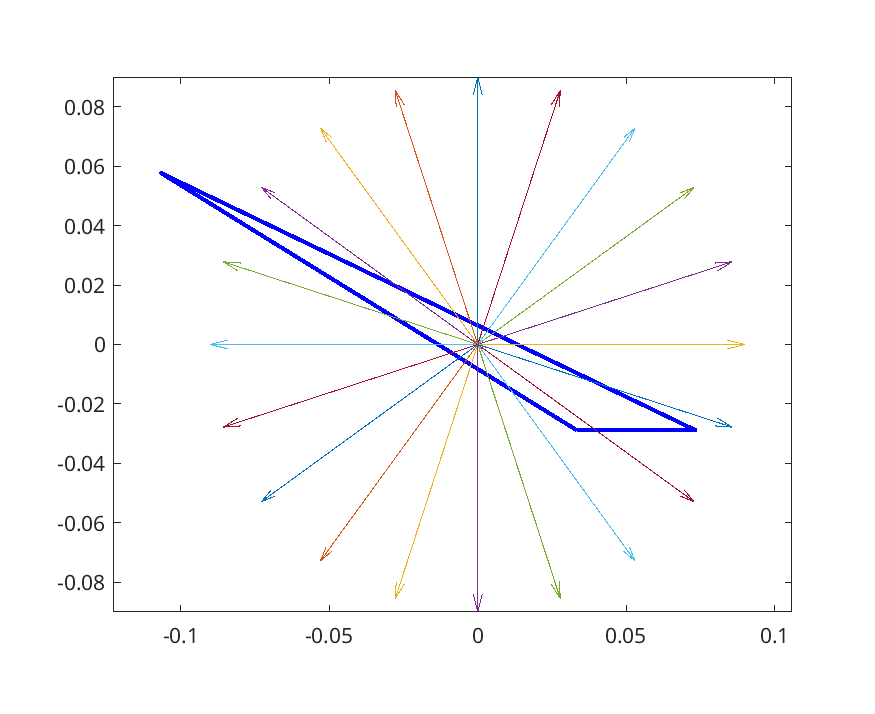}	
	
	\includegraphics[width=0.22\linewidth]{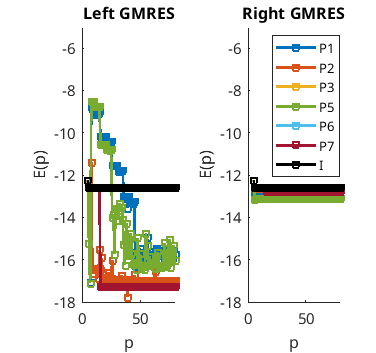}
	\includegraphics[width=0.22\linewidth]{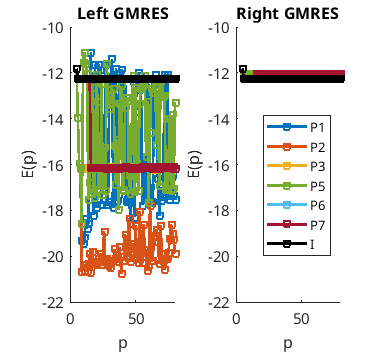}
	\includegraphics[width=0.22\linewidth]{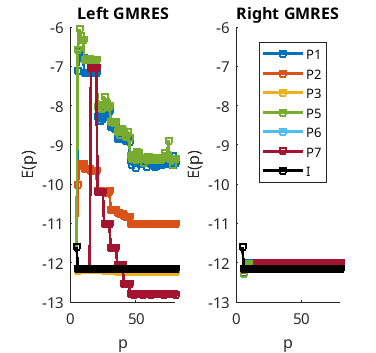}
	\includegraphics[width=0.22\linewidth]{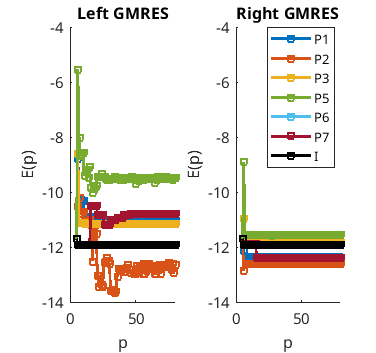}		
	\caption{Point source scattering from (left to right) equilateral triangle, regular pentagon, cyclic pentagon and obtuse triangle. In each case, $h=0.1, \kappa=0.2\pi$, GMRES tolerance =1e-6 and $\delta=1e-10$. (Note the scales on the y axes vary in these figures) }
	\label{fig:equilateralh1kappa01fan}
\end{figure}

\subsection{Discussion}

Based on the preceding analysis and examples, we summarize:

\begin{itemize}
	\item A preconditioner which reduces the Trefftz system condition number may {\it not} yield  better results in terms of the computed pde solution. This is because the preconditioners may themselves be very poorly conditioned: if $\P\M \approx \mathtt{I}$ then the conditioning of matrix-vector multiplication $\P\tt{f}$ will be poor. For instance, in Figure \ref{fig:PreconditionerComparison} preconditioner $\P_4$ produces very well-conditioned systems matrices. However, in Figure \ref{fig:l2errorspentagon1-5} we see that its performance both as part of a direct solver and for GMRES is not as good as others.
	\item For element which are close to cyclic polygons, using evenly-spaced plane-wave directions with center at the element centroid yields accurate results.
	\item For physical elements $K$ with an extreme aspect ratio, using evenly-spaced plane-wave directions with center at the element centroid provided the best results. 	
	
	\item Amongst the preconditioners considered, in most experiments  $P_{2}:=\U (\Lambda({\mathtt{SyS}^{disk}}))^{-1/2}\U^* $, the 'best circulant preconditioner' $\P_5:=(\kappa^2\C^K_{best})^{-1/2}$ and  the singular preconditioner $\P_7$ improved GMRES accuracy. 
	\item The preconditioner $\P_5$ is Hermitian, at least for equally-spaced planewaves. The singular preconditioner $\P_7$ is also Hermitian. Neither, therefore, introduce any artificial damping in the system.
	\item Left preconditioning as part of a GMRES solve yields more accurate results compared to right preconditioning.
	\item In most examples considered, \MATLAB's direct solver (backslash) on the unpreconditioned system is optimal. The benefits of preconditioning are seen with the iterative solver.

\end{itemize}
These observations lead us to conjecture that using the circulant matrices $\P_2, \P_5,$ or the singular matrix $\P_7$ \--- all easy to build and efficient to apply\--- as part of a left block-preconditioning strategy for iterative systems could improve accuracy for plane-wave Trefftz methods on meshes. We explore this in forthcoming work. 

\section{Conclusions}
In this paper, we investigate the behaviour of the plane-wave basis on a {\it single} element. We provide computational evidence that as the number of plane waves increases, the system matrix becomes nearly Toeplitz. The system matrix is exactly circulant for a disk, on which the spectrum can be characterized easily. Prompted by these observations, we propose several circulant preconditioners which are easy to build and apply. These preconditioners can reduce the condition numbers of the system matrix, and some improve the $L^2$-errors in the computed solution as part of a direct strategy or GMRES.

Amongst the preconditioners considered, in most experiments  \\
$P_{2}:=\U (\Lambda({\mathtt{SyS}^{disk}}))^{-1/2}\U^* $, the `best circulant preconditioner' $\P_5:=(\kappa^2\C^K_{best})^{-1/2}$ and  the singular preconditioner $\P_7$ improved GMRES accuracy when used as left preconditioners. In forthcoming work, we explore the performance of these preconditioners as part of a block preconditioning strategy on a mesh. 

As a final comment, we observe that the goal of reducing the condition number would lead to {\it different} recommendations than the goal of reducing the $L^2$-error.

\section*{Acknowledgements} We gratefully acknowledge  many helpful discussions with H\'el\`ene Barucq, Ilaria Perugia and Peter Monk which have influenced this work.

\bibliographystyle{plain}

\bibliography{CoyleNigam}

\clearpage

\end{document}